\documentclass[11pt]{article}

\usepackage{graphicx}

\usepackage{amsmath,amssymb}
\usepackage{bm}
\usepackage{mathrsfs}
\usepackage{amsfonts}
\usepackage{enumerate}
\usepackage{amsthm}
\usepackage[all,2cell]{xy}
\usepackage{mathrsfs}
\usepackage{mathtools}
\mathtoolsset{showonlyrefs}

\usepackage{tikz-cd}
\usepackage{tikz}
\usetikzlibrary{shapes.multipart}
\usetikzlibrary{patterns,patterns.meta}
\usepackage{xcolor}
\DeclareUnicodeCharacter{0301}{*************************************}

\usepackage[colorlinks=true]{hyperref}
\hypersetup{urlcolor=blue, citecolor=red, linkcolor=blue}

\usetikzlibrary{cd}

\numberwithin{equation}{section}

\usepackage[a4paper, left=25mm, right=25mm, top=30mm, bottom=30mm]{geometry}

\usepackage{marginnote}

\usepackage{imakeidx}
\makeindex[intoc]

\usepackage[nottoc]{tocbibind}
\usepackage[colorinlistoftodos]{todonotes} 

\renewcommand{\Bbb}{\mathbb}

\theoremstyle{plain}
\newtheorem{theorem}{Theorem}[section]
\newtheorem{corollary}[theorem]{Corollary}
\newtheorem{proposition}[theorem]{Proposition}
\newtheorem{lemma}[theorem]{Lemma}

\theoremstyle{definition}
\newtheorem{definition}[theorem]{Definition}
\newtheorem{remark}[theorem]{Remark}
\newtheorem{example}[theorem]{Example}

\newcommand\restr[2]{{
  \left.\kern-\nulldelimiterspace 
  #1 
  \right|_{#2} 
}}

\newcommand{\R}{\mathbb{R}}
\renewcommand{\d}{\mathrm{d}}

\newcommand{\Cinfty}{\mathscr{C}^\infty}
\newcommand{\T}{\mathrm{T}}
\newcommand{\cT}{\mathrm{T}^\ast}
\newcommand{\Id}{\mathrm{Id}}
\newcommand{\Lie}{\mathscr{L}}

\newcommand{\X}{\mathfrak{X}}
\newcommand{\Xh}{\X_{\mathrm{ham}}}
\newcommand{\bomega}{\bm{\omega}}

\newcommand{\g}{\mathfrak{g}}
\DeclareMathOperator{\Ad}{\mathrm{Ad}}

\newcommand{\bfX}{\mathbf{X}}
\newcommand{\bfY}{\mathbf{Y}}

\newcommand{\bfJ}{\mathbf{J}}

\newcommand{\parder}[2]{\frac{\partial #1}{\partial #2}}

\newcommand{\tparder}[2]{\partial #1/\partial #2}

\DeclareMathOperator{\inn}{\iota}

\DeclareMathOperator{\Ham}{Ham}
\DeclareMathOperator{\rk}{rank}

\DeclareMathAlphabet{\mathpzc}{OT1}{pzc}{m}{it}
\def\d{\mathrm{d}}

\def\om{\omega}

\def\wtl#1{\widetilde{#1}}
\def\B#1{\mathbf{#1}}

\def\x{\times}
\def\>{\rangle}
\def\<{\langle}
\def\Lg{\mathfrak{g}}

\let\emph\textbf

\title{{\sffamily On $k$-polycosymplectic Marsden--Weinstein reductions}}

\author{{\sffamily 
$^a$Javier de Lucas%
\thanks{e-mail:
   javier.de.lucas@fuw.edu.pl \ ORCID: 0000-0001-8643-144X}\ ,\
$^b$Xavier Rivas%
\thanks{e-mail:
   xavier.rivas@unir.net \ ORCID: 0000-0002-4175-5157}\ ,\
$^c$Silvia Vilariño%
\thanks{e-mail:
   silviavf@unizar.es \ ORCID: 0000-0003-0404-1427}\ \  and,\
$^a$Bartosz M. Zawora%
\thanks{e-mail:
   b.zawora@uw.edu.pl \ ORCID: 0000-0003-4160-1411}\
}
\\[1ex]
\normalsize\itshape\sffamily
$^a$UW Institute for Advanced Studies,\\
\normalsize\itshape\sffamily
Department of Mathematical Methods in Physics, University of Warsaw, Warszawa, Poland.
\\[1ex]
\normalsize\itshape\sffamily
$^b$Escuela Superior de Ingenier\'{\i}a y Tecnolog\'{\i}a,\\
\normalsize\itshape\sffamily
Universidad Internacional de La Rioja, Logro\~no, Spain.
\\[1ex]
\normalsize\itshape\sffamily
$^c$Centro Universitario de la Defensa de Zaragoza,
\\
\normalsize\itshape\sffamily
Instituto Universitario de Matemáticas y Aplicaciones, Zaragoza, Spain.
\\[1ex]
}

\date{{\sffamily \today}}

\begin{document}

\maketitle

\begin{abstract}
    We review and slightly improve the known  $k$-polysymplectic Marsden--Weinstein reduction theory by removing some technical conditions on $k$-polysymplectic momentum maps  by developing a theory of affine Lie group actions for $k$-polysymplectic momentum maps, removing the necessity of their co-adjoint equivariance. Then, we focus on the analysis of a particular case of $k$-polysymplectic manifolds, the so-called fibred ones, and we study their $k$-polysymplectic  Marsden--Weinstein reductions. Previous results allow us to devise a  $k$-polycosymplectic Marsden--Weinstein reduction theory, which represents one of our main results. Our findings are applied to study coupled vibrating strings and, more generally, $k$-polycosymplectic Hamiltonian systems with field symmetries. We show that $k$-polycosymplectic geometry can be understood as a particular type of $k$-polysymplectic geometry. Finally, a $k$-cosymplectic to $\ell$-cosymplectic geometric reduction theory is presented, which reduces, geometrically, the space-time variables in a $k$-cosymplectic framework. An application of this latter result to a vibrating membrane with symmetries is given.

\end{abstract}

\bigskip

\bigskip

{\setcounter{tocdepth}{2}
\def\baselinestretch{1}
\small
\def\addvspace#1{\vskip 1pt}
\parskip 0pt plus 0.1mm
\tableofcontents
}



\section{Introduction}

The problem of the reduction of systems with symmetry has attracted for decades the interest of mathematicians and theoretical physicists, who have sought to reduce the number of equations describing the behaviour of the systems by finding first-integrals or conservation laws \cite{MW01,OR04}. 

The general procedure of the symplectic reduction can be traced back to E. Cartan, and it goes as follows (see \cite[pg. 298]{AM78} or \cite{MW01,OR04} and references therein):
\begin{quote}
    ``Suppose that $P$ is a manifold and $\omega$ is a closed two-form on $P$; let $\ker \omega = \{v \in \T P \mid \inn_v\omega = 0\}$ be the {\it characteristic distribution} of $\omega$ and call $\omega$ {\it regular} if $\ker \omega$ is a subbundle of $\T P$. In the regular case, $\ker \omega$ is an involutive distribution. By the Fr\"obenius theorem, $\ker \omega$ is integrable and it defines a foliation $\mathcal{F}$ on $P$. This gives rise to a quotient space $P/\mathcal{F}$ by identification of all points on the same leaf of $\mathcal{F}$. Assume now that $P/\mathcal{F}$ is a manifold, the canonical projection $x\in P \mapsto [x]\in P/\mathcal{F}$ being a submersion. Then, the tangent space at a point $[x]$ is isomorphic to $\T_xP/ \ker \omega_x$ and $\omega$ projects onto a well-defined closed, nondegenerate two-form on $P/\mathcal{F}$; that is, $P/\mathcal{F}$ is a symplectic manifold: a so-called {\it reduced space}.''
\end{quote}

The use of geometrical methods has proved to be a powerful tool in the study of this topic and a breakthrough was performed by Marsden and Weinstein in their  work on the reduction of autonomous Hamiltonian systems on symplectic manifolds under the action of a Lie group of symmetries, with regular values of their momentum maps \cite{MW74}. Just before that work, Meyer obtained some similar findings \cite{Me73}, but not as detailed and sound as Weinstein and Marsden's ones \cite{MW01}. In a more general context, all their results were indeed the culmination of many other previous achievements by Smale, Sternberg, Kostant, Robbin, and many others, who gave partial but relevant approaches to the reduction procedure (see \cite{MW01}). In their famous work \cite{MW74}, Marsden and Weinstein apply a very powerful version of the previous reduction scheme to the case of submanifolds defined by the level sets of an equivariant momentum mapping ${\bf J}^\Phi:P\rightarrow \mathfrak{g}^*$ for a certain Lie group action on the dual $\mathfrak{g}^*$ of a Lie algebra $\mathfrak{g}$ and a Hamiltonian action on a symplectic manifold $P$ leaving invariant a certain function on $P$: the Hamiltonian. The so-obtained reduced space is a symplectic manifold and inherits a Hamiltonian dynamics from the initial Hamiltonian.

After Marsden and Weinstein's work, the Marsden--Weinstein reduction technique was subsequently applied and generalised to many different situations. For instance, the reduction of Hamiltonian systems with singular values of the momentum map has been studied in several papers such as \cite{SL91} for the autonomous case. In that work, stratified reduced manifolds admitting symplectic structures are obtained. The so-called orbifolds can also appear as Marsden--Weinstein reductions, which have motivated a separate research topic with physical and mathematical applications \cite{GB86,HSS15,LMS93}.
The reduction of time-dependent regular Hamiltonian systems (with regular values) is developed
in the framework of cosymplectic manifolds in \cite{Al89,LS93}, obtaining a reduced phase space that is a cosymplectic manifold. The study of autonomous systems coming from certain kinds of singular Lagrangians can be found in \cite{CANTRIJN1986353}, where the conditions for the reduced phase space to inherit an almost-tangent structure are given. Another approach to this question is adopted in \cite{MRSV15}, where the authors give conditions for the existence of a regular Lagrangian function in the reduced phase space, which allows them to construct the reduced cosymplectic or contact structure (and hence the reduced Hamiltonian function) from it.

It is worth noting that there have been many generalisations of the Marsden--Weinstein reduction to deal with many types of geometric structures. The Marsden--Weinstein reduction was extended by Marsden and Ratiu to Poisson manifolds in \cite{MR86}, the case of locally conformally symplectic manifolds was developed in \cite{HR01}, while the reduction of Dirac structures was solved and further analysed in several papers \cite{CF06,CFZ10b,Co90}. On the other hand, the Marsden--Weinstein reduction of Jacobi manifolds is analysed in \cite{ILM97}. 

After almost fifty years of their foundational work \cite{MW74}, the development of Marsden--Weinstein reductions for different types of geometries and schemes of reduction is a very active research field as results can be extended to many geometric structures, it admits many modifications and generalisations, e.g. for the so-called singular cases \cite{LS93,JRS01}, and it has lots of applications as illustrated by the many works constantly using them  \cite{BG15,BT21,GB86,HSS15,LMS93,SS22,Wa23}. It is worth noting that the development of a multisymplectic Marsden--Weinstein reduction has been an open problem for decades now \cite{EMR18}, and there is a lot of interest in developing even partial results \cite{Bl21}.


Geometric covariant descriptions of first-order classical field theories can be performed by appropriate generalisations of some of the above-mentioned structures. The simplest one is {\it $k$-symplectic geometry}, introduced by A. Awane \cite{Aw92,Aw00} and used later by M. de León et al. \cite{dLe88,dLe93,DeLeo1997}, and L. K. Norris \cite{Nor2000,Nor1993} to describe first-order classical field theories. They coincide with the {\it polysymplectic structures} described by G. C. Günther \cite{Gu87} (although these last ones are different from those introduced by G. Sardanashvily et al. \cite{Gia1997,Sar1995} and I. V. Kanatchikov \cite{Kan1998}, that are also called {\it polysymplectic}). This structure is applied to first-order regular autonomous field theories \cite{Bua2015,Ech1996,Rom2007}. A natural extension of the above structures are {\it $k$-cosymplectic manifolds}, which allow us to generalise the cosymplectic description of non-autonomous mechanical systems \cite{AM78} to regular field theories whose Lagrangian and/or Hamiltonian functions, in the local description, depend on the space-time coordinates \cite{LM98,DeLeo2001}. We refer to \cite{LSV16,Gra2004,MLM10} for more details on the $k$-symplectic and $k$-cosymplectic formalisms. In \cite{RRSV_2011}, the relationships that exist between $k$-symplectic, $k$-cosymplectic and multisymplectic structures are established. 

The present work is part of an ambitious research project in which we want to develop the Marsden--Weinstein multisymplectic reduction. In this paper, we are focused on the extension of the Marsden--Weinstein reduction theory to the case of $k$-polycosymplectic manifolds \cite{BG15,LSV16}. The search for a $k$-polycosymplectic reduction can be traced back to \cite{Bl89}, where a particular case was studied. Next, some ideas of how a $k$-polycosymplectic reduction should be were pointed out in \cite{MRSV10}, but no proofs were given. Some of the ideas in \cite{MRSV10} gave rise to \cite{MC11}, where no $k$-polycosymplectic reduction was studied, and \cite{MRSV15}, where the $k$-polysymplectic reduction was fully described instead. An initial approach to $k$-polysymplectic reduction was made in the seminal work of G\"unther \cite{Gu87}, where the author attempts to reply to the Marsden--Weinstein reduction theory for symplectic manifolds to the $k$-polysymplectic case. Nevertheless, the proof of one of the fundamental results fails to be true as a consequence of the technical properties of the orthogonal $k$-polysymplectic complement. In \cite{MRSV15} this problem was solved. It is worth noting that \cite{MRSV15} provided new relevant ideas and solutions that were not depicted in \cite{MRSV10} so as to develop a Marsden--Weinstein $k$-polysymplectic reduction. Nowadays, more than a decade after \cite{MRSV10}, the Marsden--Weinstein $k$-polycosymplectic reduction is still to be developed. Achieving this reduction and studying its properties and applications is the main aim of this work, which has been drawing some attention until the present day.

This work will show how some of the ideas of \cite{MRSV10}, along with other ones in \cite{MRSV15} and new others to be disclosed hereafter, allow for devising a $k$-polycosymplectic Marsden--Weinstein reduction. In particular, a $k$-polycosymplectic manifold will be associated with a $k$-polysymplectic manifold of a larger dimension and a specific type, a so-called $k$-{\it polysymplectic fibred manifold}, to be defined and studied in this work for the first time. In particular, $k$-polysymplectic fibred manifolds admit, among other properties, the hereafter referred to as $k$-{\it polysymplectic Reeb vector fields}. A slight generalisation of the $k$-polysymplectic Marsden--Weinstein reduction developed in \cite{MRSV15} will be applied to $k$-polysymplectic fibred manifolds obtained from $k$-polycosymplectic manifolds to obtain reduced $k$-polysymplectic fibred manifolds. The latter will be related to $k$-polycosymplectic manifolds that will be identified with the Marsden--Weinstein reductions of the initial $k$-polycosymplectic manifolds. In particular, the general scheme for our $k$-polycosymplectic reduction to be explained in detail in this paper is displayed in Figure \ref{Fig::Scheme}. It is worth noting that $k$-polysymplectic manifolds and $k$-polycosymplectic manifolds, in the way defined in this work, do not need to admit Darboux coordinates \cite{PhDThesisXRG}, which makes many proofs of our paper more technical and complicated than other generalisations of the Marsden--Weinstein symplectic reduction. Most main results describing our $k$-polycosymplectic Marsden--Weinstein reduction are described in Section \ref{Sec::PolycoReduction}, while Theorem \ref{Th::kpolycoreduction} is one of the most important results concerning this topic. Note that Theorem \ref{Th::PolyCoReductionDynamics} describes the reduction of Hamiltonian $k$-polycosymplectic vector fields, which is related to the HDW (Hamilton--De Donder--Weyl) equations for the $k$-polycosymplectic formalism. 
\begin{figure}\label{Fig::Scheme}
    \centering
\begin{center}
    \begin{tikzcd}
        (\R^k\times M,\widetilde{\bm\omega}={\rm pr}_M^*{\bm \omega}+\d {\bm u}\barwedge {\rm pr}_M^*{\bm \tau}) \arrow[rr,"\mathrm{pr}_M"] 
        &
        &
        (M,\bm\tau,\bm\omega)  
        \\
        \\
        (\widetilde{\bf J}^{\Phi-1}(\bm\mu) = \R^k\times {\bf J}^{\Phi-1}(\bm\mu),\widetilde\jmath^*_{\bm\mu}\widetilde{\bm\omega}) 
        \arrow[dd, "\wtl\pi_{\bm\mu}=\Id_{\R^k}\otimes\pi_{\bm\mu}"]
        \arrow[uu,"\wtl\jmath_{\bm{\mu}}"]
        &
        &
        ({\bf J}^{\Phi-1}(\bm\mu),\jmath^*_{\bm\mu}\bm\tau,\jmath^*_{\bm\mu}\bm\omega)
        \arrow[ll,"\iota_{\bm u}^{{\bf J}^{\Phi-1}(\bm\mu)}"]
        \arrow[dd,"\pi_{\bm\mu}"] \arrow[uu,"\jmath_{\bm\mu}"]
        \\ 
        \\
        (\R^k\times M^{\bm \Delta}_{\bm  \mu}, \widetilde{\bm\omega}_{\bm\mu}={\rm pr}_{M^{\bm \Delta}_{\bm \mu}}^*{\bm \omega}_{\bm \mu}+\d {\bm u}\barwedge {\rm pr}_{M^{\bm \Delta}_{\bm \mu}}^*{\bm \tau}_{\bm \mu}) \arrow[rr, swap, bend right=3, "\mathrm{pr}_{M^{\bm \Delta}_{\bm  \mu}}"] 
        &
        &
        (M^{\bm \Delta}_{\bm \mu}={\bf J}^{\Phi-1}(\bm\mu) / G_{\bm\mu}^{\bm\Delta},\bm\tau_{\bm\mu},\bm\omega_{\bm\mu})
        \arrow[ll, "\iota_{\bm u}^{M^{\bm \Delta}_{\bm \mu}}", swap, bend right=3]
    \end{tikzcd}
\end{center}
    \caption{Scheme of the different structures involved in the $k$-polycosymplectic reduction through $k$-polysymplectic fibred manifolds. It is worth noting that $k$-polysymplectic manifolds in the above diagram admit a series of vector fields satisfying properties extending the ones for Reeb vector fields in $k$-polycosymplectic geometry.}
    \label{fig:ReductionScheme}
\end{figure}

As a byproduct of our work, the $k$-polysymplectic reduction theory in \cite{MRSV15} is improved by removing some unnecessary technical conditions imposed there, namely the {\rm Ad}$^{*k}$-invariance of the momentum map and other minor aspects concerning the existence of a quotient manifold structure in the quotients spaces appearing in the $k$-polysymplectic Marsden--Weinstein reduction. This is mainly performed by generalising the standard theory of affine Lie group actions for the symplectic case (see \cite{OR04}) to the $k$-polycosymplectic and $k$-polysymplectic realms. Our techniques also generalise the known relations between cosymplectic and symplectic manifolds and their reductions \cite{LS93}.

Also as a byproduct of our results, we obtain in Theorem \ref{Prop::PolyCosymSymEqui} that $k$-polycosymplectic geometry is a particular case of $k$-polysymplectic geometry when the $k$-polysymplectic manifolds admit a certain fibration. Such $k$-polysymplectic manifolds are here called $k$-polysymplectic fibred manifolds. Note that this result is very relevant, as it shows that $k$-polycosymplectic geometry is a particular case of $k$-polysymplectic geometry and it allows us to use the techniques of $k$-polysymplectic geometry to study $k$-polycosymplectic manifolds. In fact, this result is employed here to devise a $k$-polycosymplectic reduction from a  $k$-polysymplectic fibred one. 

It is convenient to stress that Theorem \ref{Prop::PolyCosymSymEqui} connects a $k$-polycosymplectic structure on $M$ with a $k$-polysymplectic structure on $\mathbb{R}^k\times M$. The larger dimension of the $k$-polysymplectic manifold may lead to certain complications. For example,  Section \ref{Sec::Polyco&Polysym} shows that Hamiltonian $k$-vector fields in the $k$-polycosymplectic realm are related to $k$-polysymplectic Hamiltonian $k$-vector fields with different equilibrium points, which may potentially introduce difficulties to study certain problems. For instance, \cite{LMZ22} shows that the extension from cosymplectic Hamiltonian vector fields with equilibrium points may lead to Hamiltonian vector fields in the associated symplectic manifolds without them, which gives rise to problems, for instance, in the study of relative equilibrium points \cite{LMZ22}. Briefly speaking, these extensions techniques are interesting for studying geometric structures in terms of others, but their interest is quite restricted for analysing associated dynamical systems.

Finally, a Marsden--Weinstein $k$-cosymplectic to $\ell$-cosymplectic reduction is presented. This reduction can be employed to reduce geometrically $k$-cosymplectic manifolds using $k$-cosymplectic Lie group symmetries that do not leave invariant the coordinates of the base manifold. As far as we know, this approach is pioneering in the literature and radically different from previous geometric results. The corresponding reduction of the HDW equations from a $k$-cosymplectic to an $\ell$-cosymplectic reduction is accomplished provided some technical conditions are satisfied. An application of our findings to a vibrating membrane is accomplished. The procedure is developed in such a way that a momentum map relative to a $k$-cosymplectic Lie group action for an original $k$-cosymplectic manifold allows for the reduction of the resulting $\ell$-cosymplectic manifold by our previous $\ell$-polycosymplectic Marsden--Weinstein reduction. An analogue result for HDW equations is devised.

Apart from the previous theoretical results, some applications of our findings to two coupled vibrating strings, a particular sort of vibrating membrane, and field theories with symmetries are accomplished.

The structure of the paper goes as follows. Section \ref{Sec::CosGeo} introduces some basic notions from cosymplectic geometry. In particular, it presents three distinguished vector fields related to a smooth function on a cosymplectic manifold and set conventions to be used hereafter. Section \ref{Sec::CosMWRed} recalls the main results from the theory of non-$\Ad^*$-equivariant momentum maps on cosymplectic manifolds and presents a generalisation of the Marsden--Weinstein reduction theorem to the cosymplectic setting. At the end of this section, the classical Marsden--Weinstein reduction is used to obtain a cosymplectic reduced manifold, this very slightly generalises a result in \cite{LS93}, while provides a very simple example about some of the ideas to be generalised  in our main theorems. Section \ref{Sec::FundFieldTh} defines and presents all crucial properties of $k$-polysymplectic and $k$-polycosymplectic manifolds. This section finishes with an example of a vibrating membrane subject to an external force, where the $k$-polycosymplectic setting applies. Section \ref{Sec::PolysymplecticReduction} generalises the notion of the $k$-polysymplectic momentum map by omitting the $\Ad^{*k}$-equivariant condition, and as a byproduct, it also generalises the known $k$-polysymplectic reduction theorem for non-$\Ad^{*k}$-equivariant $k$-polysymplectic momentum maps. Section \ref{Sec::PolycoReduction} is divided into three subsections. $k$-Polycosymplectic momentum maps with and without $\Ad^{*k}$-equivariance are defined in Sections \ref{Sec::with} and \ref{Sec::without}, respectively. Section \ref{Sec::Polyco&Polysym} performs a $k$-polycosymplectic reduction via an extended $k$-polysymplectic manifold and a $k$-polysymplectic reduction theorem. It also discusses the conditions essential to execute the reduction and comments on how these conditions recover via the approach presented in \cite{MRSV15}. Section \ref{Se::Examples} demonstrates a $k$-polycosymplectic reduction in action on the examples of a product of cosymplectic manifolds and on a system consisting of two coupled vibrating strings. Finally, Section \ref{Sec::SPRed} studies Marsden--Weinstein reductions of $k$-cosymplectic to $\ell$-cosymplectic manifolds and other associated $\ell$-polycosymplectic reductions. Section \ref{Sec::SPRed} also analyses the associated reduction of HDW equations  and it provides an example illustrating our techniques.

\section{Cosymplectic geometry}
\label{Sec::CosGeo}
 
Let us survey the basic notions and results on cosymplectic geometry and related concepts to be used in this work (see \cite{LMZ22,Lic1951, Lic1975} for details). This will serve as a simple illustrative example of our further methods. $M$ is hereafter assumed to be a manifold. Throughout the paper, all structures are assumed to be smooth and globally defined, unless otherwise stated.  Manifolds are second-countable, Hausdorff, and connected. The sum over crossed repeated indices is understood unless indexes are hatted. Indexes have a standard range of values, which are implicitly assumed. In case a sum over an index is carried out over  a non-standard range of values, a explanatory summation symbol will be explicitly written. This will mainly occur  in Section 
\ref{Sec::SPRed}.

\begin{definition}
Let $M$ be a manifold equipped with a closed differential two-form $\omega\in\Omega^2(M)$ and a non-vanishing differential one-form $\tau\in\Omega^1(M)$ such that $\ker\tau\oplus\ker\omega=\T M$. Then,  $(M,\tau,\omega)$ is called a \textit{cosymplectic manifold}.
\end{definition}

Given a point $x\in M$, the {\it cosymplectic orthogonal} of a subspace $V_x\subset \T_xM$ with respect to the cosymplectic manifold $(M,\tau,\omega)$ is defined by
\[
V^{\perp_\omega}_x = \{z_x\in \T_xM\,\mid\,\omega_x(z_x,v_x)=0\,,\ \forall v_x\in V_x\}\,.
\]
Note that $(M,\tau,\omega)$ is a cosymplectic manifold if and only if $M$ is a $(2n+1)$-dimensional manifold such that $\omega^n\wedge\tau$ is a volume form on $M$. Hence, cosymplectic manifolds are always orientable and odd-dimensional. Physically, cosymplectic manifolds appear in the description of $t$-dependent mechanical systems \cite{Al89}.




The Darboux theorem for cosymplectic manifolds \cite{Al89,LR89,Go69,LM87} states that, given a cosymplectic manifold $(M,\tau,\omega)$, each $x\in M$ admits a local coordinate system  $\{t,q^1,\dotsc,q^n,p_1,\dotsc,p_n\}$  on an open neighbourhood of $x$ so that 
\begin{equation}
    \tau = \d t\,,\qquad \omega = \d q^i\wedge \d p_i\,.
\end{equation}
Such local coordinates are called \textit{cosymplectic Darboux coordinates}, although we will simply call them \textit{Darboux coordinates} if it is clear from the context what we mean. Note that Darboux coordinates are not unique.

Every cosymplectic manifold $(M,\tau,\omega)$ gives rise to a unique vector field $R\in\X(M)$, called the {\it Reeb vector field} of $(M,\tau,\omega)$, characterised by
\begin{equation}\label{eq:Reeb-conditions}
    \inn_R\tau = 1\,,\qquad\inn_R\omega = 0\,.
\end{equation}
In cosymplectic Darboux coordinates, the Reeb vector field reads $R = \tparder{}{t}$.

\begin{definition}
    A \textit{cosymplectomorphism} is a map $\varphi\colon M_1\rightarrow M_2$ between cosymplectic manifolds $(M_1,\tau_1,\omega_1)$ and $(M_2,\tau_2,\omega_2)$ so that $\varphi^*\tau_2=\tau_1$ and $\varphi^*\omega_2=\omega_1$. 
\end{definition}

    A cosymplectic manifold $(M,\tau,\omega)$ leads to a vector bundle isomorphism   $\flat\colon\T M\to \cT M$ given by
    \begin{equation}\label{eq::form}
        v_x\in \T_xM\longmapsto \inn_{v_x}\omega_x + (\inn_{v_x}\tau_x)\tau_x\in\cT_x M\,,\qquad \forall x\in M\,.
    \end{equation}Moreover, given closed differential forms $\tau\in\Omega^1(M)$ and $\omega\in\Omega^2(M)$,  if the map $\flat:\T M\to\cT M$ of the form \eqref{eq::form} is a vector bundle isomorphism, then $\tau$ and $\omega$ give rise to a cosymplectic manifold $(M,\tau,\omega)$.


From now on, $(M,\tau,\omega)$ will always stand for a cosymplectic manifold. In cosymplectic geometry, every function $f\in \Cinfty(M)$ allows one to define three distinguished types of vector fields given in the following definition.

\begin{definition}
    Every function $f\in\Cinfty(M)$ gives rise, via $(M,\tau,\omega)$, to three relevant vector fields on $M$:
    \begin{itemize}
        \begin{subequations}
            \item A \textit{gradient vector field}, namely
            \begin{equation}\label{Eq::GradVecField}
                \nabla f = \flat^{-1}(\d f)\,,
            \end{equation}
            which amounts to $\inn_{\nabla f}\omega = \d f - (Rf)\tau$ and $\inn_{\nabla f}\tau = Rf$.
            
            \item A \textit{Hamiltonian vector field}, $X_f$, given by
            \begin{equation}\label{Eq::HamVecField}
                X_f = \flat^{-1}(\d f - (Rf)\tau)\,,
            \end{equation}
            which is equivalent to $\inn_{X_{f}}\omega = \d f-(Rf)\tau$ and $\inn_{X_f}\tau = 0$.
            
            \item An \textit{evolution vector field} \begin{equation}\label{Eq::EvVecField}
                E_f =\frac{}{} R + X_f\,,
            \end{equation}
            which is the unique vector field on $M$ satisfying the conditions  $\inn_{E_{f}}\omega = \d f-(Rf)\tau$ and $\inn_{E_f}\tau = 1$.
        \end{subequations}
    \end{itemize}
\end{definition}

In Darboux coordinates for $(M,\tau,\omega)$ around a point $x\in M$, the vector fields \eqref{Eq::GradVecField}, \eqref{Eq::HamVecField}, and \eqref{Eq::EvVecField} read
\begin{equation}
\nabla f = \frac{\partial f}{\partial t}\frac{\partial}{\partial t}+\frac{\partial f}{\partial p_i}\frac{\partial}{\partial q^i}-\frac{\partial f}{\partial q^i}\frac{\partial}{\partial p_i}\,,\qquad
X_f = \frac{\partial f}{\partial p_i}\frac{\partial}{\partial q^i}-\frac{\partial f}{\partial q^i}\frac{\partial}{\partial p_i}\,,
\end{equation}
and
\begin{equation}
E_f=\frac{\partial}{\partial t} + \frac{\partial f}{\partial p_i}\frac{\partial}{\partial q^i}-\frac{\partial f}{\partial q^i}\frac{\partial}{\partial p_i}\,.
\end{equation}
These local expressions are quite convenient to understand, quickly, several results. The integral curves of $E_f$ are given, in Darboux coordinates, by the solutions of
\begin{equation}
\label{Eq::CosymHamEq}
\frac{\d t}{\d s}=1\,,\qquad \frac{\d q^i}{\d s}=\frac{\partial f}{\partial p_i}(t,q,p)\,,\qquad \frac{\d p_i}{\d s}=-\frac{\partial f}{\partial q^i}(t,q,p)\,,\qquad i=1,\ldots,n\,,
\end{equation}
where $(t,q,p)$ stands for $(t,q^1,\ldots,q^n,p_1,\ldots,p_n)$. 

Let $N$ be a one-dimensional manifold and let $(P,\omega_P)$ be a symplectic manifold. Let us define a cosymplectic manifold on $M=N\times P$ and its related  Hamilton equations. Let $\pi_N: M\rightarrow N$ and let $\pi_P: M\rightarrow P$ be the projections onto the first and second factors of $N\times P$, respectively. Then, $\omega=\pi_P^*\omega_P$ becomes a closed differential two-form  on $M$. Meanwhile, a non-vanishing closed differential one-form $\tau_N$ on $N$ gives rise to a non-vanishing closed differential one-form $\tau=\pi_N^*\tau_N$ on $M$ and a cosymplectic manifold $(M,\tau,\omega)$. If not otherwise stated, Darboux coordinates on $(N\times P,\tau,\omega)$ will be assumed to be of the form $\{t,q^1,\ldots,q^n,p_1,\ldots,p_n\}$, where $t$ is the pull-back to $N\times P$ of a primitive function of $\tau_N$, while $q^1,\ldots,q^n,p_1,\ldots,p_n$ are the pull-backs to $M$ of some Darboux coordinates for $\omega_P$ on $P$. For simplicity, it is common in the literature to denote the pull-backs of functions on $N$ and $P$ to $M=N\times P$ in the same manner as the initial variables in $N$ and $P$. Although this is a slight abuse of notation, it does not lead to any misunderstanding and simplifies the presentation of results.

If $M=\R\times \cT Q$, $\tau_\mathbb{R}=\d t$, $\omega_{\cT Q}=\d q^i\wedge \d p_i$, then equations \eqref{Eq::CosymHamEq} lead to
\begin{equation}
\label{Eq::HamCosEq}
    \frac{\d q^i}{\d t}=\frac{\partial f}{\partial p_i}(t,q,p)\,,\qquad \frac{\d p_i}{\d t}=-\frac{\partial f}{\partial q^i}(t,q,p)\,,\qquad i=1,\ldots,n\,.
\end{equation}
Hence, \eqref{Eq::HamCosEq} retrieves the Hamilton equations for a $t$-dependent Hamiltonian system on $\cT Q$ (see \cite{AM78,LMZ22,LZ21}). 
\begin{proposition}\label{prop:gradXR}
    The gradient vector field  of $f\in \Cinfty(M)$ relative to $(M,\tau,\omega)$ reads $\nabla f = X_f + (Rf)R$. Moreover, if $Rf$ is a locally constant function, then $[R,X_f] = 0$.
\end{proposition}

    

Every cosymplectic manifold $(M,\tau,\omega)$ yields a Poisson bracket $\{\cdot,\cdot\}_{\tau,\omega}\colon\Cinfty(M)\times \Cinfty(M)\rightarrow \Cinfty(M)$ given by
\begin{equation}\label{Eq:PoissonStructure}
    \{f,g\}_{\tau,\omega} = \omega(\nabla f,\nabla g) = \omega(X_f,X_g)\,,\qquad \forall f,g\in \Cinfty(M)\,,
\end{equation}
where the last equality results from Proposition \ref{prop:gradXR} and $\inn_R\omega = 0$.
It can be proved that
\begin{equation}\label{eq:AntiMorphism}
    X_{\{f,g\}_{\tau,\omega}}= -[X_f,X_g]\,,\qquad \forall f,g\in \Cinfty(M)\,.
\end{equation}
Then, the space of Hamiltonian vector fields relative to a cosymplectic manifold $(M,\tau,\omega)$, let us say $\Ham(M,\tau,\omega)$, is a Lie algebra relative to the commutator of vector fields. Since $\Cinfty(M)$ is, in particular, a Lie algebra relative to $\{\cdot,\cdot\}_{\tau,\omega}$, there exists a Lie algebra  homomorphism $f\in\Cinfty(M)\mapsto -X_f\in \Ham(M,\tau,\omega)$.

\section{Cosymplectic Marsden--Weinstein reduction}
\label{Sec::CosMWRed}

This section aims to introduce the notions and results needed to present the cosymplectic Marsden--Weinstein reduction \cite{Al89, LMZ22} and to analyse its relation with the standard symplectic Marsden--Weinstein one. The ideas given in this section will be generalised in forthcoming sections, with the help of other new techniques, to devise a $k$-polycosymplectic reduction theory. Moreover, we here slightly generalise known facts on the relation of the cosymplectic and the symplectic Marsden--Weinstein reductions by removing unnecessary technical conditions used in the previous literature \cite{LS93}.

\subsection{Momentum maps and cosymplectic reductions}
\label{Sec::AGeneralMomentumMap}

\begin{definition}
A \textit{cosymplectic Lie group action} relative to $(M,\tau,\omega)$ is a Lie group action $\Phi\colon G\times M\rightarrow M$ such that, for every $g\in G$, the map $\Phi_g\colon x\in M\mapsto \Phi(g,x)\in  M$ is a cosymplectomorphism. 
\end{definition}

As $M$ is assumed to be connected, $\Phi\colon G\times M\rightarrow M$ is a cosymplectic Lie group action for $(M,\tau,\omega)$ if, and only if,
\begin{equation}
    \Lie_{\xi_M}\tau = 0\,,\qquad\Lie_{\xi_M}\omega = 0\,,\qquad\forall\xi\in\g\,,
\end{equation}
where $\xi_M$ is the fundamental vector field of $\Phi$ related to $\xi\in\g$, namely
\begin{equation}
    \xi_M(x) = \restr{\frac{\d}{\d s}}{s=0}\Phi(\exp(s\xi),x)\,,\qquad \forall x\in M\,.
\end{equation}
Since $\d\tau = 0$, the condition $\Lie_{\xi_M}\tau = 0$ implies that $\inn_{\xi_M}\tau$ takes a constant, not necessarily zero, value on $M$. This apparently minor detail has relevant applications in the reduction of cosymplectic manifolds to symplectic manifolds and their applications to circular restricted three-body problems (cf. \cite{Al89,LMZ22}). As shown in Section \ref{Sec::SPRed}, the fact that $\iota_{\xi_M}\tau$ may not be zero will play a relevant role in the description of radically new types of Marsden--Weinstein reductions. 

\begin{definition}
    Consider a cosymplectic Lie group action $\Phi\colon G\times M\rightarrow M$ relative to $(M,\tau,\omega)$ such that $\inn_{\xi_M}\tau = 0$ for every $\xi\in\g$. A {\it cosymplectic momentum map} for the Lie group action $\Phi$ is a map $\bfJ^\Phi\colon M\rightarrow \g^*$ such that
    \begin{equation}\label{Eq::CosMomMap}
        \inn_{\xi_M}\omega = \d\langle \bfJ^\Phi,\xi\rangle = \d J_\xi\,,\qquad RJ_\xi = 0\,,\qquad \forall\xi\in\g\,.
    \end{equation}
\end{definition}
Instead of the term ``cosymplectic momentum map'', we will frequently use the term ``momentum map'' when it is clear from the context what we mean. The condition $RJ_\xi=0$, for every $\xi\in \Lg$, is required to apply the cosymplectic reduction theorem to be introduced afterwards in this  section. This condition is not too restrictive, as many relevant non-autonomous Hamiltonian systems satisfy it \cite{LMZ22}. It is worth noting that only cosymplectic Lie group actions whose fundamental vector fields are Hamiltonian admit cosymplectic momentum maps.

Before continuing our exposition on cosymplectic geometry, it is appropriate to recall that every Lie group action $\Phi: G\times M\rightarrow M$ amounts to a Lie group whose elements are the diffeomorphisms $\Phi_g: M\ni m\mapsto \Phi(g,m)\in M$ for every $g\in G$. Moreover, every Lie group $G$ acts on itself by inner automorphisms, namely $G$ gives rise to a Lie group action $I:(g,h)\in G\times G\mapsto I_{g}(h)=ghg^{-1}\in G$, whose mappings $I_g:h\in G\mapsto I(g,h)\in G$, with $g\in G$, are called {\it inner automorphisms}. This leads to the {\it adjoint action} of $G$ on its Lie algebra $\mathfrak{g}$, which takes the form $\Ad:(g,v)\in G\times\mathfrak{g}\mapsto \Ad_g(v)=\T_eI_g(v)\in  \mathfrak{g}$. In turn, the adjoint action gives rise to the {\it co-adjoint action} $\Ad^*:(g,\vartheta)\in G\times \mathfrak{g}^*\mapsto \Ad^*_{g^{-1}}\vartheta= \vartheta\circ \Ad_{g^{-1}}\in \mathfrak{g}^*$. It is worth stressing that, according to our notation, $
({\rm Ad}^*)_g={\rm Ad}^*_{g^{-1}}$.

Hereafter, a cosymplectic manifold $(M,\tau,\omega)$ will be sometimes denoted by $M_\tau^\omega$ to shorten the notation. 

\begin{definition}
The triple $(M^\omega_\tau,h,\bfJ^\Phi)$ is called a {\it $G$-invariant cosymplectic Hamiltonian system} relative to $(M,\om,\tau)$ when $\Phi$ is a cosymplectic Lie group action such that $\Phi_g^*h=h$ for every $g\in G$, and $\bfJ^\Phi:M\rightarrow \mathfrak{g}^*$ is a cosymplectic momentum map related to $\Phi$. An {\it $\Ad^*$-equivariant $G$-invariant cosymplectic Hamiltonian system} is a $G$-invariant cosymplectic Hamiltonian system with an {\it $\Ad^*$-equivariant momentum map}, namely ${\bf J}^\Phi\circ \Phi_g={\rm Ad}^*_{g^{-1}}\circ {\bf J}^\Phi$ for every $g\in G$.
\end{definition}

The following technical result, which is well-known \cite{AM78}, will be of interest. 

\begin{lemma}\label{AdjointActionTh}
If $\Phi: G \times M \rightarrow M$ is a Lie group action, then $\left(\Ad_{g*}\xi\right)_M= \Phi_{g*}\xi_M$ for every $g\in G$ and $\xi\in\mathfrak{g}$.
\end{lemma}

In Marsden--Weinstein reductions, the role played by the co-adjoint action can be substituted by means of a new action on $\mathfrak{g}^*$ (see \cite{LMZ22, OR04}), whose form is justified by means of the following propositions and results.

\begin{proposition}
\label{Prop::CosNonAdEquiv}
Let $(M^\omega_\tau, h,\bf J^\Phi)$ be a $G$-invariant cosymplectic Hamiltonian system. Define the functions $\psi_{g,\xi}\in \Cinfty(M)$ of the form
\[
\psi_{g,\xi}=J_{\xi}\circ\Phi_g- J_{\Ad_{g^{-1}}\xi}\,,\qquad \forall g\in G,\quad\quad \forall \xi\in \mathfrak{g}\,.
\]
Then, $\psi_{g,\xi}$ is constant on $M$ for every $g\in G$ and $\xi\in \mathfrak{g}$. Moreover, the mapping $\sigma:G\ni g\mapsto \sigma(g)\in\Lg^*$ determined by $\<\sigma(g),\xi\>=\psi_{g,\xi}$ satisfies
\begin{equation}\label{Eq::SigmaProperty}
\sigma(gh)=\sigma(g)+\Ad_{g^{-1}}\sigma(h)\,,
\end{equation}
and
\begin{equation}\label{eq::SigmaProperty2}
\sigma(g)=\B J^\Phi(\Phi_g(x))-\Ad_{g^{-1}}^*\B J^\Phi(x)\,,\quad \forall g,h\in G\,,\quad\forall x\in M\,,
\end{equation}
where $\Ad^*_{g^{-1}}\vartheta=\vartheta \circ \Ad_{g^{-1}}$ for all $g\in G$ and $\vartheta\in\mathfrak{g}^*$.
\end{proposition}

The mapping $\sigma$ defined in Proposition \ref{Prop::CosNonAdEquiv} is called the {\it  co-adjoint cocycle} associated with the cosymplectic momentum map $\B J^\Phi$ on $M$. It steams from Proposition \ref{Prop::CosNonAdEquiv} that $\mathbf{J}^\Phi$ is $\Ad^*$-equivariant if and only if $\sigma=0$. Hence, $\sigma$ measures the lack of $\Ad^*$-equivariance of its associated momentum map. A map $\sigma:G\rightarrow\mathfrak{g}^*$ is a {\it coboundary}, if there exists $\mu\in\mathfrak{g}^*$ such that
\begin{equation}
\label{CEq:on}
\sigma(g)=\mu-\Ad_{g^{-1}}^*\mu,\qquad \forall g\in G.
\end{equation}
 
Every coboundary satisfies the relation \eqref{Eq::SigmaProperty}, which motivates that it is also called a {\it co-adjoint cocycle}. The space of co-adjoint cocycles on a Lie group $G$ admits an equivalence relation, whose equivalence classes, the so-called {\it cohomology classes}, are determined by setting that two co-adjoint cocycles are related if their difference is a coboundary. The following proposition shows that a cosymplectic Lie group action admitting a cosymplectic momentum map with associated co-adjoint cocycle $\sigma$ is such that it gives rise to cohomology class $[\sigma]$ that is independent of the chosen associated cosymplectic momentum map. 

\begin{proposition}
Let $\Phi: G\times M\rightarrow M$ be a cosymplectic action. If ${\bf J}_1^\Phi$ and ${\bf J}_2^\Phi$ are two momentum maps with co-adjoint  cocycles $\sigma_1$ and $\sigma_2$, respectively, then $[\sigma_1]=[\sigma_2]$.
\end{proposition}

If a momentum map ${\bf J}^\Phi : M\rightarrow \mathfrak{g}^*$ associated with a Lie group action $\Phi: G\times M\rightarrow M$ is equivariant with respect to a Lie group action $\Psi: G\times \mathfrak{g}^*\rightarrow \mathfrak{g}^*$, then it is said that ${\bf J}^\Phi$ is 
{\it $\Psi$-equivariant}. That is why the following proposition can be summarised by saying that ${\bf J}^\Phi$ is $\Delta$-equivariant for the so-called {\it affine Lie group action} $\Delta: G\times\mathfrak{g}^*\rightarrow\mathfrak{g}^*$ to be defined by ${\bf J}^\Phi$.

\begin{proposition}\label{Prop::GenEqJ}
Let ${\bf J}^\Phi:M\rightarrow\mathfrak{g}^*$ be a momentum map associated with a cosymplectic Lie group action $\Phi:G\times M\rightarrow M$ with co-adjoint cocycle $\sigma$. Then,
\begin{enumerate}[{\rm(1)}]
    \item the map $\Delta:G\times \mathfrak{g}^*\ni(g,\mu)\mapsto \Delta_g=\Ad_{g^{-1}}^*\mu +\sigma(g)\in\mathfrak{g}^*$ is a Lie group action,
    \item the cosymplectic momentum map ${\bf J}^\Phi$ is equivariant with respect to $\Delta$, in other words, every $g\in G$ gives rise to a commutative diagram
\begin{center}
 \begin{tikzcd}
   M 
   \arrow{r}{\Phi_g}
   \arrow{d}{\B J^\Phi}
   &
   M
   \arrow{d}{\B J^\Phi}
   \\
   \mathfrak{g}^*
   \arrow{r}{\Delta_g}
   &
   \Lg^*.
 \end{tikzcd}
\end{center}
\end{enumerate}
\end{proposition}

\begin{theorem}
\label{Th::SigmabracketCo}
Let $\Phi: G\times M\rightarrow M$ be a cosymplectic Lie group action with a cosymplectic momentum map $\B J^\Phi: M\rightarrow\mathfrak{g}^*$ and let $\sigma: G\rightarrow \mathfrak{g}^*$ be the co-adjoint cocycle associated with $\B J^\Phi$. Define
\begin{equation*}
\sigma_\nu:G\ni g\longmapsto \langle\sigma(g),\nu\rangle\in\mathbb{R}\,,\qquad \Sigma:\mathfrak{g}\times\mathfrak{g}\ni(\xi_1,\xi_2)\longmapsto \T_e\sigma_{\xi_2}(\xi_1) \in\mathbb{R}\,, \quad\forall\nu\in\mathfrak{g}\,.
\end{equation*}
Then,
\begin{enumerate}[{\rm(1)}]
    \item the map $\Sigma$ is a skew-symmetric bilinear form on $\mathfrak{g}$ satisfying that
    \begin{equation}
    \Sigma(\xi,[\zeta,\nu]) + \Sigma(\nu,[\xi,\zeta]) + \Sigma(\zeta,[\nu,\xi]) = 0\,,\qquad \forall \xi,\zeta,\nu\in \mathfrak{g}\,,
    \end{equation}
    \item $\Sigma(\xi,\nu) = \{J_\nu,J_\xi\}_{\tau,\omega}-J_{[\nu,\xi]}$ for all $\xi,\nu\in\mathfrak{g}$.
\end{enumerate}
\end{theorem}


Recall that if $\mathbf{J}^\Phi$ is an $\Ad^*$-equivariant momentum map, then its associated co-adjoint cocycle satisfies $\sigma(g)=0$ for every $g\in G$. Thus, $\Sigma(\xi,\eta)=0$ for all $\xi,\eta\in\mathfrak{g}$ and the following corollary becomes an immediate consequence of Theorem \ref{Th::SigmabracketCo}.
\begin{corollary}\label{Cor:Lieal2}
    If $\mathbf{J}^\Phi:M\rightarrow\mathfrak{g}^*$ is an $\Ad^*$-equivariant momentum map relative to $(M,\tau,\omega)$, then
    \begin{equation}
    \{J_\xi,J_\eta\}_{\tau,\omega}=J_{[\xi,\eta]}\,,\qquad \forall \xi,\eta\in\mathfrak{g}\,.
    \end{equation}
    In other words, $\lambda:\xi\in \mathfrak{g}\mapsto J_\xi\in \Cinfty(M)$ is a Lie algebra homomorphism.
\end{corollary}

Let us give some technical definitions that are useful to remove certain conditions appearing in many Marsden--Weinstein reduction theories \cite{MRSV15}.

Recall that a {\it weakly regular value} of ${\bf J}^\Phi:M\rightarrow \Lg^*$ is a point $\mu\in\Lg^*$ such that ${\bf J}^{\Phi-1}(\mu)$ is a submanifold in $M$ and $\T_x({\bf J}^{\Phi-1}(\mu))=\ker \T_x{\bf J}^\Phi$ for every $x\in {\bf J}^{\Phi-1}(\mu)$ (see \cite{Al89}). It is hereafter assumed that $\mu\in \mathfrak{g}^*$ is a weakly regular value of ${\bf J}^\Phi$. Additionally, we also assume that the isotropy subgroup $G_\mu^\Delta$ of $\mu\in\Lg^*$ relative to the affine Lie group action $\Delta:G\times\mathfrak{g}^*\rightarrow\mathfrak{g}^*$ acts via $\Phi$ on ${\bf J}^{\Phi-1}(\mu)$ in a {\it quotientable manner}, namely ${\bf J}^{\Phi-1}(\mu)/G_\mu^\Delta$ is a manifold and the projection $\pi_\mu:{\bf J}^{\Phi-1}(\mu)\rightarrow {\bf J}^{\Phi-1}(\mu)/G_\mu^\Delta$ is a submersion. To guarantee that ${\bf J}^{\Phi-1}(\mu)/G^\Delta_\mu$ is a manifold, one may require  $G^\Delta_\mu$ to act freely and properly on ${\bf J}^{\Phi-1}(\mu)$ (see \cite{AM78,Al89,LMZ22} for details).

Let us enunciate, without proofs, the generalisation to the cosymplectic realm of the standard symplectic Marsden--Weinstein reduction (see \cite{Al89,LMZ22} for details).

\begin{lemma}\label{lem:NonAdPerp}
    Let $\mu\in\g^*$ be a weak regular value of a momentum map $\bfJ^\Phi\colon M\rightarrow\g^*$ associated with a cosymplectic Lie group action $\Phi: G\times M\rightarrow M$ relative to $(M,\tau,\omega)$ and let $G^\Delta_\mu$ be the isotropy group at $\mu\in\g^*$ of the action $\Delta \colon G\times\g^*\rightarrow\g^*$ relative to the co-adjoint cocycle $\sigma\colon G\rightarrow\g^*$ of $\bfJ^\Phi$. Then, for every $x\in\bfJ^{\Phi-1}(\mu)$, one has 
    \begin{enumerate}[{\rm(1)}]
        \item $\T_x(G^\Delta_{\mu} x) = \T_x(G x)\cap \T_x(\bfJ^{\Phi-1}(\mu))$, 
        \item $\T_x(\bfJ^{\Phi-1}(\mu)) = \T_x(Gx)^{\perp_\omega}$,
        \item $\left(\T_x\bfJ^{\Phi-1}(\mu)\right)^{\perp_\omega} = \T_x(Gx)\oplus\langle R_x\rangle$.
    \end{enumerate}
\end{lemma}

\begin{theorem}\label{Th:CoSymRed}
Let $\Phi: G\times M\rightarrow M$ be a cosymplectic Lie group action relative to the cosymplectic manifold $(M,\tau,\omega)$ and associated with a cosymplectic momentum map ${\bf J}^\Phi: M\rightarrow\mathfrak{g}^*$. Assume that $\mu\in\mathfrak{g}^*$ is a weakly regular value of ${\bf J}^\Phi$ and let ${\bf J}^{\Phi-1}(\mu)$ be quotientable with respect to the action of $G_\mu^\Delta$ induced by  $\Phi$. 
Let $\jmath_\mu:{\bf J}^{\Phi-1}(\mu)\hookrightarrow M$ be the natural immersion and let $\pi_\mu:{\bf J}^{\Phi-1}(\mu)\rightarrow M^\Delta_\mu={\bf J}^{\Phi-1}(\mu)/G^\Delta_\mu$ be the canonical projection. Then, there exists a unique cosymplectic manifold $(M_\mu^\Delta,\tau_\mu,\omega_\mu)$ such that
\begin{equation}
\label{Eq::CoSymRed}
\jmath_\mu^*\tau = \pi_\mu^*\tau_\mu\,,\qquad \jmath_\mu^*\omega = \pi_\mu^*\omega_\mu\,.
\end{equation}
\end{theorem}

\subsection{Cosymplectic and symplectic Marsden--Weinstein reductions}
\label{Sec:CosymSym}

Let us recall how the cosymplectic Marsden--Weinstein reduction can be reformulated by using the classical Marsden--Weinstein reduction theorem (see \cite{LS93} for details). To start with, the following lemma ensures that every cosymplectic manifold naturally gives rise to a symplectic form on a manifold of a larger dimension (see \cite{LMZ22} for details). This idea will be generalised in the following sections to devise a $k$-polycosymplectic Marsden--Weinstein reduction from a $k$-polysymplectic one.

\begin{lemma}\label{lem:CosymSym}
    Let $\omega'\in\Omega^2(M)$, $\tau'\in\Omega^1(M)$ and consider the canonical projection ${\rm pr}\colon\R\times M\rightarrow M$. Let $u$ be the pull-back to $\mathbb{R}\times M$ of the natural coordinate in $\R$ relative to the projection $\pi_\R:\R\times M\rightarrow \R$. Then, $(M,\tau',\omega')$ is a cosymplectic manifold if and only if $(\R\times M,\,\widetilde{\omega}' = {\rm pr}^*\omega' + \d u\wedge {\rm pr}^*\tau')$ is a symplectic manifold. Moreover, ${\rm pr}$ is a Poisson morphism, i.e.
    $$
    \{f_1,f_2\}_{\tau,\omega}\circ {\rm pr}=\{f_1\circ {\rm pr},f_2\circ {\rm pr}\}_{\widetilde{\omega}'}\,,\qquad \forall f_1,f_2\in \Cinfty(\R\times M)\,. 
    $$
\end{lemma}

Let $\Psi\colon G\times M\rightarrow M$ be a cosymplectic Lie group action with an associated cosymplectic momentum map $\bfJ^\Psi\colon M\rightarrow\g^*$ relative to $(M,\tau,\omega)$. Then, $\Psi$ and $\bfJ^\Psi$ can be extended to $\R\times M$ in the following ways, respectively,
\begin{equation}
    \widetilde{\Psi}\colon (g,u,x)\in G\times \R\times M\mapsto  (u,\Psi_g(x))\in \R\times M
\end{equation}
and
\begin{equation}
    \bfJ^{\wtl\Psi}:(u,x) \in \R\times M\mapsto \bfJ^\Psi(x)\in\g^*\,.
\end{equation}
The fundamental vector fields $\xi_M$, with $\xi\in\mathfrak{g}$, related to $\Psi$, can be understood as vector fields on $\R\times M$ via the isomorphisms $\T_{(u,x)}(\R\times M)\simeq \T_u \R\times \T_x M$ for every $(u,x)\in \R\times M$. In this manner, they are locally Hamiltonian relative to $\widetilde \omega$ if and only if $\d(R J_\xi) = 0$ (see \cite{LMZ22}). Hence, the condition $R J_\xi = 0$ ensures that $\widetilde \Psi$ admits a momentum map $\bfJ^{\wtl\Psi}$ relative to the symplectic manifold $(\R\times M,\widetilde\omega)$, namely $\inn_{\xi_{\R\x M}}\widetilde{\omega}=\d\<\mathbf{J}^{\wtl\Psi},\xi\>$. Moreover, if $\bfJ^\Psi$ is $\Delta$-equivariant with respect to $\Psi$, then $\mathbf{J}^{\widetilde\Psi}$ is also $\Delta$-equivariant with respect to $\widetilde\Psi$. Further, since $\mathbf{J}^{\widetilde \Psi-1}(\mu)\simeq\R\times {\mathbf{J}}^{\Psi-1}(\mu)$ for every $\mu\in\Lg^*$ and ${\rm pr}\circ \widetilde{\Psi}_g=\Psi_g\circ {\rm pr}$ for every $g\in G$, i.e. $\widetilde\Psi$ does not change the first component of $\R\times M$, then $\mathbf{J}^{\widetilde \Psi-1}(\mu)$ is quotientable by $G^\Delta_\mu$ if and only if ${\mathbf{J}}^{\Psi-1}(\mu)$ is so. Moreover, $\mu\in\g^*$ is a (resp. weak) regular value of $ {\mathbf{J}}^\Psi$ if and only if $\mu$ is a (resp. weak) regular value of $\mathbf{J}^{\widetilde \Psi}$. Therefore, we can apply the Marsden--Weinstein reduction theorem to the symplectic manifold $(\R\times M,\widetilde\omega)$ to obtain a reduced symplectic manifold $\widetilde M^\Delta_\mu = \mathbf{J}^{\widetilde \Psi-1}(\mu)/G_\mu^\Delta$ endowed with the {\it reduced symplectic form}, $\widetilde\omega_\mu$, determined univocally by the condition
\begin{equation}
    \widetilde \jmath_\mu^{\;*}\widetilde\omega=\widetilde\pi_\mu^*\widetilde\omega_\mu\,,
\end{equation}
where $\widetilde\jmath_\mu:\mathbf{J}^{\widetilde \Psi-1}(\mu)\hookrightarrow \R\times M$ is the natural immersion and $\widetilde\pi_\mu:\mathbf{J}^{\widetilde \Psi-1}(\mu)\rightarrow \mathbf{J}^{\widetilde \Psi-1}(\mu)/G^\Delta_\mu$ is the canonical projection \cite{AM78}. Note that $\widetilde M^\Delta_\mu\simeq (\R\times\mathbf{J}^{ \Psi-1}(\mu))/G^\Delta_\mu\simeq\R\times M^\Delta_\mu$, where $M^\Delta_\mu=\mathbf{J}^{\Psi-1}(\mu)/G^\Delta_\mu$. Then, a reduced cosymplectic manifold $(M_\mu^\Delta, \tau_\mu,\omega_\mu)$ can be retrieved from $\widetilde\omega_\mu$ in the following way
\begin{equation*}
     \tau_\mu = \iota_u^*\left(\inn_{\tparder{}{u}}\widetilde\omega_\mu\right)\,,\qquad \omega_\mu = \iota_u^*\widetilde\omega_\mu\,,
\end{equation*}
where $\iota_u:M^\Delta_\mu\ni[x]\mapsto(u,[x])\in\R\times M^\Delta_\mu$ and $[x]$ stands for the orbit of $x\in\mathbf{J}^{\Psi-1}(\mu)$ relative to $G^\Delta_\mu$. In particular,
\begin{equation}
    \d\omega_\mu = \d\iota_u^*\widetilde{\omega}_\mu = \iota_u^*\d\widetilde{\omega}_\mu = 0\,,\qquad \d\tau_\mu = \d\iota_u^*\left(\inn_{\tparder{}{u}}\widetilde\omega_\mu\right) = \iota_u^*\left(\Lie_{\tparder{}{u}}\widetilde\omega_\mu\right) = 0\,,
\end{equation}
where the last equality holds since $\Lie_{\tparder{}{u}}\widetilde\omega_\mu = 0$. Moreover, if $X\in\mathfrak{X}(M^\Delta_\mu)$,  then $\iota_{u*}X$ takes values in $\ker \d u$. If, additionally,  $\inn_X\omega_\mu = 0$ and $\inn_X\tau_\mu = 0$, then $\iota_{\iota_{u*}X}\widetilde{\omega}_\mu=0$ and $X=0$ because $\widetilde{\omega}_\mu$ is symplectic. Hence, $\ker\omega_\mu\cap\ker\tau_\mu=0$. A Reeb vector field $R$ on $M$ gives rise to a unique  vector field $\widetilde{R}$ on $\mathbb{R}\times M$ projecting onto $M$ via ${\rm pr}$  and taking values in $\ker \d u$. Since $\widetilde{R}$ is tangent to ${\bf J}^{\widetilde{\Psi}-1}(\mu)$ and projectable onto a vector field $\widetilde{R}_\mu$ on $\widetilde{M}^\Delta_\mu$, one has that $\iota_{R_\mu}\tau_\mu=\iota_{\widetilde{R}_\mu}\iota_{\partial/\partial u}\widetilde{\omega}_\mu=\iota_R\iota_{\partial/\partial u}\omega=1$. Hence, $\tau_\mu$ is different from zero, and $(M^\Delta_\mu,\tau_\mu,\omega_\mu)$ becomes a cosymplectic manifold. It can be shown that $(M^\Delta_\mu,\tau_\mu,\omega_\mu)$ does not depend on the particular map $\iota_u$. Indeed, this fact can be considered as a particular consequence of the proof of Theorem \ref{Th::kpolycoreduction}. 

The above approach shows that the cosymplectic Marsden--Weinstein reduction can be obtained through a symplectic reduction on $\R\times M$ of a particular type. Note that instead of $T=\R$, one can also consider $T=\mathbb{S}^1$ endowed with $\d\theta$, where $\theta$ is a locally defined angular coordinate on $\mathbb{S}^1$ giving rise to a global closed differential one-form. The above procedure is analogous in this latter case.

What has not been stressed so far in the literature is that the above discussion also implies that cosymplectic geometry can be understood as a type of symplectic geometry in a larger manifold. Although this can be of interest, it is known that the extension of mathematical entities on a cosymplectic manifold to a symplectic one of a larger dimension may change the properties of such entities in such a way that their study may be more complicated. This may happen, for instance, in the study of energy-momentum methods  \cite{LMZ22}.



\section{Fundamentals on geometric field theory}
\label{Sec::FundFieldTh}

Let us review the geometric fundamentals needed to develop our further geometric formulation for Hamiltonian field theories (see \cite{Aw92, LM98,Gra2020, Gu87,  MC11, MRSV15, MLM10} for details on $k$-polysymplectic and $k$-polycosymplectic formalisms). We hereafter assume that $\mathbb{R}^k$ has a fixed basis $\{e_1,\ldots,e_k\}$ giving rise to a dual basis $\{e^1,\ldots, e^k\}$ in $\R^{k*}$.

\subsection{\texorpdfstring{$k$}--Polysymplectic and \texorpdfstring{$k$}--symplectic manifolds}

This section surveys the theory of $k$-polysymplectic, polysymplectic, $k$-symplectic structures, and related concepts that appear in the literature and the terminology to be employed in this work. This introduces the definitions to be used and allows one to understand the relations between the results of our work and other previous ones. This description is quite relevant as the terminology appearing in the literature is not unified and the same term can refer to different not equivalent geometric concepts. Some examples of this can be found in the foundational works by Günther \cite{Gu87} and Awane \cite{Aw92}.

\begin{definition}
A {\it $k$-polysymplectic form} on $M$ is a closed non-degenerate $\R^k$-valued differential two-form
    $$ \bm{\omega} = \omega^\alpha\otimes e_\alpha\in\Omega^2(M,\R^k)\,. $$
The pair $(M,\bm{\omega})$ is called a {\it $k$-polysymplectic manifold}.
\end{definition}

$k$-Polysymplectic manifolds are called, for simplicity, polysymplectic manifolds in the literature (see \cite{MRSV15} for instance). Nevertheless, the latter term refers to a different notion that is shown below. That is why we will not simplify the term `$k$-polysymplectic manifold' and other related ones in our work so as to avoid  misunderstandings. 
It is worth noting that $M$ has a $k$-polysymplectic form $\bm{\omega}$ if, and only if, there exists a family of $k$ closed two-forms $\omega^1,\dotsc,\omega^k\in\Omega^2(M)$ such that
$$ 
\ker \bm{\omega}=\ker (\omega^\alpha\otimes e_\alpha)=\bigcap_{\alpha = 1}^k \ker\omega^\alpha = 0\,. 
$$
Hereafter, $\R^k$-valued differential forms will be written in bold.
Now, we can define polysymplectic manifolds as follows.

\begin{definition}\label{Def::CoNotions}
    Let $M$ be an $n(k+1)$-dimensional manifold. Then,
    \begin{itemize}
        \item A {\it polysymplectic structure} on $M$ is a differential two-form taking values in $\mathbb{R}^k$ given by $\bomega=\omega^\alpha\otimes e_\alpha\in\Omega^2(M,\R^k)$ for certain $\omega^1,\dotsc,\omega^k\in\Omega^2(M)$ such that
        $$
        \ker \bomega=\bigcap_{\alpha = 1}^k \ker\omega^\alpha =0\,. 
        $$
        In this case, $(M,\bomega)$ is called a {\it polysymplectic manifold}.
        \item A {\it $k$-symplectic structure} on $M$ is a pair $(\bomega,V)$, where $(M,\bomega)$ is a polysymplectic manifold and $V\subset\T M$ is an integrable distribution on $M$ of rank $nk$ such that
        $$ \restr{\bomega}{V\times V} = 0\,. $$
        In this case, $(M,\bm\omega,V)$ is a {\it $k$-symplectic manifold}. We call $V$ a {\it polarisation} of the $k$-symplectic manifold.
    \end{itemize}
    If the two-form $\bm\omega$ is exact, namely $\bomega = \d\bm\theta$ for some $\bm\theta \in\Omega^1(M,\R^k)$, in any of the above-mentioned concepts introduced in Definition \ref{Def::CoNotions}, then such concepts are said to be {\it exact}.
\end{definition}

The previous $k$-symplectic manifold notion coincides with the one given by A.~Awane \cite{Aw92, Aw00}. 
In addition, it is locally equivalent to the concepts of 
\textit{standard polysymplectic structure} of C.~Günther \cite{Gu87} 
(they are globally equivalent provided there exist compatible Darboux charts\footnote{Note that it is not clear what G\"unther means by an atlas of canonical charts, namely which is the equivalence between different pairs of Darboux charts.}) and globally equivalent to the
\textit{integrable $p$-almost cotangent structure} introduced by M.~de León \textit{et al} \cite{dLe88, dLe93}. 
In the case of $k=1$, Awane's definition reduces to the notion of \textit{polarised symplectic manifold}, 
namely a symplectic manifold with a Lagrangian distribution  \cite{LV2013}. 

G\"unther calls polysymplectic manifolds the differential geometric structures obtained from our $k$-symplectic definition by removing the existence of the distribution $V$. 
Meanwhile, a \textit{standard} polysymplectic manifold in G\"unther's paper \cite{Gu87} is a polysymplectic manifold admitting local Darboux coordinates, 
which is equivalent to our definition of a $k$-symplectic manifold. Note that the condition concerning the polarisation $V$ in the definition of a $k$-symplectic structure is necessary to ensure the existence of an atlas of compatible Darboux-type coordinates (see \cite{Aw92, GLRR22} and \cite[p. \!57]{PhDThesisXRG}) and vice versa. 

It is very important to remark that the polysymplectic reduction in \cite{MRSV15} does not rely on any relationship between the dimension of the manifold and the number $k$ of the $k$-polysymplectic form $\bm\omega$ on it. It is also relevant to stress that the polysymplectic manifold term in \cite{MRSV15} is just a simplification of the term $k$-polysymplectic manifold, which is defined in our work. Finally, Definition \ref{Def::CoNotions} leads to a linear analogue definition by assuming ${\bm \omega}$ to be restricted to a point $x\in M$. This allows us to define $k$-{\it polysymplectic structures on linear spaces}, {\it $k$-polysymplectic spaces}, and so on. In such cases, one assumes ${\bm \omega}\in \Lambda^2E^*\otimes\mathbb{R}^k$, where $E$ is a vector space and $\Lambda^2E^*$ stands for the space of two-covectors on $E$ while $V$ is substituted by a linear subspace $W\subset E$.


\begin{example}[Canonical model for $k$-symplectic manifolds]\label{ex:canonical-model-k-symplectic}\rm
    Let $Q$ be an $n$-dimensional manifold and consider the Whitney sum
    $$ \bigoplus^k\cT Q = \cT Q\oplus_Q\overset{(k)}{\dotsb}\oplus_Q \cT Q\,, $$
    with natural projections $\pi^\alpha\colon\bigoplus^k\cT Q\to\cT Q$, from the $\alpha$-th component of $\bigoplus^k\cT Q$ onto $\cT Q$, with $\alpha=1,
\ldots,k$, and $\pi_Q\colon\bigoplus^k\cT Q\to Q$. A coordinate system $\{q^i\}$ in $Q$ induces a natural coordinate system $\{q^i,p_i^\alpha\}$ in $\bigoplus^k\cT Q$, where $\alpha$ ranges from $1$ to $k$. Consider the canonical forms in the cotangent bundle $\T^*Q$  of $Q$ given by $\theta\in\Omega^1(\cT Q)$ and $\omega = -\d\theta\in\Omega^2(\cT Q)$. Hence, the Whitney sum $\bigoplus^k\cT Q$ has the canonical forms taking values in $\mathbb{R}^k$ given by
    $$ \bm{\theta}_k = (\pi^\alpha)^\ast\theta\otimes e_\alpha\ ,\quad \bm \omega_k = -\d\bm\theta_k\,, 
    $$
    which, in natural coordinates $\{q^i,p^\alpha_i\}$ in $\bigoplus^k\cT Q$, read
    $$ 
    \bm \theta_k= p_i^\alpha\d q^i\otimes e_\alpha \ ,\quad \bm\omega_k = \d q^i\wedge\d p_i^\alpha\otimes e_\alpha \,.
    $$
    Taking all this into account, the triple $(\bigoplus^k\cT Q,\boldsymbol{\omega}_k,V_k)$, with $V_k = \ker\T\pi_Q$, is a $k$-symplectic manifold. Notice that the natural coordinates $\{q^i,p_i^\alpha\}$ in $\bigoplus^k\cT Q$ are the canonical example of {\it $k$-symplectic Darboux coordinates}. 
\end{example}

Given a $k$-symplectic manifold $(M,\boldsymbol{\omega},V)$, the vector bundle morphism
\begin{equation}\label{eq:k-symplectic-flat}
    \flat: (v_1,\dotsc,v_k)\in \bigoplus^k\T M \longmapsto \inn_{v_\alpha}\langle \bm \omega,e^\alpha\rangle \in\cT M\,,
\end{equation}
induces a morphism of $\Cinfty(M)$-modules $\flat\colon\X^k(M)\to\Omega^1(M)$. 
The morphism $\flat$ is surjective because the annihilator of its image belongs to $\bigcap_{\alpha=1}^k\ker\omega^\alpha=\ker\bm\omega=0$.

\subsection{\texorpdfstring{$k$}--Polycosymplectic and \texorpdfstring{$k$}--cosymplectic manifolds}

Non-autonomous field theories can be modelled by means of the so-called $k$-polycosymplectic \cite{MC11} and $k$-cosymplectic \cite{LM98} geometries. Let us introduce the basic definitions related to these geometric theories.

\begin{definition}
\label{Def::polyco}
         A {\it $k$-polycosymplectic structure} on $M$ is a pair $(\bm\tau ,\bm\omega )$, where $\bm\tau\in\Omega^1(M,\R^k)$ and $\bm\omega \in\Omega^2(M,\R^k)$ are closed differential one- and two-forms taking values in $\R^k$ such that
        \[
        \rk \ker \bm\omega=\rk\left(\bigcap_{\alpha = 1}^k\ker\omega^\alpha\right) =  k\,,\qquad\ker\bm\omega\cap \ker \bm\tau=\bigcap_{\alpha = 1}^k(\ker\tau^\alpha\cap\ker\omega^\alpha)= 0\,. 
        \]
        In this case, $(M,\bm{\tau},\bm{\omega})$ is called a {\it $k$-polycosymplectic manifold}. If, in addition, $\dim M = k + n(k+1)$ for a certain $n\in\mathbb{N}$, it is said that $(M,\bm{\tau},\bm{\omega})$ is a {\it polycosymplectic manifold} and $(\bm\tau ,\bm\omega )$ is a {\it polycosymplectic structure}. 
        \end{definition}

        \begin{definition}\label{Def::polycoes}
         A {\it $k$-cosymplectic structure} on $M$ is a family $(\bm{\tau},\bm{\omega},V)$, where $(\bm{\tau},\bm{\omega})$ is a polycosymplectic structure on $M$ and $V\subset\T M$ is a distribution of rank $nk$ on $M$ such that
        $$
        \restr{\bm{\tau}}{V} = 0\qquad\text{and}\qquad \restr{\bm{\omega}}{V\times V} = 0\,.
        $$
        Then, $(M,\bm{\tau},\bm{\omega},V)$ is {\it $k$-cosymplectic manifold}.
    
\end{definition}
If $\bm \omega$ is exact, namely $\bm\omega = \d\bm\theta$ for some $\bm\theta\in\Omega^1(M,\R^k)$,  the polycosymplectic (resp. $k$-polycosymplectic or $k$-cosymplectic)  structure is said to be {\it exact}. If not otherwise stated, we will assume that the indices $\alpha,\beta$ range from $1$ to $k$.
Throughout the paper, $M^{\bm \omega}_{\bm\tau}$ will be sometimes used to denote a $k$-polycosymplectic manifold $(M,\bm\tau,\bm\omega)$. This will allow for shortening the notation. Note that Definition \ref{Def::polyco} and Definition \ref{Def::polycoes} can be immediately modified to define a linear analogue, namely by assuming ${\bm \tau}\in E^*\otimes \mathbb{R}^k$ and ${\bm \omega}\in \Lambda^kE^*\otimes\mathbb{R}^k$ for a linear space $E$ and $\Lambda^kE^*$ being the space of $k$-covectors of $E$. 

For clarity, it is convenient to prove the following result. Moreover, this will be necessary to relate our $k$-polycosymplectic manifolds to a certain type of $k$-polysymplectic manifolds.
\begin{proposition}\label{prop:reeb-polyco}
    Let $(M,\bm\tau,\bm\omega)$ be a $k$-polycosymplectic manifold. There exists a unique family of vector fields $R_1,\ldots,R_k$ on $M$, called {\it Reeb vector fields}, such that
    \begin{equation}
        \label{Eq::ReebCon}
    \inn_{R_\alpha}\bm \tau = e_\alpha\,,\qquad \inn_{R_\alpha}\bm \omega = 0\,,\qquad \alpha=1,\ldots,k\,.
    \end{equation}
\end{proposition}
\begin{proof} By Definition \ref{Def::polyco}, one has $\ker \bm\tau\cap \ker \bm \omega=0$, which means that, if $\bm\tau=\tau^\alpha\otimes e_\alpha$, then $\tau^1\wedge \ldots\wedge \tau^k$ does not vanish on $D=\ker \bm \omega$. The distribution $D$ has rank $k$ by the definition of a $k$-polycosymplectic manifold. Therefore,  $\tau_x^1|_{D_x},\ldots,\tau_x^k|_{D_x}$  are linearly independent at every $x\in M$ and the restrictions of $\tau^1,\ldots,\tau^k$ to $D$ admit a unique dual basis $R_1,\ldots,R_k$ of vector fields on $M$ taking values in $D$. Then, the vector fields $R_1,\ldots,R_k$ satisfy the conditions \eqref{Eq::ReebCon}.
\end{proof}



\subsection{\texorpdfstring{$k$}--Vector fields and integral sections}

The notion of a $k$-vector field is of great use in the geometric study of partial differential equations \cite{MRSV15}. In particular, the dynamics of the so-called $k$-polycosymplectic Hamiltonian systems are determined by $k$-vector fields. Moreover, the reduction of a $k$-polycosymplectic Hamiltonian system is accomplished by reducing an associated  Hamiltonian $k$-vector field. Let us present some relevant details.

 Consider the Whitney sum of $k$ copies of the tangent bundle  to $M$, namely $\bigoplus^k\T M = \T M\oplus_M\overset{(k)}{\dotsb}\oplus_M \T M$, and the natural projections
\begin{equation*}
    {\rm pr}^\alpha\colon\bigoplus^k\T M\to\T M\ , \qquad {\rm pr}^1_M\colon\bigoplus^k\T M\to M,\qquad\alpha=1,\dots,k.
\end{equation*}

\begin{definition}
    A {\it $k$-vector field} on a manifold $M$ is a section
    $ \bfX\colon M\to\bigoplus^k\T M $
    of ${\rm pr}^1_M$. The space of  $k$-vector fields on $M$ is denoted by  $\X^k(M)$.
\end{definition}

\begin{figure}[ht]
    \centering
    \begin{tikzcd}[row sep=huge,
    column sep=huge]
        & \bigoplus^k\T M \arrow[d, "{\rm pr}^\alpha"]\\
        M \arrow[r, "X_\alpha"] \arrow[ur, "\bfX"] & \T M
    \end{tikzcd}
    \label{Diag::kvectorfield}
\end{figure}

Taking into account the commutative diagrams above, which are concerned with $\alpha=1,\ldots,k$, a $k$-vector field $\bfX\in\X^k(M)$ amounts to $k$ vector fields $X_1,\dotsc,X_k\in\X(M)$ such that $X_\alpha = {\rm pr}^\alpha\circ\bfX$ with $\alpha=1,\ldots,k$. With this in mind, it makes sense to denote $\bfX = (X_1,\dotsc, X_k)$. A $k$-vector field $\bfX = (X_1,\dotsc,X_k)$ induces a decomposable contravariant skew-symmetric tensor field, $X_1\wedge\dotsb\wedge X_k$, which is a section of the bundle $\bigwedge^k\T M\to M$. This also induces a (generalised) distribution on $M$ spanned by the vector fields $X_1,\ldots, X_k$ on $M$.

\begin{definition}\label{dfn:first-prolongation-k-tangent-bundle}
    Given a map $\phi\colon U\subset\R^k\to M$, its {\it first prolongation} to $\bigoplus^k\T M$ is the map $\phi'\colon U\subset\R^k\to\bigoplus^k\T M\,$
    defined by
    $$ \phi'(s) = \left( \phi(s); \T\phi\left( \parder{}{s^1}\bigg\vert_s \right),\dotsc,\T\phi\left( \parder{}{s^k}\bigg\vert_s \right) \right) = (\phi(s); \phi'_1(s),\ldots,\phi'_k(s))\,, $$
    where $s = (s^1,\dotsc,s^k)$ and $\{s^1,\ldots,s^k\}$ are the canonical coordinates of $\R^k$.
\end{definition}

Analogously to the case of integral curves of vector fields, one can define integral sections of a $k$-vector field as follows.

\begin{definition}
    Let $\bfX = (X_1,\dotsc,X_k)\in\X^k(M)$ be a $k$-vector field. An {\it integral section} of $\bfX$ is a map $\phi\colon U\subset\R^k\to M$ such that
    $$ \phi' = \bfX\circ\phi\,, $$
    that is, $\T\phi\circ\parder{}{s^\alpha} = X_\alpha\circ\phi$ for $\alpha=1,\ldots,k$. A $k$-vector field $\bfX\in\X^k(M)$ is {\it integrable} if every point of $M$ is in the image of an integral section of $\bfX$.
\end{definition}

Consider a $k$-vector field $\bfX = (X_1,\ldots, X_k)$ on $M$ with a local expression
\[
X_\alpha = X_\alpha^i\parder{}{x^i}\,,\qquad \alpha=1,\ldots,k\,,
\]
where $i$ ranges from $1$ to $\dim M$.
Then, $\phi\colon s\in U\subset\R^k\mapsto \phi(s)\in M$ is an integral section of $\bfX$ if, and only if, $\phi$ is a solution of the system of partial differential equations
$$ \parder{\phi^i}{s^\alpha} = X_\alpha^i(\phi)\,,\qquad i=1,\ldots,\dim M,\quad \alpha=1,\ldots,k\,.$$

Let $\bfX = (X_1,\dotsc,X_k)$ be a $k$-vector field on $M$. Then, $\bfX$ is integrable if, and only if, $[X_\alpha,X_\beta] = 0$ for $\alpha,\beta=1,\ldots,k$. These are precisely the necessary and sufficient conditions for the integrability of the above systems of partial differential equations in {\it normal form}, namely the partial derivatives of the coordinates $\phi$ at a point of $M$ can be written as  functions of the value of $\phi$ at that point \cite{Le13}.

\subsection{Non-autonomous field theory}

Given an $\R^k$-valued differential $\ell$-form $\bm\theta = \theta^\alpha\otimes e_\alpha\in\Omega^\ell(M,\R^k)$, the contraction of $\bm \theta$ with a vector field $X\in\X(M)$ is defined as
\[
\inn_X\bm\theta = (\inn_X\theta^\alpha)\otimes e_\alpha\in\Omega^{\ell-1}(M,\R^k)\,.
\]
For a $k$-vector field $\bfX = (X_1,\dotsc,X_k)\in\X^k(M)$, its contraction with $\bm\theta$ reads
\[
\inn_{\bfX}\bm\theta = \inn_{X_\alpha}\theta^\alpha\in\Omega^{\ell-1}(M)\,. 
\]
The exterior product of two $\mathbb{R}^k$-valued differential forms $\bm{\vartheta} = \vartheta^\alpha\otimes e_\alpha\in\Omega^{\ell_1}(M,\R^k)$ and $\bm{\mu} = \mu^\alpha\otimes e_\alpha\in\Omega^{\ell_2}(M,\R^k)$ is defined as follows
$$ \bm{\vartheta}\barwedge\bm{\mu} = \sum_{\alpha=1}^k(\vartheta^\alpha\wedge\mu^\alpha)\otimes e_\alpha\in\Omega^{\ell_1 + \ell_2}(M,\R^k)\,. $$

The above definitions will be very useful to simplify the notation of our further theory. Note that a point-wise analogue of the above definitions can be defined {\it mutatis mutandis}. In particular, the latter conveys implicitly some definitions for the contraction of elements of $E$ or $E\otimes \mathbb{R}^k$ with $E^*\otimes \mathbb{R}^k$ for a vector space $E$.

\begin{definition}
\label{Def::PolycoHamSystem}
    Let $(M,\bm\tau,\bm\omega)$ be a $k$-polycosymplectic manifold and let $h\in\Cinfty(M)$. Then, $(M,\bm\tau,\bm\omega,h)$ is called a {\it $k$-polycosymplectic Hamiltonian system}. A $k$-vector field on $M$, let us say $\bfX = (X_1,\dotsc,X_k)\in\X^k(M)$, is a {\it $k$-polycosymplectic Hamiltonian $k$-vector field} if it is a solution to the system of equations
    \begin{equation}\label{eq:polycosymplectic-equations-fields}
        \begin{dcases}
            \inn_{\bfX}\bomega = \d h - (R_\alpha h)\tau^\alpha\,,\\
            \inn_{X_\beta}\bm\tau = e_\beta\,,
        \end{dcases}\qquad \beta=1,\ldots,k.
    \end{equation}
    The function $h$ is the {\it Hamiltonian function} of $\bfX$.
    Hereafter, $\Xh^k(M,{\bm \tau},{\bm \omega})$ stands for the space of $k$-polycosymplectic Hamiltonian $k$-vector fields on $M$. 
\end{definition}

If the $k$-polycosymplectic manifold is understood from the context, we will simply write $\mathfrak{X}^k(M)$ for its space of Hamiltonian $k$-polycosymplectic vector fields on $M$. It is worth noting that every function $f\in\Cinfty(M)$ admits different $k$-polycosymplectic Hamiltonian $k$-vector fields. 
If $k = 1$, Definition \ref{Def::PolycoHamSystem} retrieves the notion of a cosymplectic Hamiltonian system.

Suppose that, around a point $x\in M$, we have Darboux coordinates $\{t^\alpha,q^i,p_i^\alpha\}$ and consider now a $k$-vector field $\bfX = (X_1,\dotsc,X_k)\in\X^k(M)$ given locally by
$$
X_\alpha = (X_\alpha)^\beta\parder{}{t^\beta} + (X_\alpha)^i\parder{}{q^i} + (X_\alpha)^\beta_i\parder{}{p_i^\beta}\,. 
$$
If $\bfX$ is a $k$-polycosymplectic Hamiltonian vector field,  conditions \eqref{eq:polycosymplectic-equations-fields} give
\begin{equation}\label{eq:polycosymplectic-equations-fields-coordinates}
    (X_\alpha)^\beta = \delta_\alpha^\beta\,,\qquad
    \parder{h}{p_i^\beta} = (X_\beta)^i\,,\qquad
    \parder{h}{q^i} = -\sum_{\alpha=1}^k(X_\alpha)^\alpha_i\,.
\end{equation}
This shows that every $h\in \Cinfty(M)$ may have different $k$-polycosymplectic Hamiltonian $k$-vector fields.
Let $\psi:\R^k\to M$ be an integral section with local expression
$$
\psi(s) = (t^\alpha(s),q^i(s), p_i^\alpha(s))\,,\quad s\in\R^k\,,
$$
 of a $k$-polycosymplectic Hamiltonian $k$-vector field $\bfX$.
In this case, $\psi$ satisfies the systems of partial differential equations
\begin{equation}\label{eq:polycosymplectic-hamiltonian-equations-coordinates}
        \parder{t^\beta}{s^\alpha} = \delta_\alpha^\beta\,,\qquad
        \parder{q^i}{s^\alpha} = \parder{h}{p_i^\alpha}\,,\qquad
\sum_{\alpha=1}^k\parder{p_i^\alpha}{s^\alpha} = -\parder{h}{q^i}        \,.
\end{equation}

\begin{example}[The vibrating membrane with external force]

Consider a horizontal vibrating membrane with coordinates $\{x,y\}$ subjected to a time-dependent external force given by a function $f(t,x,y)$. The phase space of this system is $M = \R^3\times\bigoplus^3\cT\R$ and it admits global Darboux coordinates $\{t,x,y, \zeta, p^t,p^x,p^y\}$, where $u$ stands for the distance of every point in the membrane with respect to its equilibrium position, and $p^t,p^x,p^y$ are the corresponding momenta. This system is described by the Hamiltonian function $h\in\Cinfty(M)$ given by
$$ h(t,x,y,\zeta,p^t,p^x,p^y) = \frac{1}{2}(p^t)^2 - \frac{1}{2c^2}(p^x)^2 - \frac{1}{2c^2}(p^y)^2 - \zeta f(t,x,y)\,, $$
where $c\in\R$ is a constant related to the physical features of the membrane, which takes into account its properties and tension. In this case, equations \eqref{eq:polycosymplectic-hamiltonian-equations-coordinates} for a section $$\psi:(t,x,y)\in \mathbb{R}^3\mapsto (t,x,y,\zeta(t,x,y),p^t(t,x,y),p^x(t,x,y),p^y(t,x,y))\in \R^3\times\textstyle\bigoplus_{\alpha=1}^3\cT \R$$ 
yield
\begin{gather}
    \parder{p^t}{t} + \parder{p^x}{x} + \parder{p^y}{y} = f(t,x,y)\,,\\
    \parder{\zeta}{t} = p^t\,,\quad \parder{\zeta}{x} = -\frac{1}{c^2}p^x\,,\quad \parder{\zeta}{y} = -\frac{1}{c^2}p^y\,.
\end{gather}
Note that the use of $\{t,x,y\}$ as the coordinates of the domain of an integral section of ${\bf X}_{\widetilde{h}}$ is a slight abuse of notation that is nevertheless common in the literature \cite{PhDThesisXRG} and it will be hereafter employed in this work.   
Combining the previous equations, we obtain the equation of a forced vibrating membrane, namely
$$ \parder{^2\zeta}{t^2} - \frac{1}{c^2}\left( \parder{^2\zeta}{x^2} + \parder{^2\zeta}{y^2} \right) = f(t,x,y)\,. $$
To describe this system in polar coordinates, we can consider the Hamiltonian function
$$ \widetilde h(t,r,\theta,\zeta,p^t,p^r,p^\theta) = \frac{1}{2r}\left( (p^t)^2 - \frac{1}{c^2}(p^r)^2 - \frac{r^2}{c^2}(p^\theta)^2 \right) - r \zeta f(t,r,\theta)\,, $$
and now equations \eqref{eq:polycosymplectic-hamiltonian-equations-coordinates} for a section $$\psi:(t,r,\theta)\in \mathbb{R}^3\mapsto (t,r,\theta,\zeta(t,r,\theta),p^t(t,r,\theta),p^x(t,r,\theta),p^y(t,r,\theta))\in \R^3\times\textstyle\bigoplus_{\alpha=1}^3\cT \R $$ become
\begin{gather}
    \parder{p^t}{t} + \parder{p^r}{r} + \parder{p^\theta}{\theta} = rf(t,r,\theta)\,,\\
    \parder{\zeta}{t} = \frac{1}{r}p^t\,,\quad \parder{\zeta}{r} = -\frac{1}{rc^2}p^r\,,\quad \parder{\zeta}{\theta} = -\frac{r}{c^2}p^\theta\,.
\end{gather}
Combining these equations, the equation of a forced vibrating membrane in polar coordinates reads
$$ \parder{^2\zeta}{t^2} - c^2\left( \parder{^2\zeta}{r^2} + \frac{1}{r}\parder{\zeta}{r} + \frac{1}{r^2}\parder{^2\zeta}{\theta^2} \right) = f(t,r,\theta)\,. $$
\end{example}

\section{\texorpdfstring{$k$}--Polysymplectic Marsden--Weinstein reduction without Ad\texorpdfstring{$^*$}--equivariance}
\label{Sec::PolysymplecticReduction}

This section introduces the notion of a momentum map for $k$-polysymplectic manifolds. In particular, we present $\Ad^{*k}$-equivariant $k$-polysymplectic momentum maps, which are well-known in Section \ref{Sec::ParticularPoly}. A theory of affine Lie group actions for $k$-polysymplectic manifolds, which seems to be new, will be devised in Section \ref{SubSec::Generalmomentummapspolysym}. We will show that ${\rm Ad}^{*k}$-equivariance is not required to achieve a $k$-polysymplectic Marsden--Weinstein reduction in Section \ref{SubSec::FinalRedPoly}.

\subsection{\texorpdfstring{$k$}--Polysymplectic momentum maps}\label{Sec::ParticularPoly}

The following definition introduces a fundamental class of Lie group actions preserving a $k$-polysymplectic form. As shown later, these actions play a crucial role in the Marsden--Weinstein $k$-polysymplectic reduction theory.

\begin{definition}
        An action $\Phi\colon G\times P\to P$ of a Lie group $G$ on a $k$-polysymplectic manifold $(P,\bm\omega)$ is said to be a \textit{$k$-polysymplectic Lie group action} if $\Phi_g^*\bm\om=\bm\om$ for each $g\in G$.
\end{definition}

\begin{definition}
\label{Def::PolysymMomentumMap}
A {\it $k$-polysymplectic momentum map} for a Lie group action $\Phi: G\!\x \!P\!\rightarrow\! P$ relative to a $k$-polysymplectic manifold $(P,\bm\om)$ is a map $\mathbf{J}^\Phi:P\rightarrow (\mathfrak{g}^*)^k$ such that
\begin{equation}
\label{Eq::PolycoMomentumMap}
\inn_{\xi_P}\bm\omega=(\iota_{\xi_P}\omega^\alpha)\otimes e_\alpha=\d\left<\mathbf{J}^\Phi,\xi\right\>,\qquad \forall \xi\in \Lg.
\end{equation}
\end{definition}
Note that equation \eqref{Eq::PolycoMomentumMap} implies that $\mathbf{J}^\Phi:P\rightarrow(\Lg^*)^k$ satisfies 
\begin{equation}
\label{Eq:EqCond}
\inn_{\boldsymbol{\xi}_P}\boldsymbol{\omega}=\d\left\<\mathbf{J}^\Phi,\boldsymbol{\xi}\right\>,\qquad\forall\boldsymbol{\xi}\in\Lg^k,
\end{equation}
where ${\bm \xi}_P$ is the $k$-vector field on $P$ whose $k$ vector field components are the fundamental vector fields of $\Phi$ related to the $k$-components of ${\bm \xi}\in \mathfrak{g}^k$. If we write ${\bm \xi}=(0,\ldots,  \overset{(\alpha)}{\xi} ,\ldots,0)\in \mathfrak{g}^k$ for any $\xi\in \mathfrak{g}$ and $\alpha=1,\ldots,k$ and impose \eqref{Eq:EqCond} to hold for a basis $\{\xi_1,\ldots,\xi_r\}$ for $\mathfrak{g}$,  we obtain $kr$ conditions, which uniquely fix the value of the $kr$ coordinates of ${\bf J}^\Phi$.
Conversely, the equation \eqref{Eq::PolycoMomentumMap} evaluated on the previous basis of $\mathfrak{g}$ imposes $r$ conditions for each one of the $k$ components of $\mathbf{J}^\Phi$, giving rise to $kr$ conditions.

The following definition is well-known and widely used in the literature \cite{MRSV15}, but we changed the notation of ${\rm Coad}^k$ to $\Ad^{*k}$ to shorten it. Nevertheless, we will see later that the $\Ad^{*k}$-equivariance condition is not necessary.

\begin{definition}
    A $k$-polysymplectic momentum map $\mathbf{J}^\Phi:P\rightarrow (\Lg^*)^k$ is {\it $\Ad^{*k}$-equivariant} if
    \[
    \mathbf{J}^\Phi\circ \Phi_g =\Ad^{*k}_{g^{-1}}\circ\,\, \mathbf{J}^\Phi\,,\quad \forall g\in G\,,
    \]
    where ${\rm Ad}^{*k}_{g^{-1}}={\rm Ad}^*_{g^{-1}}\stackrel{}\otimes \overset{(k)}{\dotsb}\otimes {\rm Ad}^*_{g^{-1}}$ and
    \[
    \begin{array}{rccc}
    \Ad^{*k}&:G\x(\Lg^*)^k & \longrightarrow & (\Lg^*)^k\\
    & (g,\bm \mu) &\longmapsto & \Ad^{*k}_{g^{-1}}\bm\mu\,.
    \end{array}
    \]
    In other words, for every $g\in G$, the following diagram commutes
\begin{center}
    \begin{tikzcd}
    P
    \arrow[r,"\mathbf{J}^\Phi"]
    \arrow[d,"\Phi_g"]& (\Lg^*)^k
    \arrow[d,"\Ad^{*k}_{g^{-1}}"]\\
    P
    \arrow[r,"\mathbf{J}^\Phi"]&
    (\Lg^*)^k.
    \end{tikzcd}
    \end{center}
\end{definition}

To simplify the notation, let us introduce the following definition. 

\begin{definition}
The four-tuple $(P,\bm\omega,h,{\bf J}^\Phi)$ is called a {\it $G$-invariant $k$-polysymplectic Hamiltonian system} if it consists of a $k$-polysymplectic manifold $(P,\bm \om)$, a $k$-polysymplectic Lie group action $\Phi:G\x P\rightarrow P$ such that $\Phi_g^*h=h$ for every $g\in G$, and a $k$-polysymplectic momentum map $\mathbf{J}^\Phi$ related to $\Phi$. An {\it $\Ad^{*k}$-equivariant $G$-invariant $k$-polysymplectic Hamiltonian system} is a $G$-invariant $k$-polysymplectic Hamiltonian system whose $k$-polysymplectic momentum map is $\Ad^{*k}$-equivariant.
\end{definition}

\subsection{General \texorpdfstring{$k$}--polysymplectic momentum maps}
\label{SubSec::Generalmomentummapspolysym}
This section develops the theory of $k$-polysymplectic momentum maps that are not necessarily $\Ad^{*k}$-equivariant. In particular, we demonstrate that every $k$-polysymplectic momentum map $\mathbf{J}^\Phi: P\rightarrow(\Lg^*)^k$ admits a Lie group action on $(\Lg^*)^k$ such that $\mathbf{J}^\Phi$ is equivariant with respect to such a Lie group action. The hereafter introduced techniques are more complex technically, but analogous, to the ones used in \cite{OR04, Za21}, where the $\Ad^{*k}$-equivariant momentum map methods on symplectic manifolds were extended by removing the ${\rm Ad}^{*k}$-equivariance condition. Recall that all manifolds are assumed to be connected unless otherwise stated.

\begin{proposition}
\label{Prop::PsiConstant}
Let $(P,\bm\omega,h,\mathbf{J}^\Phi)$ be a $G$-invariant $k$-polysymplectic Hamiltonian system and let us define the functions on $P$ given by
\[
    \bm\psi _{g,\boldsymbol{\xi}}:P\ni x\longmapsto  {\bf J}^\Phi_{\boldsymbol{\xi}}(\Phi_g(x))-{\bf J}^\Phi_{\Ad_{g^{-1}}^k\boldsymbol{\xi}}(x)\in\R\,,\quad\forall g\in G\,,\quad\forall \boldsymbol{\xi}\in\mathfrak{g}^k\,.
\]
Then, $\bm\psi_{g,\boldsymbol{\xi}}$ is constant on $P$ for every $g\in G$ and $\boldsymbol{\xi}\in\Lg^k$. Moreover, the function $\bm\sigma:G\ni g\mapsto {\bm \sigma}(g)\in\mathfrak (\mathfrak{g}^*)^k$ univocally determined by the condition $\<\bm\sigma(g),\boldsymbol{\xi}\>=\bm\psi_{g,\boldsymbol{\xi}}$ satisfies
\[    \bm\sigma(g_1g_2)=\bm\sigma(g_1)+\Ad^{*k}_{g_1^{-1}}\bm\sigma(g_2)\,,\quad\forall g_1,g_2\in G\,.
\]
\end{proposition}
\begin{proof}
    Note that
    \begin{align*}
        \d\bm\psi_{g,\boldsymbol{\xi}} &=   \d({\bf J}^\Phi_{\bm\xi}\circ\Phi_g)-\d {\bf J}^\Phi_{\Ad_{g^-1}^k\boldsymbol{\xi}} =  \Phi_g^*(\inn_{\boldsymbol{\xi}_P}\bomega) - \inn_{(\Ad_{g^{-1}}^k\boldsymbol{\xi})_P}\bm\omega  \\
        &=  \inn_{\Phi_{g^{-1}*}\boldsymbol{\xi}_P}\Phi_g^*\bm\omega-\inn_{\Phi_{g^{-1}*}\boldsymbol{\xi}_P}\bm\omega =  \inn_{\Phi_{g^{-1}*}\boldsymbol{\xi}_P}\bm\omega-\inn_{\Phi_{g^{-1}*}\boldsymbol{\xi}_P}\bm\omega =0\,,
    \end{align*}
    where we have used that $ \Phi$ is a $k$-polysymplectic Lie group action and the well-known fact that $(\Ad_g\xi)_P=\Phi_{g*}\xi_P$ for every $g\in G$ and each $\xi\in\Lg$ (see \cite{AM78}). Hence, $(\Ad^k_{g^{-1}}\boldsymbol{\xi})_P=\Phi_{g^{-1}*}\boldsymbol{\xi}_P$ for every $\boldsymbol{\xi}\in\Lg^k$. Therefore, $\bm\psi_{g,\boldsymbol{\xi}}$ is constant on $P$ for all $g\in G$ and any $\boldsymbol{\xi}\in\Lg^k$.
    
    Let us rewrite $\bm\psi_{g,\boldsymbol{\xi}}$ as follows
    \begin{multline*}
        \bm\psi _{g,\boldsymbol{\xi}}  =   {\bf J}^\Phi_{\boldsymbol{\xi}}\circ \Phi_g-{\bf J}^\Phi_{\Ad_{g^{-1}}^k\boldsymbol{\xi}}= \left\langle {\bf J}^\Phi\circ\Phi_g,\bm\xi\right\rangle- \left\langle {\bf J}^\Phi,\Ad_{g^{-1}}^k\bm\xi\right\rangle\\
        =  \left\langle {\bf J}^\Phi\circ \Phi_g,\bm\xi\right\rangle- \left\langle \Ad_{g^{-1}}^{*k}{\bf J}^\Phi,\bm\xi\right\rangle =\left\langle   {\bf J}^\Phi\circ\Phi_g-\Ad_{g^{-1}}^{*k}{\bf J}^\Phi ,\bm\xi\right\rangle\,,
    \end{multline*}
    where $\Ad^{k}_{g^{-1}}:\Lg^k\rightarrow \Lg^k$ is the transpose to $\Ad^{*k}_{g^{-1}}$.
    Hence,
    \[
        \bm\sigma:G\ni g\longmapsto  {\bf J}^\Phi\circ \Phi_g-\Ad_{g^{-1}}^{*k}{\bf J}^\Phi= \bm\sigma(g)\in(\mathfrak{g}^*)^k\,.
    \]
Thus, $\bm\sigma(g)$ is constant on $P$ for every $g\in G$ and $\<\bm\sigma(g),\bm\xi\> =\bm\psi_{g,\boldsymbol{\xi}}$, for every $g\in G$ and $\boldsymbol{\xi}\in \mathfrak{g}^k$. Then, 
    \begin{align*}
        \bm\sigma(g_1g_2) &=  \left({\bf J}^\Phi\circ\Phi_{g_1g_2}-\Ad_{{(g_1g_2)}^{-1}}^{*k}{\bf J}^\Phi\right)= \left({\bf J}^\Phi\circ\Phi_{g_1}\circ\Phi_{g_2}-\Ad_{{g_1}^{-1}}^{*k}\Ad_{{g_2}^{-1}}^{*k}{\bf J}^\Phi\right)\\
        &=  \left({\bf J}^\Phi\circ\Phi_{g_1}\circ\Phi_{g_2}-\Ad^{*k}_{{g_1}^{-1}}({\bf J}^\Phi\circ\Phi_{g_2})+\Ad^{*k}_{g_1^{-1}}({\bf J}^\Phi\circ\Phi_{g_2})-\Ad_{{g_1}^{-1}}^{*k}\Ad_{{g_2}^{-1}}^{*k}{\bf J}^\Phi\right)\\
        &=  \left({\bf J}^\Phi\circ\Phi_{g_1}\!-\!\Ad^{*k}_{{g_1}^{-1}}{\bf J}^\Phi\!+\!\Ad^{*k}_{{g_1}^{-1}}({\bf J}^\Phi\circ\Phi_{g_2}\!-\!\Ad_{{g_2}^{-1}}^{*k}{\bf J}^\Phi)\right) = \bm\sigma(g_1)\!+\!\Ad_{{g_1}^{-1}}^{*k}\bm\sigma(g_2)
    \end{align*}
    for any $g_1,g_2\in G$.
\end{proof}

The map $\bm\sigma:G\rightarrow (\mathfrak{g}^*)^k$ of the form
\[
\bm\sigma(g)=\mathbf{J}^\Phi\circ \Phi_g-\Ad^{*k}_{g^{-1}}\mathbf{J}^\Phi,\qquad g\in G,
\]
is called the {\it co-adjoint cocycle} associated with the $k$-polysymplectic momentum map $\mathbf{J}^\Phi$ on $P$. A map $\bm\sigma:G\rightarrow(\mathfrak{g}^*)^k$ is a {\it coboundary} if there exists $\bm\mu\in(\mathfrak{g}^*)^k$ such that
\[
\bm\sigma(g)=\bm\mu-\Ad_{g^{-1}}^{*k}\bm\mu,\qquad \forall g\in G.
\]
If ${\bf J}^\Phi$ is an $\Ad^{*k}$-equivariant $k$-polysymplectic momentum map, then $\bm\sigma=0$.

The next proposition shows that every $k$-polysymplectic Lie group action admitting a $k$-polysymplectic momentum map induces a well-defined cohomology class $[\bm\sigma]$, namely $[{\bm \sigma}]$ is the same for every $k$-polysymplectic momentum map induced by the initial $k$-polysymplectic Lie group action. The proof is analogous to the one of the equivalent proposition on symplectic manifolds \cite{OR04, Za21}.

\begin{proposition}\label{Prop::CoAdjCocycles}
Let $\Phi: G\times P \rightarrow P$ be a $k$-polysymplectic Lie group action. If $\mathbf{J}^{1\Phi}$ and $\mathbf{J}^{2\Phi}$ are two associated $k$-polysymplectic momentum maps with co-adjoint cocycles $\bm\sigma_1$ and $\bm\sigma_2$, respectively, then $[\bm\sigma_1]=[\bm\sigma_2]$.
\end{proposition}

\begin{proposition}\label{Prop::GenEqJPolySym}
Let ${\bf J}^\Phi:P\rightarrow (\Lg^*)^k$ be a $k$-polysymplectic momentum map associated with a $k$-polysymplectic Lie group action $\Phi:G\times P\rightarrow P$ with co-adjoint cocycle $\bm\sigma$. Then,
\begin{enumerate}[{\rm(1)}]
    \item the map 
    \[
    \bm\Delta:G\times (\Lg^*)^k\ni(g,\bm \mu)\mapsto \bm\sigma(g)+\Ad^{*k}_{g^{-1}}\bm\mu=\bm\Delta_g(\bm\mu)\in(\Lg^*)^k,
    \]
    is a Lie group action of $G$ on $(\Lg^*)^k$,
    \item the $k$-polysymplectic momentum map ${\bf J}^\Phi$ is equivariant with respect to $\bm\Delta$, in other words, for every $g\in G$, one has a commutative diagram
\end{enumerate}

\begin{center}
    \begin{tikzcd}
    P
    \arrow[r,"\mathbf{J}^\Phi"]
    \arrow[d,"\Phi_g"]& (\Lg^*)^k
    \arrow[d,"\bm\Delta_g"]\\
    P
    \arrow[r,"\mathbf{J}^\Phi"]&
    (\Lg^*)^k.
    \end{tikzcd}
    \end{center}
\end{proposition}
\begin{proof}
   First, since $\bm\sigma(e)=0$, one has
    \[
    \bm\Delta(e,\bm\mu)=\bm\mu+\bm\sigma(e)=\bm\mu\,,
    \]
    Thus, $\bm\Delta(e,\bm \mu)=\bm \mu$. Moreover, by Proposition \ref{Prop::PsiConstant}, one gets
    \begin{align*}
    \bm\Delta(g_1,\bm\Delta(g_2,\bm\mu))&=\Ad_{g_{1}^{-1}}^{*k}(\Ad_{g_{2}^{-1}}^{*k}\bm\mu +\bm\sigma(g_2))+\bm\sigma(g_1)=\Ad_{g_{1}^{-1}}^{*k}\Ad_{g_{2}^{-1}}^{*k}\bm\mu +\Ad_{g_{1}^{-1}}^{*k}\bm\sigma(g_2)+\bm\sigma(g_1)\\&=\Ad^{*k}_{(g_1g_2)^{-1}}\bm\mu +\Ad_{g_{1}^{-1}}^{*k}\bm\sigma(g_2)+\bm\sigma(g_1)=\Ad^{*k}_{(g_1g_2)^{-1}}\bm\mu +\bm\sigma(g_1g_2)=\bm\Delta(g_1g_2,\bm\mu)\,,
\end{align*}
for every $g_1,g_2\in G$ and $\bm \mu\in(\Lg^*)^k$. Hence, $\bm\Delta$ is a Lie group action of $G$ on $(\Lg^*)^k$.
Second, by definition of $\bm\Delta$ and $\bm\sigma$, one has
\[
    \bm\Delta_g\circ\mathbf{J}^\Phi=\Ad^{*k}_{g^{-1}}\mathbf{J}^\Phi+\bm\sigma(g)=\mathbf{J}^\Phi\circ \Phi_g\,,\qquad\forall g\in G\,,
\]
which shows that $\mathbf{J}^\Phi$ is $\bm\Delta$-equivariant.
\end{proof}

Proposition \ref{Prop::GenEqJPolySym} ensures that a general $k$-polysymplectic momentum map $\B J^\Phi$ gives rise to an equivariant $k$-polysymplectic momentum map relative to a new action $\bm\Delta: G\x(\mathfrak{g}^*)^k\rightarrow(\mathfrak{g}^*)^k$, called a {\it $k$-polysymplectic affine Lie group action}.

It should be noted that an affine Lie group action can also be expressed by writing $\bm\Delta(g,\bm\mu)=(\Delta^1_g\mu^1,\ldots,\Delta^k_g\mu^k)\in (\mathfrak{g}^*)^k$, where the mappings $\Delta^1,\ldots,\Delta^k$ take the form $\Delta^\alpha: G\x \mathfrak{g}^*\ni(g,\vartheta)\mapsto \Ad_{g^{-1}}^*\vartheta+\sigma^\alpha(g)=\Delta_g^\alpha(\vartheta)\in\mathfrak{g}^*$  and ${\bm \sigma}(g)=(\sigma^1(g),\ldots,\sigma^k(g))$, where  $\sigma^\alpha(g)=\mathbf{J}^\Phi_\alpha\circ\Phi_g-\Ad^*_{g^{-1}}\mathbf{J}^\Phi_\alpha$ for $\alpha=1,\ldots,k$.

\subsection{\texorpdfstring{$k$}--Polysymplectic Marsden--Weinstein reduction}\label{SubSec::FinalRedPoly}

This section aims to extend the $k$-polysymplectic reduction to the case when the associated $k$-polysymplectic momentum map $\mathbf{J}^\Phi$ is not necessarily $\Ad^{*k}$-equivariant. Recall that, at the end of  Section \ref{SubSec::Generalmomentummapspolysym}, we proved that every $k$-polysymplectic momentum map $\mathbf{J}^\Phi$ gives rise to a $k$-polysymplectic affine Lie group action $\bm\Delta: G\times (\mathfrak{g}^*)^k\rightarrow (\mathfrak{g}^*)^k$ making $\mathbf{J}^\Phi$ to be $\bm\Delta$-equivariant. Moreover, the isotropy group, $G^{\bm \Delta}_{\bm\mu}$, of $\bm\mu\in(\mathfrak{g}^*)^k$ relative to $\bm\Delta$ may be different from the isotropy group of $\boldsymbol{\mu}$ with respect to ${\rm Ad}^{*k}$ since the action of both on $(\mathfrak{g}^*)^k$ is not necessarily the same (cf. \cite[Theorem 6.1.1]{OR04}). This observation is crucial to generalise the $k$-polysymplectic Marsden--Weinstein reduction theorem \cite{MRSV15} to general not necessarily $\Ad^{*k}$-equivariant $k$-polysymplectic momentum maps. We here provide the generalisation to this new realm of the theorems presented in \cite{MRSV15}. Additionally, we introduce the quotients of $k$-polysymplectic manifolds related to general $k$-polysymplectic momentum maps, which will inherit, under mild conditions, a $k$-polysymplectic manifold (see \cite{MRSV15} for details). 

Let $(E,\bm\om)$ be a $k$-polysymplectic vector space and let $W$ be a linear subspace of $E$. Then, the {\it $k$-polysymplectic orthogonal complement} of $W$ relative to  $(E,\bm\om)$ is the linear subspace defined by
\[
    W^{\perp,k} = \{v\in E\mid \inn_w\inn_v\bm\omega=0\,, \forall w\in W\}\,.
\]

The following results are a slight generalisation of some known findings in $k$-polysymplectic geometry devised in \cite{MRSV15} so as to cover the case when $k$-polysymplectic momentum maps $\mathbf{J}^\Phi$ does not need to be ${\rm Ad}^*$-equivariant. In fact, the whole work \cite{MRSV15} can be rewritten by using $k$-polysymplectic affine Lie group actions instead of coadjoint actions. The changes to be performed are based merely on  substituting $G_{\mu^\alpha}$ by $G^{\Delta^\alpha}_{\mu^\alpha}$, $\mathfrak{g}_{\mu^\alpha}$ by $\mathfrak{g}^{\Delta^\alpha}_{\mu^\alpha}$, $G_{\bm \mu}$ by $G^{\bm \Delta}_{\bm \mu}$, as well as others minor amendments of the sort. Although changes to perform are simple, verifying where corrections have to be done represents a long work.

Let us start with our slight generalisation of a standard result in $k$-polysymplectic Marsden--Weinstein reduction theory.

\begin{lemma}
\label{Lemm::NonAdPerpPS}
Let $(P,\bm\om,h,\mathbf{J}^\Phi)$ be a $G$-invariant $k$-polysymplectic Hamiltonian system and let $\bm\mu\in(\Lg^*)^k$ be a weak regular value of a $k$-polysymplectic momentum map ${\bf J}^\Phi:P\rightarrow (\Lg^*)^k$. Then, for every $p\in {\bf J}^{\Phi-1}(\bm\mu)$, one has 
\begin{enumerate}[{\rm(1)}]
\item $\T_{p}(G^{\bm\Delta}_{\bm\mu} p) = \T_{p}(G p)\cap \T_{p}\left({\bf J}^{\Phi-1}({\bm\mu})\right)$, 
\item $\T_{p}({\bf J}^{\Phi-1}(\bm\mu)) = \T_{p}(Gp)^{\perp,k}$.
\end{enumerate}
\end{lemma}
\begin{proof}
Let $\xi_P(p)\in \T_p(G p)$ for some $\xi\in\mathfrak{g}$. Then,  $\xi_P(p)\in \T_p(G^{\bm\Delta}_{\bm\mu}p)$ if, and only if, $\xi_P(p)\in \T_p\left(\mathbf{J}^{\Phi-1}(\bm\mu)\right)$, or equivalently $\xi\in\mathfrak{g}^{\bm\Delta}_{\bm\mu}$ if, and only if, $\xi_P(p)\in \T_p\left(\mathbf{J}^{\Phi-1}(\bm\mu)\right)$, where $\mathfrak{g}^{\bm\Delta}_{\bm\mu}$ is the Lie algebra of $G^{\bm\Delta}_{\bm\mu}$.

The proof of (2) follows essentially the same as in \cite{MRSV15}.
\end{proof}


Let us provide a non-necessarily ${\rm Ad}^{*k}$-equivariant counterpart of some relevant technical results in \cite{MRSV15}, namely Lemmas 3.15 and 3.16, leading to the $k$-polysymplectic Marsden--Weinstein reduction theory.
\begin{lemma} The linear map
$$
\widetilde{\pi}_p^\alpha:\frac{\T_p({\bf J}^{\Phi-1}({\bm \mu}))}{\T_p(G^{\bm \Delta}_{\bm \mu}p)}\longrightarrow\frac{\left(\frac{\T_p{\bf J}^\Phi_\alpha}{\ker \omega^\alpha_p}\right)}{\{\T_p\pi_{\bm\mu}(\xi_P)_p:\xi\in \mathfrak{g}_{\mu^\alpha}^{\Delta^\alpha}\}},\qquad p\in {\bf J}^{\Phi-1}({\bm \mu})\subset P,
$$
for some $\alpha$ belonging to $\{1,\ldots,k\}$ is an epimorphism if, and only if,
$$
\ker \T_p{\bf J}^{\Phi}_\alpha=\T_p({\bf J}^{\Phi-1}({\bm \mu}))+\ker\omega^\alpha_p+\T_p(G_{\mu^\alpha}^{\Delta^\alpha}p).
$$
Meanwhile $\bigcap_{\alpha=1}^k\ker\widetilde{\pi}_p^\alpha=0$ if, and only if,
$$
\T_p(G_{\bm\mu}^{\bm\Delta} p)=\bigcap^k_{\alpha=1}\left(\ker\om^\alpha_p+\T_p(G^{\Delta^\alpha}_{\mu^\alpha}p)\right)\cap \T_p(\mathbf{J}^{\Phi-1}(\bm\mu))\,.
$$
\end{lemma}
The following theorems present and describe the main results of our slightly generalised $k$-polysymplectic Marsden--Weinstein reduction theory. It is extremely important to stress that there is no restriction on the dimension of $P$ for the following Theorem \ref{Th::PolisymplecticReductionJ} and Theorem \ref{Th::PolyReductionDynamics} to hold.

\begin{theorem}[The $k$-polysymplectic Marsden--Weinstein reduction theorem]\label{Th::PolisymplecticReductionJ}
    Consider a $G$-invariant $k$-polysymplectic Hamiltonian system  $(P,\bm\om,h,\mathbf{J}^\Phi)$. Assume that $\bm\mu\in (\Lg^*)^k$ is a weak regular value of $\mathbf{J}^\Phi$ and $G_{\bm\mu}^{\bm\Delta}$ acts in a quotientable manner on $\mathbf{J}^{\Phi-1}(\bm\mu)$. Let $G^{
    \Delta^\alpha}_{\mu^\alpha}$ denote the isotropy group at $\mu^\alpha$ of the Lie group action $\Delta^\alpha:(g,\vartheta)\in G\x\mathfrak{g}^*\mapsto  \Delta^\alpha(g,\vartheta)\in \mathfrak{g}^*$ for $\alpha=1,\ldots,k$. Moreover, let the following conditions hold
    \begin{equation}\label{Eq::PolysymplecticReduction1eq}
        \ker (\T_p\mathbf{J}_\alpha^\Phi) = \T_p(\mathbf{J}^{\Phi-1}(\bm\mu))+\ker\om^\alpha_p + \T_p(G^{\Delta^\alpha}_{\mu^\alpha} p)\,,
    \end{equation}
    \begin{equation}\label{Eq::PolysymplecticReduction2eq}
        \T_p(G_{\bm\mu}^{\bm\Delta} p)=\bigcap^k_{\alpha=1}\left(\ker\om^\alpha_p+\T_p(G^{\Delta^\alpha}_{\mu^\alpha}p)\right)\cap \T_p(\mathbf{J}^{\Phi-1}(\bm\mu))\,,
    \end{equation}
    for every $p\in {\bf J}^{\Phi-1}({\bm \mu})$ and all $\alpha=1,\ldots,k$. Then,  $(\mathbf{J}^{\Phi-1}(\bm\mu)/G^{\bm\Delta}_{\bm\mu},\bm\om_{\bm\mu})$ is a $k$-polysymplectic manifold, where ${\bm \omega}_{\bm \mu}$ is  univocally determined by
    \[
        \pi_{\bm\mu}^*\bm\om_{\bm\mu}=\jmath_{\bm\mu}^*\bm\om\,,
    \]
    while $\jmath_{\bm\mu}:\B J^{\Phi-1}(\bm\mu)\hookrightarrow P$ is the natural immersion and $\pi_{\bm\mu}:\mathbf{J}^{\Phi-1}(\bm\mu)\rightarrow \mathbf{J}^{\Phi-1}(\bm\mu)/G^{\bm\Delta}_{\bm\mu}$ is the canonical projection.
\end{theorem}

Note that the above theorem introduces a few improvements relative to the $k$-polysymplectic Marsden--Weinstein reduction theorem in \cite{MRSV15}.  For instance, the ${\rm Ad}^*$-equivariance of ${\bf J}^\Phi$ is removed, which is achieved by introducing a $k$-polysymplectic affine Lie group action ${\bm \Delta}$.

The  theorem below shows when a class of $k$-polysymplectic Hamiltonian system $(P,\bm\omega,h,\mathbf{J}^{\Phi})$ induces a $k$-polysymplectic Hamiltonian system on the reduced manifold $\mathbf{J}^{\Phi-1}(\bm\mu)/G^{\bm\Delta}_{\bm\mu}$ obtained via Theorem \ref{Th::PolisymplecticReductionJ}.

\begin{theorem}
\label{Th::PolyReductionDynamics}
Let the assumptions of Theorem \ref{Th::PolisymplecticReductionJ} be satisfied. Let $h\in \Cinfty(P)$ be a $G$-invariant Hamiltonian function relative to $\Phi$ and let $ {\bfX}^h=(X^h_1,\ldots, X^h_k)$ be a $k$-vector field associated with $h$. Assume that $\Phi_{g*}\mathbf{X}^h=\mathbf{X}^h$ for every $g\in G$ and ${\bf X}^h$ is tangent to $\mathbf{J}^{\Phi-1}(\bm\mu)$. Then, the flows $F^\alpha_t$ of $X^h_\alpha$ leave $\mathbf{J}^{\Phi-1}(\bm\mu)$ invariant and they induce a unique flow $\mathcal{F}^\alpha_t$ on $\mathbf{J}^{\Phi-1}(\bm\mu)/G_{\bm\mu}^{\bm\Delta}$ satisfying $\pi_{\bm\mu}\circ F^\alpha_t=\mathcal{F}^\alpha_t\circ \pi_{\bm\mu}$ for every $\alpha=1,\ldots,k$.
\end{theorem}

The proof of Theorem \ref{Th::PolyReductionDynamics}, based on Theorem \ref{Th::PolisymplecticReductionJ}, is essentially the same as in \cite{MRSV15}. It is worth noting that ${\bfX}^h$ does not need to be invariant relative to the action of $G$ even when its HDW equations are so. This will be illustrated in the examples of Section \ref{Se::Examples}.

\section{A \texorpdfstring{$k$}--polycosymplectic Marsden--Weinstein reduction}
\label{Sec::PolycoReduction}
This section contains some of the main results of the paper: the development of a new $k$-polycosymplectic Marsden--Weinstein reduction that does not depend on the ${\rm Ad}^{*k}$-equivariance of the employed $k$-polycosymplectic momentum map. As a byproduct,  $k$-polycosymplectic geometry is considered as a particular case of $k$-polysymplectic geometry. Despite the relevance of the latter result, it is worth noting that this is done by showing that a $k$-polycosymplectic structure gives rise to a $k$-polysymplectic structure on a manifold of larger dimension, which may introduce additional complications to study problems in the initial $k$-polycosymplectic manifold and may need to be addressed (see \cite{LMZ22} for a more detailed comment of this fact in the cosymplectic case). In a few words, while $k$-polycosymplectic structures may be equivalent to some $k$-polysymplectic structures on manifolds of higher-dimension, the corresponding Hamiltonian $k$-polycosymplectic systems are not ``geometrically equivalent" to the associated Hamiltonian $k$-polysymplectic ones on manifolds of a higher-dimension.

\subsection{\texorpdfstring{$k$}--Polycosymplectic momentum maps}\label{Sec::with}

Let us develop a $k$-polycosymplectic momentum map notion by extending  the ideas employed to define cosymplectic momentum maps to the $k$-polycosymplectic realm. The definition of a $k$-polycosymplectic $\Ad^{*k}$-equivariant momentum map is also provided. 

Every $k$-polycosymplectic structure on a manifold $M$ gives rise to two closed differential forms $\bm\omega\in\Omega^2(M,\R^k)$ and $\bm\tau\in\Omega^1(M,\R^k)$ given by
\[
\bm\omega=\omega^\alpha\otimes e_\alpha,\qquad \bm\tau=\tau^\alpha\otimes e_\alpha,
\]
for a canonical basis $\{e_1,\ldots,e_k\}$ in $\mathbb{R}^k$ and some differential two- and one-forms on $M$ given by $\omega^\alpha$ and $\tau^\alpha$ for $\alpha=1,\ldots,k$, respectively. The standard differential calculus for differential forms can be naturally extended to differential forms taking values in $\mathbb{R}^k$ \cite{KN96I}. 

\begin{definition}
        A Lie group action $\Phi\colon G\times M\to M$   is said to be a $k$-\textit{polycosymplectic Lie group action} relative to the $k$-polycosymplectic manifold  $(M,\bm\tau,\bm\omega)$ if, for each $g\in G$, the diffeomorphism $\Phi_g:M\rightarrow M$ satisfies $\Phi_g^*\bm\om=\bm\om$ and $\Phi_g^*\bm\tau=\bm\tau$.
\end{definition}

\begin{definition}
\label{Def::kpolycosymMomentumMap}
A {\it $k$-polycosymplectic momentum map} for a Lie group action $\Phi:G\!\x \!M\!\rightarrow\! M$ relative to a $k$-polycosymplectic manifold $(M,\bm\tau,\bm\om)$ such that $\xi_M$ takes values in $\ker{\bm \tau}$ for every $\xi\in \mathfrak{g}$, is a map $\mathbf{J}^\Phi:M\rightarrow (\mathfrak{g}^*)^k$ satisfying  that
\begin{equation}
\label{Eq::polycosymMomentuMap}
\inn_{\xi_M}\bm\omega=\d\left<\mathbf{J}^\Phi,\xi\right>=\d J^\Phi_\xi\,,\quad \inn_{\xi_M}\bm\tau=0\,,\quad \Lie_{R_\alpha} J^\Phi_\xi=0\,,\quad\forall\xi\in\mathfrak{g}\,,\quad \alpha=1,\ldots,k\,.
\end{equation}
\end{definition}
Note that $J_\xi^\Phi$ in the above definition is to be considered, for any fixed $\xi\in \mathfrak{g}$, as a function on $M$ taking values in $\mathbb{R}^k$. According to Definition \ref{Def::PolysymMomentumMap}, we can express the first and the second conditions in \eqref{Eq::polycosymMomentuMap} in the following manner
\[
\iota_{\bm\xi_M}\bm\omega=\d\left\<\mathbf{J}^\Phi,\bm\xi \right\>\quad {\rm and }\quad \iota_{\bm\xi_M}\bm\tau=0,\qquad\forall \bm\xi\in\mathfrak{g}^k,
\]
where $\bm\xi=(0,\ldots,\overset{(\alpha)}{\xi},\ldots,0)\in\mathfrak{g}^k$ for any $\xi\in\mathfrak{g}$ and $\alpha=1,\ldots,k$.
The Reeb vector fields $R_1,\ldots,R_k$ corresponding to $(M,\bm\tau,\bm\om)$ are tangent to the level sets of $\mathbf{J}^\Phi$. However, $R_1,\ldots, R_k$ are not tangent to the orbits of $\Phi$ since ${\bm 
\tau}$ is required to vanish when restricted to the tangent space to the orbits of $\Phi$.

The following $\Ad^{*k}$-equivariance definition is well-known and widely used in the literature. Nevertheless, we will hereafter show that it is not necessary.

\begin{definition}
    A $k$-polycosymplectic momentum map $\mathbf{J}^\Phi:M\rightarrow (\Lg^*)^k$ is {\it $\Ad^{*k}$-equivariant} if it satisfies
    \[
    \mathbf{J}^\Phi\circ \Phi_g=\Ad^{*k}_{g^{-1}}\circ\,\, \mathbf{J}^\Phi,\qquad \forall g\in G\,,
    \]
    where
    \[
    \begin{array}{ccc}
    \Ad^{*k}:G\x(\Lg^*)^k & \longrightarrow & (\Lg^*)^k\\
    (g,{\bm {\mu}}) &\longmapsto & \stackrel{k-\text{times}}{\overbrace{(\Ad^*_{g^{-1}}\otimes \dotsb \otimes \Ad^*_{g^{-1}})}}(\bm \mu). 
    \end{array}\,
    \]
    In other words, for every $g\in G$, the following diagram commutes 
\\
\begin{center}
    \begin{tikzcd}
    M
    \arrow[r,"\mathbf{J}^\Phi"]
    \arrow[d,"\Phi_g"]& (\Lg^*)^k
    \arrow[d,"\Ad^{*k}_{g^{-1}}"]\\
    M
    \arrow[r,"\mathbf{J}^\Phi"]&
    (\Lg^*)^k
    \end{tikzcd}
    \end{center}
\end{definition}

Note that $k$-polycosymplectic Lie group actions are the analogue in $k$-polycosymplectic geometry to $k$-polysymplectic Lie group actions.

Let us introduce the following definition. 

\begin{definition}
The four-tuple $(M^{\bm\omega}_{\bm\tau},h,{\bf J}^\Phi)$ is called a {\it $G$-invariant $k$-polycosymplectic Hamiltonian system} if  it consists of a $k$-polycosymplectic manifold $(M,\bm\tau,\bm\omega)$, a $k$-polycosymplectic Lie group action $\Phi:G\x M\rightarrow M$ such that $\Phi_g^*h=h$ for every $g\in G$, and the $k$-polycosymplectic momentum map $\mathbf{J}^\Phi$ related to $\Phi$. An {\it $\Ad^{*k}$-equivariant $G$-invariant $k$-polycosymplectic Hamiltonian system} is a $G$-invariant $k$-polycosymplectic Hamiltonian whose $k$-polycosymplectic momentum map is $\Ad^{*k}$-equivariant.
\end{definition}

\subsection{General \texorpdfstring{$k$}--polycosymplectic momentum maps}\label{Sec::without}

Let us prove in this section that the $\Ad^{*k}$-equivariance condition for a $k$-polycosymplectic momentum map $\mathbf{J}^\Phi$ can be substituted by a new type of equivariance. More exactly, we show that a $k$-polycosymplectic momentum map $\mathbf{J}^\Phi$ is ${\bm\Delta}$-equivariant relative to a hereafter constructed affine Lie group action on $(\Lg^*)^k$. The proofs of the following results that are essentially the same as in the $k$-polysymplectic case will be hereafter omitted. Additionally, the techniques introduced in this section are analogous to the procedures used in \cite{LMZ22}, where momentum maps on symplectic manifolds were generalised to cosymplectic manifolds, but much more technically involved.

\begin{proposition}
\label{Prop::PsiConstant-2}
Let $(M^{\bm\omega}_{\bm\tau},h,\mathbf{J}^\Phi)$ be a $G$-invariant $k$-polycosymplectic Hamiltonian system.  Define the functions on $M$ of the form
\[
    \bm{\psi}_{g,\boldsymbol{\xi}}=\mathbf{J}^\Phi_{\boldsymbol{\xi}}\circ \Phi_g-\mathbf{J}_{\Ad_{g^{-1}}^k\boldsymbol{\xi}}^\Phi:M\rightarrow \R\,,\quad\forall g\in G\,,\quad\forall \boldsymbol{\xi}\in\mathfrak{g}^k\,.
\]
Then, each $\bm\psi_{g,\boldsymbol{\xi}}$ is constant on $M$ for every $g\in G$ and $\boldsymbol{\xi}\in\Lg^k$. Moreover, the map $\bm\sigma:G\ni g\mapsto \bm\sigma(g)\in(\mathfrak{g}^*)^k$ such that $\<\bm\sigma(g),\bm\xi\>=\bm\psi_{g,\boldsymbol{\xi}}$ satisfies
\[
    \bm\sigma(g_1g_2)=\bm\sigma(g_1)+\Ad^{*k}_{g_1^{-1}} \bm\sigma(g_2)\,,\qquad\forall g_1,g_2\in G\,.
\]
\end{proposition}
The proof of Proposition \ref{Prop::PsiConstant-2} is essentially the same as the proof of Proposition \ref{Prop::PsiConstant} and it will be therefore omitted.
Note that the map ${\bm \sigma}$ in Proposition \ref{Prop::PsiConstant-2} can be brought into the form 
\[
\bm\sigma(g)=\mathbf{J}^\Phi\circ \Phi_g-\Ad^{*k}_{g^{-1}}\mathbf{J}^\Phi=(\sigma^1(g),\ldots,\sigma^k(g))\in (\mathfrak{g}^*)^k\,,
\]
where  $\sigma^\alpha(g)=\mathbf{J}^\Phi_\alpha\circ\Phi_g-\Ad^*_{g^{-1}}\mathbf{J}^\Phi_\alpha$ for each $\alpha=1,\ldots,k$.
The map ${\bm \sigma}$ is called the {\it co-adjoint cocycle associated with $\mathbf{J}^\Phi$}. 
It is worth stressing that ${\bf J}^\Phi$ is an $\Ad^{*k}$-equivariant $k$-polycosymplectic momentum map if, and only if,  $\bm\sigma=0$. Moreover,
 every $k$-polycosymplectic Lie group action admitting a $k$-polycosymplectic momentum map also induces a well-defined cohomology class $[\bm\sigma]$.


Let us introduce an analogue to Proposition \ref{Prop::GenEqJPolySym} to show that a $k$-polycosymplectic momentum map $\mathbf{J}^\Phi$ gives rise to a so-called affine Lie group action $\bm\Delta$ of $G$ on $(\Lg^*)^k$ satisfying $\mathbf{J}^\Phi\circ\Phi_g=\bm\Delta_g\circ\mathbf{J}^\Phi$ for every $g\in G$.

\begin{proposition}\label{Prop::GenEqJPolyco}
Let ${\bf J}^\Phi$ be a momentum map for the $k$-polycosymplectic Lie group action $\Phi$ with an associated co-adjoint cocycle $\bm\sigma$. Then,
\begin{enumerate}[{\rm(1)}]
    \item the map 
    \[
    \bm\Delta :G\times (\Lg^*)^k\ni(g,\bm \mu)\longmapsto \Ad^{*k}_{g^{-1}}\bm\mu+\bm\sigma(g)={\bm\Delta}_g{\bm \mu}\in(\Lg^*)^k\,,
    \]
is a Lie group action of $G$ on $(\Lg^*)^k$,
    \item the $k$-polycosymplectic momentum map ${\bf J}^\Phi$ is equivariant with respect to $
    \bm\Delta$, in other words, every $g\in G$ gives rise to a commutative diagram
\end{enumerate}

\begin{center}
    \begin{tikzcd}
    M
    \arrow[r,"\mathbf{J}^\Phi"]
    \arrow[d,"\Phi_g"]& (\Lg^*)^k
    \arrow[d,"\bm\Delta_g"]\\
    M
    \arrow[r,"\mathbf{J}^\Phi"]&
    (\Lg^*)^k
    \end{tikzcd}
    \end{center}
\end{proposition}

Again, the proof of this proposition is the same as in the $k$-polysymplectic case, i.e. like in Proposition \ref{Prop::GenEqJPolyco}. 

Note that ${\bm \Delta}$ can be rewritten in the following way
\begin{align}
{\bm \Delta}(g,\mu^1,\ldots,\mu^k)&=({\rm Ad}_{g^{-1}}^*(\mu^1)+\sigma^1(g),\ldots,{\rm Ad}_{g^{-1}}^*(\mu^k)+\sigma^k(g))\\&=(\Delta^1(g,\mu^1),\ldots,\Delta^k(g,\mu^k))\in (\mathfrak{g}^*)^k\,,
\end{align}
which gives rise to defining $k$ Lie group actions 
\[\
\Delta^\alpha:(g,\vartheta)\in G\times \mathfrak{g}^*\mapsto {\rm Ad}_{g^{-1}}^*(\vartheta)+\sigma^\alpha(g)\in \mathfrak{g}^*\,,\qquad \alpha=1,\ldots,k\,.
\]

\subsection{\texorpdfstring{$k$}--Polycosymplectic Marsden--Weinstein reduction theorem}
\label{Sec::Polyco&Polysym}

This section presents a $k$-polycosymplectic Marsden--Weinstein reduction procedure by means of a particular type of $k$-polysymplectic Marsden--Weinstein reduction. Similarly to the cosymplectic case in Section \ref{Sec:CosymSym}, one can extend a $k$-polycosymplectic manifold to a $k$-polysymplectic manifold of a special sort: a so-called fibred one. Moreover, the $k$-polycosymplectic momentum map $\mathbf{J}^\Phi: M\rightarrow (\Lg^*)^k$ related to a Lie group action $\Phi: G\x M\rightarrow M$ gives rise to an extended momentum map of an extended Lie group action defined on a  manifold $\R^k\x M$ endowed with a $k$-polysymplectic fibred structure. Then, the $k$-polycosymplectic Marsden--Weinstein reduction boils down to a Marsden--Weinstein reduction, devised in \cite{MRSV15}, of a $k$-polysymplectic fibred structure on $\R^k\x M$.

Theorem \ref{Prop::PolyCosymSymEqui} illustrates  how a $k$-polycosymplectic manifold $(M,{\bm \tau},{\bm \omega})$ gives rise to a $k$-polysymplectic fibred manifold $(\mathbb{R}^k\times M,\widetilde{\bm \omega})$ with the so-called {\it $k$-polysymplectic Reeb vector fields} and vice versa. We will show that the $k$-polysymplectic fibred manifold admits a global symmetry. 
\begin{theorem}
\label{Prop::PolyCosymSymEqui}
    Let $\bm\omega\in\Omega^2(M,\R^k)$, $\bm\tau\in\Omega^1(M,\R^k)$, and let ${\rm pr}_M:\R^k\x M\rightarrow M$ be the canonical projection onto $M$. Let $\bm u=(u^1,\ldots,u^k)$ be a natural global coordinate system in $\R^k$. Then,  $(M,\bm\tau,\bm\omega)$ is a $k$-polycosymplectic manifold if, and only if,  $(\R^k\times M,{\rm pr}_M^*\bm\omega+\d\bm u\barwedge{\rm pr}_M^*\bm\tau=\widetilde{\bm\omega})$ is a $k$-polysymplectic manifold, where $\d\bm u=\d u^\alpha\otimes e_\alpha$, admitting some  vector fields $\widetilde{R}_1,\ldots, \widetilde{R}_k$ on $\R^k\times M$, so-called $k$-polysymplectic Reeb vector fields,  such that $\iota_{\widetilde{R}_\alpha} {\widetilde{\omega}}^\beta=-\delta^\beta_\alpha\d u^\alpha$ and $ \widetilde{R}_\alpha  u^\beta=0$ for $\alpha,\beta=1,\ldots,k$. 
\end{theorem}

\begin{proof}
     First, note that  $\widetilde{\bm\omega}$ decomposes into $k$ components, which, in view of the proof of Lemma \ref{lem:CosymSym}, yields that $\widetilde{\bm\omega}$ is closed if, and only if, $\bm\omega$ and $\bm\tau$ are closed.
    

Second, let us show that if $(M,{\bm \tau},{\bm \omega})$ is a $k$-polycosymplectic manifold, then $\widetilde{\bm\omega}$ is non-degenerate and it possesses $k$-polysymplectic Reeb vector fields.

By Proposition \ref{prop:reeb-polyco}, there exists a family of Reeb vector fields $R_1,\dots, R_k$ on $M$ for $(M,{\bm \tau},{\bm \omega})$ that can be lifted, in a unique manner, to vector fields $\widetilde R_1,\ldots,\widetilde R_k\in\X(\mathbb{R}^k\times M)$ so that $\widetilde R_\beta u^\alpha=0$ for $\alpha,\beta=1,\ldots,k$ and they project onto $R_1,\dots, R_k$ via ${\rm pr}_M$. By construction of $\widetilde{\bm\omega}$, it follows that $\widetilde R_1,\ldots,\widetilde{R}_k$ satisfy $\iota_{\widetilde{R}_\alpha}\widetilde{\omega}^\beta=-\delta_\alpha^{\widehat{\beta}} du^{\widehat{\beta}}$ for $\alpha,\beta=1,\ldots,k$ and they become $k$-polysymplectic Reeb vector fields for $\widetilde{\bm\omega}$.

Assume that $X\in \X(\R^k\x M)$ takes values in $\ker({\rm pr}_M^*\bm\omega + \d \bm u\barwedge{\rm pr}_M^*\bm\tau)$. Then,
\[
\iota_{\partial/\partial u^\alpha}\iota_X\bm{ \widetilde{\omega}}=0 \quad\Longrightarrow\quad \iota_X{\rm pr}^*_M\tau^\alpha=0\,,\qquad \alpha=1,\ldots,k\,.
\]
Hence, $X$ takes values in $\ker {\rm pr}^*_M\boldsymbol\tau$.
 Then,
\[
\iota_{\widetilde{R}_\alpha}\iota_X\bm{ \widetilde{\omega}}=0\quad\Longrightarrow\quad Xu^\alpha=0\,,\qquad \alpha=1,\ldots,k\,.
\]
Therefore,
$$
\iota_X{\rm pr}_M^*\bm{\omega}=0
$$
and $X=0$ since $\ker\d\bm u\cap\ker {\rm pr}^*_M \bm\tau\cap \ker{\rm pr}^*_M \bm \omega=0$. Hence, $\bm{\widetilde\omega}$ is non-degenerate. 

Third, and the other way around, if $(\R^k\x M,\wtl{\bm\omega})$ is a $k$-polysymplectic manifold with $k$-polysymplectic Reeb vector fields and $X$ takes values in $\ker \bm \omega\cap \ker \bm \tau$, one can lift $X$ to a vector field $\widetilde{X}$ on $\R^k\times M$ in a unique way so that ${\rm pr}_{M*}\widetilde{X}=X$ and $\iota_{\widetilde{X}}\d{\bm{u}}=0$. It follows that $\iota_{\widetilde{X}}\widetilde{\bm\omega}=0$ and $\widetilde{X}=0$ since $\widetilde{\bm \omega}$ is assumed to be non-degenerate. Hence, $X=0$, which shows that $\ker\bm\omega\cap\ker\bm\tau=0$.

Let us prove that the $k$-polysymplectic Reeb vector fields $\widetilde{R}_1,\ldots,\widetilde{R}_k$ project onto vector fields on $M$ spanning a distribution of rank $k$ equal to the kernel of $\bm\omega$. 
Note that $\iota_{\widetilde{R}_\alpha}\widetilde{\omega}^\beta=-\delta^\beta_{\widehat{\alpha}}\d u^{\widehat{\alpha}}$ for $\alpha,\beta=1,\ldots,k$. Since $\widetilde{\bm \omega}$ is a $k$-polysymplectic form and invariant relative to the Lie derivatives with respect to $\partial/\partial u^1,\ldots,\partial/\partial u^k$, it follows from the definition $\widetilde{R}_1,\ldots,\widetilde{R}_k$ that $\Lie_{\partial/\partial u^\alpha}\iota_{\widetilde{R}_\beta}\widetilde{\bm \omega}=0$ and then $\iota_{[\partial/\partial u^\alpha,\widetilde{R}_\beta]}\widetilde{\bm \omega}=0$ for every $\alpha,\beta=1,\ldots,k$.  Hence, $\widetilde{R}_1,\ldots,\widetilde{R}_k$ project onto $M$. Then, for every $\alpha,\beta=1,\ldots,k$, one has that
\[
\iota_{\widetilde{R}_\alpha}\widetilde{\omega}^\beta=\iota_{\widetilde{R}_\alpha}{\rm pr}_M^*\omega^\beta+(\iota_{\widetilde{R}_\alpha}\d u^{\widehat{\beta}}){\rm pr}_M^*\tau^{\widehat{\beta}} -( \iota_{\widetilde{R}_\alpha}{\rm pr}_M^*\tau^{\widehat{\beta}})\d u^{\widehat{\beta}}=-\d u^{\widehat{\alpha}}\delta^\beta_{\widehat{\alpha}}\,,
\]
where we stress that there is no sum over the possible values of $\beta$ or $\alpha$ as indicated by the hatted indexes.
Hence, again without summing over $\beta$, 
\begin{equation}\label{eq::con1}
\iota_{\partial/\partial s^{\widehat{\beta}}}\iota_{\widetilde{R}_\alpha}\widetilde{\omega}^{\widehat{\beta}}=-\iota_{\widetilde{R}_\alpha}{\rm pr}_M^*\tau^\beta=-\delta^\alpha_\beta \quad\Longrightarrow\quad \langle \tau^\beta,{\rm pr}_{M*}\widetilde{R}_\alpha\rangle=\delta^\alpha_\beta\,,\qquad\forall \alpha,\beta=1,\ldots,k\,,
\end{equation}
and
$$
\iota_{\widetilde{R}_\alpha}{\rm pr}_M^*\bm{\omega}=0\,.
$$
The condition \eqref{eq::con1} yields that the vector fields ${\rm pr}_{M*}\widetilde{R}_\alpha=R_\alpha$ with $\alpha=1,\ldots,k$, span a distribution on $M$ of rank $k$ taking values in $\ker {\bm \omega}$. The rank of $D=\ker {\bm \omega}$ cannot be larger than $k$. Otherwise, there would exist a non-zero tangent vector $v_x\in \ker {\bm \omega}\cap \ker {\bm \tau}$ since the annihilator of $\langle \tau^1|_D,\ldots,\tau^k|_D\rangle$ in $D$ would be non-zero and we have already proved that $\ker {\bm \omega}\cap \ker {\bm \tau}=0$.

The first, second, and third points proved above give the searched equivalence between a $k$-polycosymplectic manifold $(M,{\bm \tau},{\bm \omega})$ and a $k$-polysymplectic manifold $(\mathbb{R}^k\times M,\widetilde{\bm \omega})$ having $k$-polysymplectic Reeb vector fields.
\end{proof}

\begin{definition}
The $k$-polysymplectic manifolds satisfying the hypotheses of Theorem \ref{Prop::PolyCosymSymEqui} are called \textit{$k$-polysymplectic fibred manifolds}. In particular, $(\R^k\times M,{\rm pr}_M^*\bm\omega+\d\bm u\barwedge{\rm pr}_M^*{\bm\tau}=\widetilde{\bm\omega})$ is called the   $k$-polysymplectic fibred manifold associated with the $k$-polycosymplectic manifold $(M,{\bm \tau},{\bm \omega})$.
\end{definition}

\begin{remark}
The condition on the existence of the $k$-polysymplectic Reeb vector fields $\widetilde R_\alpha$ in Theorem \ref{Prop::PolyCosymSymEqui} is essential to ensure that a $k$-polysymplectic structure on $\R^k\times M$ gives rise to a $k$-polycosymplectic one on $M$. Consider for instance the manifold $M = \R^4$ with natural linear coordinates $\{x,y,w,v\}$ and the closed differential forms in $M$ given by
$$ \tau^1 = \d y\,,\qquad \tau^2 = \d x\,,\qquad \omega^1 = \d x \wedge\d w\,,\qquad \omega^2 = \d y\wedge\d v\,. $$
These give rise to $\bm \tau=\tau^1\otimes e_1+ \tau^2\otimes e_2$ and $\bm\omega = \omega^1\otimes e_1 + \omega^2\otimes e_2$. We can construct the closed differential forms $\widetilde\omega^1,\widetilde\omega^2\in\Omega^2(\R^2\times M)$ given by
$$ \widetilde\omega^1 = \omega^1 + \d u^1\wedge \tau^1 = \d x\wedge\d w + \d u^1\wedge\d y\,, $$
$$ \widetilde\omega^2 = \omega^2 + \d u^2\wedge \tau^2 = \d y\wedge\d v + \d u^2\wedge\d x\,, $$
which in turn give rise to the $\R^2$-valued differential form $\widetilde{\bm\omega} = \widetilde\omega^1\otimes e_1 + \widetilde\omega^2\otimes e_2\in\Omega^2(\R^2\times M,\R^2)$.
Since $\ker\widetilde\omega^1 = \left< \tparder{}{u^2},\tparder{}{v} \right>$ and $\ker\widetilde\omega^2 = \left< \tparder{}{u^1},\tparder{}{w} \right>$, it is clear that 
\[
\ker \widetilde{\boldsymbol{\omega}}=\ker\left(\wtl\omega^1\otimes e_1+\wtl\omega^2\otimes e_2\right)=\ker\widetilde\omega^1\cap\ker\widetilde\omega^2 = 0\,,
\]
and hence $(\R^2\times M,\boldsymbol{\widetilde\omega})$ is a two-polysymplectic manifold. A simple computation shows that $\widetilde{\bm \omega}$ has no two-polysymplectic Reeb vector fields. 
Now, although
$$
\ker\boldsymbol{\tau}\cap\ker\boldsymbol{\omega}=\ker\tau^1\cap\ker\tau^2\cap\ker\omega^1\cap\ker\omega^2 = 0\,, 
$$
the rank of $\ker\omega^1\cap\ker\omega^2$ is not 2, and thus $(\bm\tau,\bm\omega)$ cannot be a two-polycosymplectic structure on $M$. 

\end{remark}
Let $(M^{\bm\omega}_{\bm\tau}, h,{\bf J}^\Phi)$ be a $G$-invariant $k$-polycosymplectic Hamiltonian system. Then, a Lie group action $\Phi:G\x M\rightarrow M$ with an associated $k$-polycosymplectic momentum map $\mathbf{J}^\Phi:M\rightarrow(\mathfrak{g}^*)^k$ extends to $\R^k\x M$ in the following way
\begin{equation}\label{Eq::ExtendedLieActionPoly}
    \wtl \Phi:G\x\R^k\x M\ni(g,\bm{u},x)\longmapsto(\bm{u},\Phi(g,x))\in\R^k\x M\,,
\end{equation}
and a mapping
\begin{equation}\label{Eq::ExtendedMomentumMapPoly}
    \bfJ^{\widetilde{\Phi}}:\R^k\x M\ni(\bm{u},x)\longmapsto\mathbf{J}^\Phi(x)\in (\Lg^*)^k\,,
\end{equation}
for every $(\bm{u},x)\in \R^k\x M$ and every $g\in G$. Note that $\mathbf{ J}^{\widetilde{\Phi}}$ is a $k$-polysymplectic momentum map for $\widetilde{\Phi}$ relative to the $k$-polysymplectic manifold $(\R^k\x M,\widetilde{\bm\omega})$ by extending the $k$-polycosymplectic manifold $(M,\bm\tau,\bm\omega)$. Let us provide Lemmas \ref{Cor::PolyDecomposition} and \ref{Lemm::MomentumMapRegValue} whose proofs are a straightforward consequence of previous facts and the relation ${\rm pr}_{\mathbb{R}^k}\circ\widetilde\Phi_g={\rm pr}_{\mathbb{R}^k}$, where ${\rm pr}_{\mathbb{R}^k}:\mathbb{R}^k\times M\rightarrow \mathbb{R}^k$, is the natural projection onto $\mathbb{R}^k$, for every $g\in G$. 

\begin{lemma}
\label{Cor::PolyDecomposition}
Assume that the hypotheses of Theorem \ref{Th::PolisymplecticReductionJ} remain valid. If $(\R^k\x M,\widetilde{\bm\omega})$ is a $k$-polysymplectic manifold, then 
\[
{\bf J}^{\widetilde{\Phi}-1}(\bm{\mu})\simeq\R^k\times {\rm pr}_M\big({\bf J}^{\widetilde{\Phi} -1}(\bm{\mu})\big)\simeq \R^k\x \mathbf{J}^{\Phi-1}(\bm\mu)\,,\quad
\mathbf{J}^{\widetilde{\Phi}-1}(\bm{\mu})/G^\Delta_{\bm{\mu}}\simeq \R^k\x \left(\mathbf{J}^{\Phi-1}(\bm{\mu})/G^{\bm\Delta}_{\bm{\mu}}\right)
\]
for every $\bm{\mu}\in(\Lg^*)^k$, where the quotients $\mathbf{J}^{\widetilde{\Phi}-1}(\bm{\mu})/G^\Delta_{\bm{\mu}}$  and $\mathbf{J}^{\Phi-1}(\bm{\mu})/G^{\bm\Delta}_{\bm{\mu}}$ are relative to the actions of $G^{\bm\Delta}_{\bm{\mu}}$ on $\mathbf{J}^{\widetilde{\Phi}-1}(\bm{\mu})$ and $\mathbf{J}^{\Phi-1}(\bm{\mu})$, respectively.
\end{lemma}

\begin{lemma}
\label{Lemm::MomentumMapRegValue}
A $k$-polycosymplectic momentum map $\mathbf{J}^\Phi:M\rightarrow (\mathfrak{g}^*)^k$ is $\bm\Delta$-equivariant with respect to a Lie group action $\Phi:G\times M\rightarrow M$ if and only if the associated $k$-polysymplectic momentum map $\mathbf{J}^{\widetilde\Phi}:\mathbb{R}^k\times M\rightarrow (\mathfrak{g}^*)^k$ is $\bm\Delta$-equivariant relative to $\widetilde{\Phi}:G\times \mathbb{R}^k\times M\rightarrow \mathbb{R}^k\times M$. Additionally, $\bm{\mu}\in(\Lg^*)^k$ is a (resp. weak) regular value of $\mathbf{J}^\Phi$ if and only if $\bm{\mu}$ is a (resp. weak) regular value of $\mathbf{J}^{\widetilde \Phi}$. Moreover, $\mathbf{J}^{\widetilde \Phi-1}(\bm{\mu})=\R^k\times \mathbf{J}^{\Phi-1}(\bm{\mu})$ and $\mathbf{J}^{\widetilde \Phi-1}(\bm\mu)$ is quotientable by $G^{\bm \Delta}_{\bm \mu}$ if and only if $\mathbf{J}^{\Phi-1}(\bm{\mu})$ is so. 
\end{lemma}
The $k$-polysymplectic Marsden--Weinstein reduction Theorem \ref{Th::PolisymplecticReductionJ} provides the conditions  \eqref{Eq::PolysymplecticReduction1eq} and \eqref{Eq::PolysymplecticReduction2eq} to ensure the existence of a $k$-polysymplectic structure on  $\mathbf{J}^{\widetilde{\Phi}-1}(\bm\mu)/G^{\bm\Delta}_{\bm{\mu}}$. Note that Lemmas \ref{Lemm::MomentumMapRegValue}, \ref{Lem::Transcrip}, and \ref{Cor::PolyDecomposition} allow us to see that the $k$-polysymplectic fibred manifold induced by a $k$-polycosymplectic one has some required conditions to perform a $k$-polysymplectic reduction. Moreover, the $k$-polysymplectic Reeb vector fields on $\mathbb{R}^k\times M$ are tangent to ${\bf J}^{\widetilde{\Phi} -1}(\bm \mu)$, because $R_1,\ldots, R_k$ are tangent to ${\bf J}^{\Phi-1}(\bm \mu)$ and project onto the quotient manifold as they are invariant relative to the action of $\widetilde{\Phi}$. Note that Theorem 
\ref{Th::kpolycoreduction} analyses the previous result and ensures that  the reduced $k$-polysymplectic form, which is defined on a manifold of the form $\mathbf{J}^{\widetilde{\Phi}-1}(\bm\mu)/G^{\bm\Delta}_{\bm{\mu}}\simeq\mathbb{R}^k\times M^{\bm\Delta}_{\bm \mu}$, is fibred and so related to a $k$-polycosymplectic form on $M^{\bm\Delta}_{\bm \mu}$. This finishes the development of a $k$-polycosymplectic reduction manifold theory. The reduction of Hamiltonian systems will be dealt with after this.

\begin{lemma}\label{Lem::Transcrip}Let $(M,\bm{\tau},\bm{\omega})$ be a $k$-polycosymplectic manifold and let $(\R^k\times M,\widetilde{\bm{\omega}})$ be its associated $k$-polysymplectic fibred manifold. If 
\begin{equation}\label{Eq::PolycosymplecticRedEq1}
    \T_x(G^\Delta_{\bm{\mu}} x) = \bigcap^k_{\alpha=1}\left( (\ker\omega_x^\alpha\cap\ker\tau_x^\alpha) + \T_x (G_{\mu^\alpha}^{\Delta^\alpha
} x)\right)\cap  \T_x\mathbf{J}^{\Phi-1}(\bm\mu)
\end{equation}
and
\begin{equation}
\label{Eq::PolycosymplecticRedEq2}
    \ker \T_x\mathbf{J}^{\Phi}_\alpha = \ker\omega_x^\alpha\cap\ker\tau_x^\alpha + \T_x\mathbf{J}^{\Phi-1}(\bm\mu) + \T_x(G_{\mu^\alpha}^{\Delta^\alpha} x)
\end{equation}
for every $x\in {\bf J}^{\Phi-1}({\bf \mu})$, then expressions \eqref{Eq::PolysymplecticReduction1eq} and \eqref{Eq::PolysymplecticReduction2eq} concerning the extensions ${\bf J}^{\widetilde{\Phi}}$ and $\widetilde{\Phi}$ to $\R^k\times M$ of ${\bf J}^\Phi$ and $\Phi$, namely
\begin{equation}\label{Eq::PolysymplecticReduction1eq2}
        \ker (\T_p\mathbf{J}_\alpha^{\widetilde{\Phi}}) = \T_p(\mathbf{J}^{\widetilde{\Phi}-1}(\bm\mu))+\ker\widetilde{\om}^\alpha_p + \T_p(G^{\Delta^\alpha}_{\mu^\alpha} p)
    \end{equation}
and \begin{equation}\label{Eq::PolysymplecticReduction2eq2}
        \T_p(G_{\bm\mu}^{\bm\Delta} p)=\bigcap^k_{\alpha=1}\left(\ker\widetilde{\om}^\alpha_p+\T_p(G^{\Delta^\alpha}_{\mu^\alpha}p)\right)\cap \T_p(\mathbf{J}^{\widetilde{\Phi}-1}(\bm\mu))\,,
    \end{equation}
    for every $p=({\bf u},x)\in {\bf J}^{\widetilde{\Phi}-1}({\bm \mu})$ and all $\alpha=1,\ldots,k$, are satisfied.
\end{lemma}
\begin{proof}
Given the canonical projection ${\rm pr}_M:\R^k\x M\rightarrow M$ and the natural isomorphisms $\T_{(\bm{u},x)} (\mathbb{R}^k\times M)\simeq \T_{\bm u}\mathbb{R}^k\oplus \T_xM$ for every $({\bm u},x)\in \mathbb{R}^k\times M$, it follows that, for $\alpha=1,\ldots,k$, one has that
\begin{equation}\label{Eq::PolyCospaces1}
\begin{gathered}
    (\ker{\rm pr}_M^*\omega^\alpha)_{(\bm{u},x)} = \T_{\bm{u}}\mathbb{R}^k\oplus \ker\omega^\alpha_x\,,\qquad
    (\ker{\rm pr}_M^*\tau^\alpha)_{(\bm{u},x)} = \T_{\bm{u}}\mathbb{R}^k\oplus \ker\tau^\alpha_x\,,\\
    (\ker \d u^\alpha)_{(\bm{u},x)} = A_{\bm{u}}^\alpha\oplus \T_xM\,,
    \end{gathered}
\end{equation}
where $A_{\bm{u}}^\alpha = \T_{\bm{u}}\mathbb{R}^k\cap (\ker \d u^\alpha)_{(\bm{u},x)}$ and $\ker\bm\omega_x,\ker\bm\tau_x\subset \T_xM$  for every $({\bf u},x)\in \mathbb{R}^k\times M$. The contraction of $\widetilde{\omega}^\alpha$ with $\partial/\partial u^\alpha$ and the extended Reeb vector fields $\widetilde{R}_1,\ldots,\widetilde{R}_k$ on $\mathbb{R}^k\times M$ give, along with \eqref{Eq::PolyCospaces1}, that
\begin{multline}
\label{Eq::keromegatilde}
(\ker \widetilde{\omega}^\alpha)_{(\bm{u},x)}=(\ker{\rm pr}^*_M\omega^\alpha\cap \ker \d u^\alpha\cap\ker{\rm pr}^*_M\tau^\alpha)_{(\bm{u},x)}\\=\left(\T_{\bm{u}}\R^k\oplus \ker\omega^\alpha_x \right)\cap\left( A^\alpha_{\bm{u}}\oplus \T_xM\right)\cap\left( \T_{\bm{u}}\R^k\oplus \ker\tau^\alpha_x \right)=A^\alpha_{\bm{u}}\oplus\left(\ker\omega^\alpha_x\cap\ker\tau^\alpha_x\right),
\end{multline}
for every $({\bm u},x)\in \mathbb{R}^k\times M$.
Moreover, it follows from the extension formulas \eqref{Eq::ExtendedLieActionPoly} and \eqref{Eq::ExtendedMomentumMapPoly} that $\Lie_{\frac{\partial}{\partial u^\alpha}}\mathbf{J}^{\widetilde{\Phi}}=0$ and $\inn_{\xi_{\R^k\x M}}\d \bm u=0$ hold for every $\xi\in\Lg$ and $\alpha=1,\ldots,k$. Hence, 
\begin{equation}\label{Eq::PolyCospaces2}
\begin{gathered}
    \T_{({\bm{u}},x)}(G_{\bm{\mu}}^{\bm \Delta} ({\bm{u}},x)) = \T_x(G^{\bm \Delta}_{\bm{\mu}} x)\,,\qquad\T_{({\bm{u}},x)}\mathbf{J}^{\widetilde{\Phi}-1}({\bm{\mu}}) = \T_{\bm{u}}\R^k\oplus\T_x (\mathbf{J}^{\Phi-1}({\bm{\mu}}))\,,\,\\
    \T_{({\bm{u}},x)}(G_{\mu^\alpha}^{\Delta^\alpha}({\bm{u}},x)) = \T_x(G^{\Delta^\alpha}_{\mu^\alpha}x)\,,\qquad
\ker(\T_{({\bm{u}},x)}\mathbf{J}^{\widetilde{\Phi}}_\alpha) = \T_{\bm{u}}\R^k\oplus \ker \T_x\mathbf{J}^\Phi_\alpha\,,
\end{gathered}
\end{equation}
for $\alpha=1,\ldots,k$ and arbitrary ${\bm u}\in \mathbb{R}^k$ and points $x\in {\bf J}^{\Phi^{-1}}({\bm \mu})$. Then, condition \eqref{Eq::PolysymplecticReduction2eq} for our $k$-polysymplectic fibred manifold can be written as follows
\begin{multline}
\label{Eq::PolyRedEq1}
     \T_x(G^{\bm \Delta}_{\bm{\mu}} x) = \T_{({\bm{u}},x)}\left(G_{\bm{\mu}}^{\bm \Delta} ({\bm{u}},x) \right) = \bigcap^k_{\alpha=1}\left( (\ker\widetilde{\omega}^\alpha)_{({\bm{u}},x)} + \T_x\left(G_{\mu^\alpha}^{\Delta^\alpha}x\right)\right)\cap \T_{({\bm{u}},x)}\mathbf{J}^{\widetilde{\Phi}-1}({\bm{\mu}})\\
    =\bigcap^k_{\alpha=1}\left(A^\alpha_{\bm{u}}\oplus\left(\ker\omega_x^\alpha\cap\ker\tau_x^\alpha\right) + \T_x (G_{\mu^\alpha}^{\Delta^\alpha} x)\right)\cap \left(\T_{\bm{u}}\R^k\oplus \T_x\mathbf{J}^{\Phi-1}(\bm{\mu})\right)\\
    = \bigcap^k_{\alpha=1}\left[ (\ker\omega_x^\alpha\oplus\ker\tau_x^\alpha) + \T_x (G_{\mu^\alpha}^{\Delta^\alpha} x)\right]\cap  \T_x\mathbf{J}^{\Phi-1}({\bm{\mu}})\,,
\end{multline}
and \eqref{Eq::PolysymplecticReduction1eq} amounts to
\begin{multline}
\label{Eq::PolyRedEq2}
    \T_{\bm{u}}\R^k\oplus\ker \T_x\mathbf{J}^{\Phi}_\alpha=\ker \T_{({\bm{u}},x)}\mathbf{J}^{\widetilde{\Phi}}_\alpha = \T_{({\bm{u}},x)}(\mathbf{J}^{\widetilde{\Phi}-1}(\bm\mu)) + \ker\widetilde{\omega}_{({\bm{u}},x)}^\alpha + \T_{({\bm{u}},x)}\left( G^{\Delta^\alpha}_{\mu^\alpha}({\bm{u}},x)\right)\\
    = \T_{\bm{u}}\R^k\oplus \T_x\mathbf{J}^{\Phi-1}(\bm\mu) + A^\alpha_{\bm{u}}\oplus\left(\ker\omega_x^\alpha\cap\ker\tau_x^\alpha\right) +  \T_x\left(G_{\mu^\alpha}^{\Delta^\alpha} x\right)\\
=\T_{\bm{u}}\R^k\oplus\left(\ker\omega_x^\alpha\cap\ker\tau_x^\alpha + \T_x\mathbf{J}^{\Phi-1}(\bm\mu) + \T_x(G_{\mu^\alpha}^{\Delta^\alpha} x) \right)\,,
\end{multline}
where we have used \eqref{Eq::PolyCospaces1}, \eqref{Eq::keromegatilde}, and \eqref{Eq::PolyCospaces2}, for every $({\bm{u}},x)\in \R^k\x {\bf J}^{\Phi-1}({\bm \mu})$, every $\bm{\mu}\in(\Lg^*)^k$, and $\alpha=1,\ldots,k$.
Thus, \eqref{Eq::PolyRedEq1} and \eqref{Eq::PolyRedEq2} amount to the conditions \eqref{Eq::PolycosymplecticRedEq1} and \eqref{Eq::PolycosymplecticRedEq2}. This finishes the proof. 
\end{proof}

Note that, in view of Lemma \ref{Lem::Transcrip}, if \eqref{Eq::PolycosymplecticRedEq1} and \eqref{Eq::PolycosymplecticRedEq2} hold, Theorem \ref{Th::PolisymplecticReductionJ} yields that a $k$-polycosymplectic manifold $(M,{\bm \tau},{\bm \omega})$ ensures that its associated $k$-polysymplectic fibred manifold $(\mathbb{R}^k\times M,\widetilde{\bm \omega})$ gives rise to an induced reduced $k$-polysymplectic manifold $(\mathbf{J}^{\widetilde{\Phi}-1}(\bm\mu)/G^{\bm\Delta}_{\bm\mu},\widetilde{\bm\omega}_{\bm\mu})$. It is left to prove that this latter one amounts via Theorem \ref{Prop::PolyCosymSymEqui} to a $k$-polycosymplectic  Marsden--Weinstein reduction $(\mathbf{J}^{\Phi-1}(\bm\mu)/G^{\bm\Delta}_{\bm\mu},\bm\tau_{\bm\mu},\bm\omega_{\bm\mu})$ of our initial $k$-polycosymplectic manifold $(M,{\bm \tau},{\bm \omega})$.

\begin{theorem}
\label{Th::kpolycoreduction}
Let $(M,\bm\tau,\bm\om)$ be a $k$-polycosymplectic manifold and let ${\bf J}^{\Phi}:M\rightarrow (\mathfrak{g}^{*})^k$ be a $k$-polycosymplectic momentum map associated with a $k$-polycosymplectic Lie group action $\Phi:G\times M\rightarrow M$. Let $\bm\mu\in (\Lg^*)^k$ be a weak regular value of $\mathbf{J}^\Phi$ and let $\mathbf{J}^{\Phi-1}(\bm \mu)$ be quotientable by $G^{\bm\Delta}_{\bm\mu}$. Moreover, assume that
    
    \begin{equation}\label{Eq::PolycosymplecticReductionTh1}
        \T_x(G_{\boldsymbol{\mu}}^{\boldsymbol{\Delta}} x)=\bigcap^k_{\alpha=1}\left(\ker\om^\alpha_x\cap\ker\tau^\alpha_x+\T_x(G^{\Delta_\alpha}_{\mu^\alpha}x)\right)\cap \T_x\mathbf{J}^{\Phi-1}(\boldsymbol{\mu})\,,
    \end{equation}
    and 
    \begin{equation}
\label{Eq::PolycosymplecticReductionTh2}
    \ker \T_x\mathbf{J}^{\Phi}_\alpha = \ker\omega_x^\alpha\cap\ker\tau_x^\alpha + \T_x\mathbf{J}^{\Phi-1}(\boldsymbol{\mu}) + \T_x(G_{\mu^\alpha}^{\Delta_\alpha} x)\,,
\end{equation}
    for every $x\in {\bf J}^{\Phi-1}(\bm \mu)$ and  $\alpha=1,\ldots,k$. Then, $(\mathbf{J}^{\Phi-1}(\bm\mu)/G^{\boldsymbol{\Delta}}_{\bm\mu},\bm\tau_{\bm\mu},\bm\omega_{\bm\mu})$ is a $k$-polycosymplectic manifold such that ${\bm \tau_{\bm\mu}}$ and ${\bm \omega_{\bm\mu}}$ are defined univocally by
    \[
        \pi_{\bm\mu}^*\bm\tau_{\bm \mu}=\jmath_{\bm\mu}^*\bm\tau\,,\qquad \pi^*_{\bm \mu}\bm\om_{\bm\mu}=\jmath_{\bm\mu}^*\bm\om\,,
    \]
    where $\jmath_{\bm\mu}:\B J^{\Phi-1}(\bm\mu)\hookrightarrow M $ is the natural immersion and $\pi_{\bm\mu}:\mathbf{J}^{\Phi-1}(\bm\mu)\rightarrow \mathbf{J}^{\Phi-1}(\bm\mu)/G^{\bm\Delta}_{\bm\mu}$ is the canonical projection.
\end{theorem}
\begin{proof} 
Consider, as standard, the $k$-polysymplectic manifold $(\mathbb{R}^k\times M,\widetilde{\bm\omega})$, the extended action $\widetilde{\Phi}:G\times \R^k\times M\rightarrow \R^k\x M$ and its associated $k$-polysymplectic momentum map ${\bf{J}}^{\widetilde{\Phi}}:\R^k\times M\rightarrow (\mathfrak{g}^*)^k$. We write $\{u^1,\ldots,u^k\}$ for a standard coordinate system on $\mathbb{R}^k$ that gives rise, in the standard way, to $k$ coordinates on  $\R^k\times M$ that will be denoted in the same manner. According to Lemma \ref{Lemm::MomentumMapRegValue}, if ${\bm \mu}$ is a weak regular value for $\mathbf{J}^{\widetilde{\Phi}}$, then ${\bm \mu}$  is also a weak regular value for $\mathbf{J}^\Phi$. Additionally, ${\bf{ J}}^{\widetilde{\Phi}-1}(\bm \mu)$ is quotientable by the restriction of $\widetilde{\Phi}$ to $G^{\bm\Delta}_{\bm\mu}$ if and only if ${\bf J}^{\Phi-1}(\bm \mu)$ is so relative to the restriction of the Lie group action $\Phi$ to $G^{\bm \Delta}_{\bm \mu}$.

Lemma \ref{Lem::Transcrip} ensures that the conditions \eqref{Eq::PolycosymplecticReductionTh1} and \eqref{Eq::PolycosymplecticReductionTh2} imply that the conditions \eqref{Eq::PolysymplecticReduction1eq}  and \eqref{Eq::PolysymplecticReduction2eq} for the $k$-polysymplectic Marsden--Weinstein reduction on ${\bf J}^{\widetilde{\Phi}-1}({\bm \mu})$ hold. Hence, it is possible to accomplish a $k$-polysymplectic Marsden--Weinstein reduction on $\R^k\x {\bf J}^{\Phi-1}(\bm \mu)$. Let us now prove that the resulting  reduced $k$-polysymplectic manifold is a fibred one.

The $k$-polysymplectic manifold $(\R^k\times M,\widetilde{\bm \omega})$ admits, by the assumptions of the present theorem and Theorem \ref{Prop::PolyCosymSymEqui}, some $k$-polysymplectic Reeb vector fields $\widetilde R_1,\ldots, \widetilde{R}_k$ that are tangent to ${\bf J}^{\widetilde{\Phi}-1}(\bm \mu)$ because the Reeb vector fields $R_1,\ldots, R_k$ for $(M,\bm\tau,\bm\omega)$ are tangent to $\mathbf{J}^{\Phi-1}(\bm\mu)$ and ${\rm pr}_{M*}\widetilde{R}_\alpha=R_\alpha$ for $\alpha=1,\ldots,k$. The Reeb vector fields $R_1,\ldots,R_k$ are invariant under the action of $G^{\bm\Delta}_{\bm \mu}$ via $\Phi$. Therefore, the extensions $\widetilde{R}_1,\ldots,\widetilde{R}_k$, which also satisfy $\widetilde{R}_\alpha u^\beta=0$ for $\alpha,\beta=1,\ldots,k$, are also invariant relative to the action of $G^{\bm\Delta}_{\bm \mu}$ via $\widetilde{\Phi}$. Then, the projections of the restrictions of $\widetilde {R}_1,\ldots,\widetilde{R}_k$ to $\mathbf{J}^{\Phi-1}({\bm \mu})$ onto $\mathbf{J}^{\Phi-1}(\bm\mu)/G^{\bm\Delta}_{\bm\mu}$ are
$k$-polysymplectic Reeb vector fields $\widetilde {R}_{1\bm\mu},\ldots,\widetilde{R}_{k\bm\mu}$ on the reduced $k$-polysymplectic manifold $(\mathbf{J}^{\widetilde{\Phi}-1}(\bm\mu)/G^{\bm\Delta}_{\bm \mu},\widetilde{\bm \omega}_{\bm \mu})$.

Furthermore, the vector fields $\partial/\partial u^1,\ldots,\partial/\partial u^k$ project onto $\mathbf{J}^{\widetilde{\Phi}-1}({\bm \mu})/G_{\bm \mu}^{\bm \Delta}$, which is diffeomorphic to $ \R^k\x( \mathbf{J}^{\Phi-1}(\bm\mu)/G^{\bm\Delta}_{\bm\mu})$ by Lemma \ref{Cor::PolyDecomposition}, and their projections are linearly independent. In fact, the contractions $\iota_{\partial/\partial u^\beta}\iota_{\widetilde{R}_\alpha}\widetilde{\bm \omega}$ are projectable from $\mathbf{J}^{\widetilde{\Phi}-1}(\bm\mu)$ onto  $\mathbf{J}^{\widetilde{\Phi}-1}({\bm \mu})/G_{\bm \mu}^{\bm \Delta}$ and proportional, up to a non-zero constant, to $\delta^\beta_\alpha$.

Now, let us show how the reduced $k$-polysymplectic manifold $\left(\widetilde{M^{\bm \Delta}_{\bm \mu}}=\mathbf{J}^{\widetilde{\Phi}-1}(\bm\mu)/G^{\bm\Delta}_{\bm \mu},\widetilde{\bm \omega}_{\bm \mu}\right)$ gives rise to a $k$-polycosymplectic structure on $M^{\bm \Delta}_{\bm \mu}={\bf J}^{\Phi-1}({\bm \mu})/G^{\bm\Delta}_{\bm \mu}$. Consider the embedding $\inn_{\bm u}:x\in \mathbf{J}^{\Phi-1}(\bm{\mu})\ni x\mapsto (\bm{u},x)\in\R^k\x\mathbf{J}^{\Phi-1}(\bm{\mu})$ for any $\bm{u}\in \mathbb{R}^k$. Using Lemma \ref{Cor::PolyDecomposition}, we can define a reduced $k$-polycosymplectic structure on $M^{\bm \Delta}_{\bm \mu}$ via  $\widetilde{\bm\omega}_{\bm\mu}$ as follows
\begin{equation}
\label{Eq::retrievedpolycosymstructure}
{\bm \omega}_{\bm \mu}=\inn_{\bm u}^*{\widetilde{\bm \omega}}_{\bm \mu}\,,\quad \bm{\tau}_{\bm \mu}=\sum_{\alpha=1}^k\inn_{\bm u}^*\left(\inn_{\frac{\partial}{\partial u^\alpha}}{\widetilde{\bm \omega}}_{\bm \mu}\right)\,.
\end{equation}
Since $\widetilde{\bm\omega}_{\bm\mu}$ is closed and $\Lie_{\partial/\partial u^\alpha}\widetilde{\bm\omega}_{\bm \mu}=0$ for $\alpha=1,\ldots,k$, 
it follows that $\bm\omega_\mu$ and $\bm\tau_\mu$ are closed forms. Let ${\rm pr}_{M^{\bm\Delta}_{\bm \mu}}:\R^k\times M^{\bm\Delta}_{\bm \mu}\mapsto M^{\bm\Delta}_{\bm \mu}$ and $\widetilde{\pi}_{\bm \mu}:{\bf J}^{\widetilde{\Phi}-1}({\bm \mu})\rightarrow \widetilde{M^{\bm \Delta}_{\bm \mu}}$ be the canonical projections.
Then, the reduced $k$-polysymplectic form can be expressed as $\widetilde{\bm\omega}_{\bm \mu}={\rm pr}_{M^{\bm\Delta}_{\bm \mu}}^*{\bm\omega}_{\bm\mu}+\d{\bm u}\barwedge {\rm pr}_{M^{\bm\Delta}_{\bm \mu}}^*{\bm\tau}_{\bm \mu}$. Indeed, this expression satisfies previous relations with $\bm \omega_{\bm \mu}$ and ${\bm \tau}_{\bm \mu}$, and, more importantly, 
\begin{equation}\label{eq:relll}
\widetilde{\pi}_{\bm \mu}^*({\rm pr}_{M^{\bm\Delta}_{\bm \mu}}^*{\bm\omega}_{\bm\mu}+\d{\bm u}\barwedge {\rm pr}_{M^{\bm\Delta}_{\bm \mu}}^*{\bm\tau}_{\bm\mu})=\widetilde{\jmath}^*_{\bm \mu}\widetilde{\bm \omega}\,,
\end{equation}
 which determines univocally the $k$-polycosymplectic reduced structure on $M^{\bm \Delta}_{\bm \mu}$. To prove \eqref{eq:relll}, note that both sides vanish on pairs of tangent vectors belonging to $\T_{\bm u}\mathbb{R}^k$ understood as a subspace of $\T_{({\bm u},x)}{\bf J}^{\widetilde{\Phi}-1}({\bm \mu})\simeq \T_{\bm u}\mathbb{R}^k\oplus \T_{x}({\bf J}^{\Phi-1}({\bm \mu}))$. Moreover, due to the first expression in \eqref{Eq::retrievedpolycosymstructure} both sides of equality \eqref{eq:relll} take the same values on pairs of tangent vectors of the space $\T_x{\bf J}^{\Phi-1}(\bm \mu)$. Finally, given two tangent vectors belonging to $\T_{\bm u}\mathbb{R}^k$ and $\T_x{\bf J}^{\Phi-1}(\bm \mu)$, respectively, a short calculation shows that both sides also match, which, along with previous facts, yield that \eqref{eq:relll} holds. 

Since $(\R^k\x{\bf J}^{\Phi-1}(\bm\mu)/G^{\bm\Delta}_{\bm\mu},{\rm pr}_{M^{\bm\Delta}_{\bm \mu}}^*{\bm\omega}_{\bm\mu}+\d{\bm u}\barwedge {\rm pr}_{M^{\bm\Delta}_{\bm \mu}}^*{\bm\tau}_{\bm\mu})$ is a $k$-polysymplectic manifold by Theorem \ref{Th::PolisymplecticReductionJ} and it admits $k$-polysymplectic Reeb vector fields, then   $(\widetilde{M}^{\bm \Delta}_{\bm \mu},\widetilde{\bm \omega}_{\bm \mu})$ is a $k$-polysymplectic fibred manifold and Theorem \ref{Prop::PolyCosymSymEqui} gives that $(M^{\bf \Delta}_{\bm \mu},{\bm \tau}_{\bm \mu},{\bm \omega}_{\bm \mu})$ is a $k$-polycosymplectic manifold.

Let us prove that
\begin{equation}\label{eq::Rel41}
\jmath_{\bm\mu}^*\bm\omega=\pi_{\bm\mu}^*\bm\omega_{\bm\mu}\,,\qquad \jmath_{\bm\mu}^*\bm\tau=\pi_{\bm\mu}^*\bm\tau_{\bm\mu}\,.
\end{equation}
It stems from \eqref{eq:relll} that
\begin{equation*}
 \widetilde{\pi}_{\bm\mu}^*\left({\rm pr}^*_{M^{\bm\Delta}_{\bm\mu}} {\bm\omega}_{\bm\mu}+\d {\bm{u}} \barwedge {\rm pr}_{M_{\bm\mu}^\Delta}^* {\bm\tau}_{\bm\mu} \right)=\widetilde{\jmath}_{\bm\mu}^*\left({\rm pr}_M^* \bm\omega+\d {\bm u} \barwedge  {\rm pr}_M^*\bm\tau \right), 
 \end{equation*}
 which amounts to
 \begin{equation}\label{eq:Fuck}
 ({\rm pr}_{M^{\bm\Delta}_{\bm\mu}}\circ \widetilde{\pi}_{\bm\mu})^* \bm\omega_{\bm \mu} +\d \bm u\barwedge ({\rm pr}_{M^{\bm\Delta}_{\bm\mu}}\circ \widetilde{\pi}_{\bm\mu})^*{\bm \tau}_{\bm \mu} = ({\rm pr}_{M}\circ \widetilde{\jmath}_{\bm\mu})^* \bm\omega  +\d \bm u\barwedge ({\rm pr}_{M}\circ \widetilde{\jmath}_{\bm\mu})^*\bm \tau .
\end{equation}
Composing on both sides of the last equality by $\iota_{\bm u}^*$, one gets
$$
({\rm pr}_{M^{\bm\Delta}_{\bm\mu}}\circ \widetilde{\pi}_{\bm\mu}\circ \iota_{\bm u})^* \bm\omega_{\bm \mu} =({\rm pr}_{M}\circ \widetilde{\jmath}_{\bm\mu}\circ\iota_{\bm u})^* \bm\omega  
$$
and since $\jmath_{\bm \mu}={\rm pr}_{M}\circ \widetilde{\jmath}_{\bm\mu}\circ\iota_{\bm u}$ and
${\rm pr}_{M^{\bm\Delta}_{\bm\mu}}\circ \widetilde{\pi}_{\bm\mu}\circ \iota_{\bm u}=\pi_{\bm\mu}$, one has that
$$
\pi_{\bm\mu}^* \bm\omega_{\bm \mu}  =\jmath_{\bm\mu}^* \bm\omega  ,
$$
which proves the first equality in \eqref{eq::Rel41}. Composing \eqref{eq:Fuck} with $\partial/\partial u^1,\ldots,\partial/\partial u^k$ and repeating the above procedure, we get the second equality in \eqref{eq::Rel41}.
 \end{proof}

\begin{theorem}
\label{Th::PolyCoReductionDynamics}
Let the assumptions of Theorem \ref{Th::PolisymplecticReductionJ} be satisfied. Let $\bfX^h=(X_{1}^h,\ldots,X_{k}^h)$ be a $k$-polycosymplectic Hamiltonian $k$-vector field associated with a $G$-invariant function $h\in \Cinfty(M)$ relative to the Lie group action $\Phi$. Assume that $\Phi_{g*}\bfX^h=\bfX^h$ for every $g\in G$ and  $\bfX_h$ is tangent to $\mathbf{J}^{\Phi-1}(\boldsymbol{\mu})$. Then, for every $\alpha=1,\ldots,k$, the flow $F^\alpha_s$ of $X_{\alpha}^h$ leave $\mathbf{J}^{\Phi-1}(\boldsymbol{\mu})$ invariant and induces a unique flow $K^\alpha_s$ on $\mathbf{J}^{\Phi-1}(\boldsymbol{\mu})/G_{\boldsymbol{\mu}}^{\boldsymbol{\Delta}}$ satisfying $\pi_{\boldsymbol{\mu}}\circ F^\alpha_s=K^\alpha_s\circ \pi_{\boldsymbol{\mu}}$.
\end{theorem}

\begin{proof}
    Given the assumption that a $k$-polycosymplectic Hamiltonian $k$-vector field $\bfX^h$ is tangent to $\mathbf{J}^{\Phi-1}(\boldsymbol{\mu})$, it follows that each integral curve $F_s^\alpha$ of $X^h_\alpha$ with initial condition within $\mathbf{J}^{\Phi-1}(\boldsymbol{\mu})$ is contained in $\mathbf{J}^{\Phi-1}(\boldsymbol{\mu})$ for all the values of its parameter $s \in \mathbb{R}$ and $\alpha=1,\ldots,k$. Since $\Phi_{g*}\bfX^h=\bfX^h$ for every $g\in G$, one has that $\Lie_{\xi_M}X^h_\alpha=0$ for $\alpha = 1, \ldots, k$. This yields a reduced $k$-vector field $\mathbf{Y}=(Y_1,\ldots, Y_k)$ defined on the quotient manifold $\mathbf{J}^{\Phi-1}(\boldsymbol{\mu})/G^{\boldsymbol{\Delta}}_{\boldsymbol{\mu}}$, such that $\pi_{\bm\mu}\circ F^\alpha_s=K^\alpha_s\circ \pi_{\bm \mu}$, where $K_s^\alpha$ is the flow of $Y_\alpha$, for $\alpha=1,\ldots,k$. Furthermore, the $G$-invariance of $h\in\Cinfty(M)$ yields that there exists a reduced Hamiltonian function $h_{\boldsymbol{\mu}}\in\Cinfty(\mathbf{J}^{\Phi-1}(\boldsymbol{\mu})/G^{\boldsymbol{\Delta}}_{\boldsymbol{\mu}})$ such that $\pi_{\boldsymbol{\mu}}^*h_{\boldsymbol{\mu}} = \jmath_{\boldsymbol{\mu}}^*h$.
    
    Next, let us verify that $\bfY$ is a reduced $k$-polycosymplectic Hamiltonian $k$-vector field associated with $h_{\bm\mu}$. Indeed, the Reeb vector fields $R_1,\ldots, R_k$ are tangent to $\mathbf{J}^{\Phi-1}(\boldsymbol{\mu})$ and give rise to linearly independent vector fields $R^{\mathbf{J}^\Phi}_1,\ldots,R^{\mathbf{J}^\Phi}_k$ on ${\bf J}^{\Phi-1}({\bm \mu})$. Due to this fact and Theorem \ref{Th::kpolycoreduction}, it follows that
    \begin{multline*}
\d\pi^*_{\bm\mu}h_{\bm\mu}=\d\jmath_{\bm\mu}^*h=\jmath^*_{\bm\mu}(\inn_{\mathbf{X}^h}\bm\omega+(R_\alpha h)\tau^\alpha)=\inn_{\mathbf{X}^h}\jmath^*_{\bm\mu}\bm\omega+(R^{\mathbf{J}^\Phi}_\alpha \jmath^*_{\bm\mu}h)\jmath^*_{\bm\mu}\tau^\alpha)\\=\inn_{\mathbf{X}^h}\pi^*_{\bm\mu}\bm\omega_{\bm \mu}+(R^{\mathbf{J}^\Phi}_\alpha \pi^*_{\bm\mu}h_{\bm \mu})\pi^*_{\bm\mu}\tau_{\bm \mu}^\alpha)=\pi^*_{\bm\mu}(\inn_{\bm Y}\bm\omega_{\bm\mu}+(R_{{\alpha\bm\mu} }h_{\bm\mu})\tau_{\bm\mu}^\alpha),
    \end{multline*}
where we denoted by $\bfX^h$ both a $k$-polycosymplectic Hamiltonian $k$-vector field on $M$ and its restriction to ${\bf J}^{\Phi -1}(\bm\mu)$. Moreover,
    \begin{gather*}
        \pi^*_{\bm\mu}(\inn_{Y_\alpha}\tau^\beta_{\bm\mu})=\inn_{X^h_\alpha}(\pi^*_{\bm\mu}\tau^\beta_{\bm\mu})=\jmath_{\bm\mu}^*(\inn_{X^h_\alpha}\tau^\beta)=\delta^\beta_\alpha.
    \end{gather*}
Therefore, $\bfY$ is a reduced $k$-polycosymplectic Hamiltonian $k$-vector field such that $\pi_{\bm\mu*}\mathbf{X}^h=\mathbf{Y}$ and $\pi_{\bm\mu}\circ F^\alpha_s=K^\alpha_s\circ \pi_{\bm\mu}$ holds for $\alpha=1,\ldots,k$ and $s\in \mathbb{R}$.
\end{proof}
Additionally, Theorem \ref{Th::PolyCoReductionDynamics} could be established via $k$-polysymplectic reduction, i.e. Theorem \ref{Th::PolisymplecticReductionJ}, by extending the Hamiltonian function $h\in \Cinfty(M)$ to $\R^k\x M$. Let us study this in some detail.

Consider the extended Hamiltonian function $\wtl h\in\Cinfty(\R^k\x M)$ given by
\begin{equation}
\widetilde{h}(\bm{u},x)=h(x)-\sum^k_{\alpha=1}u^\alpha\,,\qquad \forall x\in M\,,\qquad 
\forall {\bm u}=(u_1,\ldots,u_k)\in\R^k\,.
\end{equation}
Recall that a $k$-polycosymplectic Hamiltonian $k$-vector field $\bfX^h$ associated with $h$ satisfies the following equations
\[
\inn_{\bfX^h}\boldsymbol{\omega}=\d h-(R_\beta h)\tau^\beta\,,\qquad \inn_{X^\alpha_h}\boldsymbol{\tau}=e^\alpha\,,\qquad \alpha=1,\ldots,k\,.
\]
Then, our aim is to extend $\bfX^h$ to a $k$-polysymplectic Hamiltonian $k$-vector field $\bfX^{\wtl h}$ associated with $\wtl h$. It can be verified that $\bfX^{\wtl h}$ of the form
\[
\bfX^{\widetilde{h}}=\bfX^h+(R_\alpha h)\frac{\partial}{\partial u^\alpha}\,,
\]
satisfies the required conditions for a $k$-polysymplectic Hamiltonian $k$-vector field on $\R^k\x M$ relative to $\widetilde{\bm\omega}$, namely
\[
\inn_{\bfX^{\widetilde{h}}}\widetilde{\bm\omega}=\inn_{\bfX^{h}}\bm\omega+(R_\beta h)\tau^\beta-\sum^k_{\alpha=1}(\inn_{X^\alpha_h}\tau^\alpha) \d u^\alpha=\d h-\sum^k_{\alpha=1}\d u^\alpha=\d\widetilde{h}\,,
\]
where we have used the natural isomorphism $\T_{(\bm{u},x)}(\R^k\times M)\simeq \T_{\bm{u}}\R^k\times \T_x M$ for every $(\bm u,x)\in\R^k\x M$.
Therefore, $\bfX^{\widetilde{h}}$ is a $k$-polysymplectic Hamiltonian $k$-vector field on $\mathbb{R}^k\times M$ related to $\widetilde{h}\in \Cinfty(\R^k\times M)$ with respect to $\widetilde{\bm \omega}$. From Theorem \ref{Th::PolyReductionDynamics} and Lemma \ref{Lemm::MomentumMapRegValue}, it follows that the reduced $k$-polysymplectic Hamiltonian $k$-vector field $\bfX^h_{\bm \mu}$ has the following form
\[
\bfX^{\widetilde{h}}_{\bm \mu}=\bfX^h_{\bm\mu}+(R_{\alpha\bm \mu} 
\,h_{\bm\mu})\frac{\partial}{\partial u^\alpha}\,.
\]
Consequently, $\bfX^{\widetilde{h}}_{\boldsymbol{\mu}}$ projects onto $\mathbf{J}^{\Phi-1}(\bm\mu)/G^{\bm\Delta}_{\bm \mu}$ and its projection is $\bfX^h_{\bm\mu}$. It is immediate that the latter gives the desired reduction.

\section{Examples}\label{Se::Examples}

\subsection{The product of cosymplectic manifolds}

This section presents an illustrative example of the $k$-polycosymplectic reduction of a product of $k$ cosymplectic manifolds. For simplicity, we will assume some technical conditions. Let $M=\bigtimes^k_{\alpha=1}M_\alpha$ for some $k$ cosymplectic manifolds $(M_\alpha,\tau^\alpha_M,\omega^\alpha_M)$ for $\alpha=1,\ldots,k$. If ${\rm pr}_\alpha: M\rightarrow M_\alpha$ is the canonical projection onto the $\alpha$-component, $(M,\sum_{\alpha=1}^k{\rm pr}_\alpha^*\tau_M^\alpha\otimes e_\alpha,\sum_{\alpha=1}^k{\rm pr}_\alpha^*\omega^\alpha_M\otimes e_\alpha)$ is a $k$-polycosymplectic manifold. Moreover, assume that a Lie group action $\Phi^\alpha: G_\alpha\x M_\alpha\rightarrow M_\alpha$  admits a cosymplectic momentum map $\mathbf{J}^{\Phi^\alpha}: M_\alpha\rightarrow\mathfrak{g}^*_\alpha$ for each $\alpha=1,\ldots,k$  and each $\Phi^\alpha$  acts in a quotientable manner on the level sets given by regular values of ${\bf J}^{\Phi^\alpha}$.

Then, define the Lie group action of $G=G_1\x\ldots\x G_k$ on $M$ as
\[
    \Phi:G\x M\ni (g_1,\ldots,g_k,x_1,\ldots,x_k)\longmapsto (\Phi^1_{g_1}(x_1),\ldots,\Phi^k_{g_k}(x_k))\in M\,.
\]
Moreover, let $\mathfrak{g}=\mathfrak{g}_1\x\ldots\x\mathfrak{g}_k$ be the Lie algebra of $G$. Then, we have the $k$-polycosymplectic momentum map 
\[
\mathbf{J}:M\ni(x_1,\ldots,x_k)\longmapsto(\mathbf{J}^{\Phi_1}(x_1),\ldots,\mathbf{J}^{\Phi_k}(x_k))\in\mathfrak{g}^*,
\]
where $\mathfrak{g}^*=\mathfrak{g}_1^*\x\ldots\x\mathfrak{g}_k^*$ is dual space to $\mathfrak{g}$.
Suppose, that $\mu^\alpha\in\mathfrak{g}^*_\alpha$ is a regular value of $\mathbf{J}^{\Phi_\alpha}:M_\alpha\rightarrow\mathfrak{g}^*_\alpha$ for each $\alpha=1,\ldots,k$. Hence, $\bm{\mu}=(\mu^1,\ldots,\mu^k)\in(\mathfrak{g}^*)^k$ is a regular value of $\mathbf{J}$. Then,   $\Phi$ acts in a quotientable on the associated level sets of ${\bf J}$.

Therefore, if $x=(x_1,\ldots,x_k)\in\mathbf{J}^{-1}(\bm\mu)$, it follows that
\begin{align}
    \ker\T_x \mathbf{J}_\alpha &= \T_{x_1}M_1\oplus\ldots\oplus \ker \T_{x_\alpha} \mathbf{J}^{\Phi_\alpha}\oplus\ldots\oplus \T_{x_k}M_k,\\
    \T_x\left(\mathbf{J}^{-1}(\bm\mu)\right)&=\ker \T_{x_1}\mathbf{J}^{\Phi_1}\oplus\ldots\oplus \ker \T_{x_k}\mathbf{J}^{\Phi_k},\\
    \ker \omega_x^\alpha\cap\ker\tau_x^\alpha &= \T_{x_1}M_1\oplus\ldots\oplus \T_{x_{\alpha-1}}M_{\alpha-1}\oplus \{0\}\oplus \T_{x_{\alpha+1}}M_{\alpha+1}\oplus\ldots\oplus \T_{x_k}M_k\,,\\
    \T_x\left(G^{\Delta^\alpha}_{\mu^\alpha}x\right) &= \T_{x_1}\left(G_1 x_1\right)\oplus\ldots\oplus \T_{x_\alpha}\left(G^{\Delta^\alpha}_{\alpha\mu^\alpha}x_\alpha\right)\oplus\ldots\oplus \T_{x_k}\left(G_{k}x_k\right),\\
    \T_x\left(G_{\bm\mu}^{\bm\Delta} x\right) &= \T_{x_1}\left(G^{\Delta^1}_{1\mu^1}x_1\right)\oplus\ldots\oplus \T_{x_k}\left(G^{\Delta^k}_{k\mu^k}x_k\right).
\end{align}
Then,
\begin{align*}
\ker\T_x\mathbf{J}_\alpha &= \T_x\left( \mathbf{J}^{-1}({\bf \mu})\right)+\ker\omega_x^\alpha\cap\ker\tau_x^\alpha,\\ \T_x\left(G_{\bm\mu}^{\bm\Delta} x\right)&=\bigcap^k_{\beta=1}\left(\ker\omega_x^\beta\cap\ker\tau_x^\beta + \T_x\left(G^{\Delta^\beta}_{\mu^\beta}x\right)\right),
\end{align*}
for $\alpha=1,\ldots,k$ and every 
 regular $\bm\mu\in (\mathfrak{g}^*)^k$ and $x\in {\bf J}^{-1}({\bm \mu})$. Recall that, by Theorem \ref{Th::PolisymplecticReductionJ}, these equations guarantee that the reduced space $\mathbf{J}^{-1}(\bm\mu)/G^{\bm\Delta}_{\bm\mu}$ can be endowed with a $k$-polycosymplectic structure, while
\[
\mathbf{J}^{-1}(\bm\mu)/G^{\bm\Delta}_{\bm\mu}\simeq \mathbf{J}^{\Phi_1-1}(\mu^1)/G^{\Delta^1}_{1\mu^1}\x\ldots\x \mathbf{J}^{\Phi_k-1}(\mu^k)/G^{\Delta^k}_{k\mu^k}.
\]


\subsection{Two coupled vibrating strings}

Consider the manifold $M = \R^2\x\bigoplus^2\cT\R^2$ with adapted coordinates $\{t,x;q^1,q^2,p_1^t,p_1^x,p_2^t,p_2^x\}$ and the standard associated two-polycosymplectic structure
\[
\boldsymbol{\tau}=\d t\otimes e_1+\d x\otimes e_2,\qquad \boldsymbol{\omega}=(\d q^1\wedge \d p_1^t+\d q^2\wedge \d p_2^t)\otimes e_1+(\d q^1\wedge \d p_1^x+\d q^2\wedge \d p_2^x)\otimes e_2.
\]
Consider the Hamiltonian function $h\in\Cinfty(M)$ given by
\[
h(t,x,q^1,q^2,p_1^t,p_2^t,p_1^x,p_2^x) = \frac{1}{2}\left( (p_1^t)^2 + (p_2^t)^2 - (p_1^x)^2 - (p_2^x)^2 \right) + C(t,x,q^1-q^2)\,,
\]
where $C(t,x,q^1-q^2)$ is a coupling function between the two strings. This system admits a Lie symmetry given by 
\[
\xi_M=
\frac{\partial}{\partial q^1}+\frac{\partial}{\partial q^2}
\]
associated with the Lie group action $\Phi:\R\x M\rightarrow M$ acting by translations along the $q^1+q^2$ direction, namely
\[
\Phi:(\lambda;t,x,q^1,q^2,p^t_1,p^t_2,p^x_1,p^x_2)\ni \mathbb{R}\times M\mapsto (t,x,q^1+\lambda,q^2+\lambda,p^t_1,p^t_2,p^x_1,p^x_2)\in  M.
\]
The Lie group action $\Phi$ gives rise to a two-polycosymplectic momentum map $\mathbf{J}^\Phi$ given by
\[\textstyle
{\bf J}^\Phi:(t,x,q^1,q^2,p^t_1,p^t_2,p^x_1,p^x_2)\in \R^2\times\bigoplus^2\cT\R^2\mapsto (p_1^t+p^t_2,p_1^x+p_2^x)=:(\mu^1,\mu^2)=\bm\mu\in (\R^*)^2.
\]
Consequently, the level set of the two-polycosymplectic momentum map $\mathbf{J}^\Phi$ is as follows
\[
{\bf J}^{\Phi-1}(\bm\mu)=\{(t,x,q^1,q^2,p_1^t,\mu^1-p_1^t,p_1^x,\mu^2-p^x_1)\in M:(t,x,q^1,q^2,p_1^t,p_1^x)\in \mathbb{R}^6\}.
\]
It is immediate that $\bm\mu=(\mu^1,
\mu^2)$ is a weak regular value of $\mathbf{J}^\Phi$ and $\mathbf{J}^\Phi$ is ${\Ad}^{*2}$-equivariant. Note that $\mathbf{J}^{\Phi-1}(\bm\mu)\simeq \R^6$ and $\R=G_{\boldsymbol{\mu}}=G_{\mu^\alpha}$ for $\alpha=1,2$. Then,
\begin{gather*}
    \T_m\left( G_{\bm\mu}m\right)=\T_m\left( G_{\mu^\alpha}m\right)=\left\langle \frac{\partial}{\partial q^1}+\frac{\partial}{\partial q ^2}\right\rangle_m,\\
    (\ker \omega^1\cap\ker\tau^1)_m=\left\< \frac{\partial}{\partial x},\frac{\partial}{\partial p^x_1},\frac{\partial}{\partial p^x_2}\right\>_m,\quad(\ker\omega^2\cap\ker\tau^2)_m=\left\< \frac{\partial}{\partial t},\frac{\partial}{\partial p^t_1},\frac{\partial}{\partial p^t_2}\right\>_m,\\\quad \T_m\mathbf{J}^{\Phi-1}(\bm\mu)=\left\< \frac{\partial}{\partial t},\frac{\partial}{\partial x},\frac{\partial}{\partial q^1},\frac{\partial}{\partial q^2},\frac{\partial}{\partial p^t_1}-\frac{\partial}{\partial p^t_2},\frac{\partial}{\partial p^x_1}-\frac{\partial}{\partial p^x_2}\right\>_m,\\
    \ker \T_m\mathbf{J}^\Phi_1=\left\< \frac{\partial}{\partial t}, \frac{\partial}{\partial x}, \frac{\partial}{\partial q^1},\frac{\partial}{\partial q^2}, \frac{\partial}{\partial p^x_1}, \frac{\partial}{\partial p^t_1}-\frac{\partial}{\partial p^t_2},\frac{\partial}{\partial p^x_2}\right\>_m,\\
    \ker \T_m\mathbf{J}^\Phi_2=\left\< \frac{\partial}{\partial t}, \frac{\partial}{\partial x}, \frac{\partial}{\partial q^1},\frac{\partial}{\partial q^2}, \frac{\partial}{\partial p^x_1}-\frac{\partial}{\partial p^x_2}, \frac{\partial}{\partial p^t_1},\frac{\partial}{\partial p^t_2}\right\>_m
\end{gather*}
and, indeed, the conditions \eqref{Eq::PolycosymplecticReductionTh1} and \eqref{Eq::PolycosymplecticReductionTh2} hold. 

Recall that the dynamics on $M$ is given by a two-polycosymplectic Hamiltonian two-vector field. Therefore, let us consider a general two-vector field ${\bf X}^h = (X^h_1,X^h_2)\in\X^2(M)$ with local expression
\[
    X^h_\alpha = A_\alpha^t\parder{}{t} + A_\alpha^x\parder{}{x} + B_\alpha^1\parder{}{q^1} + B_\alpha^2\parder{}{q^2} + C_{\alpha 1}^t\parder{}{p_1^t} + C_{\alpha 1}^x\parder{}{p_1^x} + C_{\alpha 2}^t\parder{}{p_2^t} + C_{\alpha 2}^x\parder{}{p_2^x}\,.
\]
Imposing the two-polycosymplectic Hamiltonian equations \eqref{eq:polycosymplectic-equations-fields}, the previous two-polycosymplectic Hamiltonian two-vector field $\bfX^h = (X^h_1,X^h_2)$ must be of the form
\[
\begin{gathered}
    X_1^h = \parder{}{t} + p_1^t\parder{}{q^1} + p_2^t\parder{}{q^2} + C_{1 1}^t\parder{}{p_1^t} + C_{1 1}^x\parder{}{p_1^x} + C_{1 2}^t\parder{}{p_2^t} + C_{1 2}^x\parder{}{p_2^x}\,,\\
    X_2^h = \parder{}{x} - p_1^x\parder{}{q^1} - p_2^x\parder{}{q^2} + C_{2 1}^t\parder{}{p_1^t} - \left( C_{11}^t + \parder{C}{q} \right)\parder{}{p_1^x} + C_{2 2}^t\parder{}{p_2^t}+\left(\frac{\partial C}{\partial q} -C_{12}^t\right)\parder{}{p_2^x}\,,
\end{gathered}
\]
where $q=q^1-q^2$ and $C_{11}^t,C_{11}^x,C_{12}^t,C_{12}^x,C_{21}^t,C_{22}^t\in\Cinfty(M)$ are, in principle, arbitrary functions. 

 Its integral sections, with $t,x$ being the coordinates in its domain, satisfy 
 \begin{gather}
 \label{Eq::FunctionsExample}
     C^t_{11}=\frac{\partial p^t_1}{\partial t}\,,\qquad C^x_{11}=\frac{\partial p^x_1}{\partial t}\,,\qquad C^x_{12}=\frac{\partial p^x_2}{\partial t}\,,\qquad C^t_{21}=\frac{\partial p^t_1}{\partial x}\,,\qquad C^t_{22}=\frac{\partial p^t_2}{\partial x}\,,\\
     -C^t_{11}-\frac{\partial C}{\partial q}=\frac{\partial p^x_1}{\partial x}\,,\qquad \frac{\partial C}{\partial q}-C^t_{12}=\frac{\partial p^x_2}{\partial x}\,.
 \end{gather}
 This system of PDEs is integrable when $[X^h_1,X^h_2]=0$, for instance, if $C=qF(x)+\widehat{F}(t,x)$ for arbitrary functions $\widehat{F}(t,x)$, $F(x)$, while $C_{11}^t,C_{11}^x,C_{12}^t,C_{12}^x,C_{21}^t,C_{22}^t$ vanish.
Then, to apply Theorem \ref{Th::PolyCoReductionDynamics}, we require $\bfX^h$ to be tangent to $\mathbf{J}^{\Phi-1}(\bm{\mu})$ and $\Lie_{\xi_M}X_\alpha=0$ for $\alpha=1,2$. Thus, $C_{12}^t+C^t_{11}=0$, $C^x_{11}+C^x_{12}=0$, and $C^t_{21}+C^t_{22}=0$ and $C_{ij}^t,C_{ij}^x$ must be first-integrals of $\xi_M$ for $i,j=1,2$. A two-polycosymplectic Hamiltonian two-vector field gives rise to the following Hamilton--De Donder--Weyl equations
\begin{gather}
    \parder{q^1}{t} = p_1^t\,,\qquad \parder{q^1}{x} = -p_1^x\,,\qquad \parder{q^2}{t} = p_2^t\,,\qquad \parder{q^2}{x} = -p_2^x\,,\\
    \parder{p_1^t}{t} + \parder{p_1^x}{x} = -\parder{C}{q}\,,\qquad \parder{p_2^t}{t} + \parder{p_2^x}{x} = \frac{\partial C}{\partial q}\,.
\end{gather}

Since $G=\R$ acts on $\mathbf{J}^{\Phi-1}(\boldsymbol\mu)$ by translations along the $q^1+q^2$ direction, the Lie group action $\Phi$ is free and proper. Therefore, $\mathbf{J}^{\Phi-1}(\boldsymbol{\mu})/G_{\boldsymbol{\mu}}$ is a smooth manifold and 
\[
\mathbf{J}^{\Phi-1}(\boldsymbol{\mu})/G_{\boldsymbol{\mu}}\simeq \R^2\x\T^*\R^2/\R\simeq\R^2\times \R^2/\R\times \R\simeq\R^2\x \R\x\R^2.
\]
Then, on the reduced manifold $\mathbf{J}^{\Phi-1}(\boldsymbol{\mu})/\R$, the reduced two-polycosymplectic structure reads
\[
\bm\tau_{\bm\mu}=\d t\otimes e_1+\d x\otimes e_2\,,\qquad \bm\omega_{\bm\mu}=\d q\wedge \d p_1^t\otimes e_1+\d q\wedge \d p_1^x\otimes e_2\,.
\]
Indeed, it becomes a two-cosymplectic structure since
\[
\ker \bm\omega_{\bm\mu}=\left\langle \frac{\partial}{\partial t},\frac{\partial}{\partial x}\right\rangle\,,\qquad \ker \bm\tau_{\bm\mu}\cap \ker\bm\omega_{\bm\mu}=0\,.
\]
The reduced dynamics on $\mathbf{J}^{\Phi-1}(\boldsymbol{\mu})/\R$ is given by the reduced two-polycosymplectic Hamiltonian two-vector field $\mathbf{X}^{h_{\boldsymbol{\mu}}}=(X^{h_{\boldsymbol{\mu}}}_1,X^{h_{\boldsymbol{\mu}}}_2)$ of the form
\begin{align}
X^{h_{\boldsymbol{\mu}}}_1&=\parder{}{t}+p^t\parder{}{q}+2C^t_{11}\parder{}{p^t}+2C^{x}_{11}\parder{}{p^x}\,,\\
X^{h_{\boldsymbol{\mu}}}_2&=\parder{}{x}-p^x\parder{}{q}-2C^t_{22}\parder{}{p^t}-2\left(C^t_{11}+\frac{\partial C}{\partial q}\right)\parder{}{p^x}\,,
\end{align}
where $h_{\boldsymbol{\mu}}=\frac{1}{4}\left( (p^t)^2+(p^x)^2+(\mu^1)^2-(\mu^2)^2\right)+C(t,x,q)$ is the reduced Hamiltonian function, where $p^t=p_1^t-p^t_2$ and $p^x=p^x_1-p^x_2$. A reduced two-polycosymplectic Hamiltonian two-vector field induces the following Hamilton--De Donder--Weyl equations
\begin{gather*}
    \frac{\partial q}{\partial x}=-p^x\,,\qquad \frac{\partial q}{\partial t}=p^t\,,\\
    \frac{\partial p^t}{\partial t}+\frac{\partial p^x}{\partial x}=-2\frac{\partial C}{\partial q}\,.
\end{gather*}

Let us generalise some of the ideas of the previous example. The starting point is the following proposition, whose proof is an immediate extension to the $k$-polycosymplectic realm of results in symplectic geometry.
\begin{proposition}\label{Prop::Red} Let $\Phi:G\times Q\rightarrow Q$ be a Lie group action. There exists a Lie group action $\widetilde{\Phi}:G\times \mathbb{R}^k\times \bigoplus_{\alpha=1}^k\cT Q\rightarrow \mathbb{R}^k\times \bigoplus_{\alpha=1}^k\cT Q$ of the form
$$
\widetilde\Phi(g,x,q,p^\alpha)=(x,\Phi(g,q),(\T_{\Phi(g,q)}\Phi_{g^-1})^*p^\alpha)\in \mathbb{R}^k\times \bigoplus_{\alpha=1}^k\cT Q\,.
$$
This Lie group action is $k$-polycosymplectic relative to the canonical $k$-polycosymplectic structure on $\mathbb{R}^k\times \bigoplus_{\alpha=1}^k\cT Q$ with the unique momentum map ${\bf J}^{\widetilde{\Phi}}:\mathbb{R}^k\times \bigoplus_{\alpha=1}^k \cT Q\rightarrow (\mathfrak{g}^{*})^k$ such that
$$
\langle{\bf J}^{\widetilde{\Phi}}(x,q,p^\alpha),{\bm \xi}\rangle=\iota_{{\bm \xi}_M}{\bf p}(q)\,,\qquad \forall {\bm \xi}\in \mathfrak{g}^k\,,\qquad \forall (x,q,p^\alpha)\in \mathbb{R}^k\times \bigoplus_{\alpha=1}^k\cT Q\,.
$$
Moreover, ${\bf J}^{\widetilde{\Phi}}$ is ${\rm Ad}^{*k}$-equivariant. 
\end{proposition}

Proposition \ref{Prop::Red} implies that one may accomplish a $k$-polycosymplectic reduction relative to the level sets of ${\bf J}^{\widetilde{\Phi}}$, provided our technical conditions, e.g. \eqref{Eq::PolycosymplecticReductionTh1} and \eqref{Eq::PolycosymplecticReductionTh2}, are satisfied. In fact, the reduction procedure accomplished in the example of this section is nothing but a particular case of this construction for a Lie group symmetry. It is relevant to stress that this procedure does not allow for a reduction involving the variables of  $\mathbb{R}^k$. Physically, this reduction involves the variables of $\bigoplus_{\alpha=1}^k\cT Q$, namely the so-called fields and their momenta. Meanwhile, the variables in $\mathbb{R}^k$, which are physically related to space-times and other manifolds where the problem under study occurs, cannot be reduced. To do so, we will develop a new method described in the next section.

\section{A \texorpdfstring{$k$}--cosymplectic to \texorpdfstring{$\ell$}--polycosymplectic reduction}\label{Sec::SPRed}

This section performs a Marsden--Weinstein reduction from a $k$-cosymplectic to an $\ell$-cosymplectic manifold. Moreover, our procedure may allow, under certain conditions, for the further successive application of the techniques developed in previous sections to obtain a smaller $\ell$-polycosymplectic manifold. To simplify our presentation and avoid explaining trivial results, we hereafter assume $\ell<k$. In physics, this permits us to develop Marsden--Weinstein reduction techniques for field theories involving the elimination of space-time variables, which is not possible via the approaches described in the previous sections since the fundamental vector fields of the involved $k$-polycosymplectic Lie group action took values in $\ker \bm\tau$. It is worth noting that the reductions developed in this section represent a rather pioneering and non-standard approach in the literature, which is frequently based on other methods, e.g. principal bundles and Lie group actions not involving the reduction of base manifolds \cite{CM08, CR03}. Moreover, we will hereafter give conditions allowing for the reduction of the HDW equations on a $k$-cosymplectic manifold to the ones in the reduced $\ell$-cosymplectic manifold and its potential further $\ell$-polycosymplectic reductions. 

This section is restricted to the study of a canonical $k$-cosymplectic manifold $(M_k\!=\!\mathbb{R}^k\times \bigoplus_{\alpha=1}^k\cT Q,{\bm \tau}_k,{\bm \omega}_k)$ with its natural polarisation $V_k$. It is worth noting that, in virtue of the Darboux's theorem for $k$-cosymplectic manifolds \cite{LM98}, every $k$-cosymplectic manifold is locally diffeomorphic to $(M_k,{\bm \tau}_k,{\bm \omega}_k)$. Hence, our results apply, locally, to every $k$-cosymplectic manifold. Let us start by accomplishing a $k$-cosymplectic to $\ell$-cosymplectic Marsden--Weinstein manifold reduction. As before, $\{e_1,\ldots,e_k\}$ is a basis for $\mathbb{R}^k$, while ${\bm \tau}_k=\tau^\alpha\otimes e_\alpha\in \Omega^1(M_k,\mathbb{R}^k)$ and ${\bm \omega}_k=\omega^\alpha\otimes e_\alpha\in \Omega^2(M_k,\mathbb{R}^k)$. Recall that the sum over repeated crossed indexes of a particular type will be over its standard range, e.g. $\alpha$ ranges from 1 to $k$, unless a summation symbol indicating otherwise is  used. It is worth stressing that all results to be hereafter described are local.

\begin{theorem}
\label{Th::kcotosco}
    Let $(M_k,\bm\tau_k, \bm\omega_k, V_k)$ be a canonical $k$-cosymplectic manifold, let $\Phi: G\times M_k\rightarrow M_k$ be an associated $k$-cosymplectic Lie group action, and let $\{\overline{\tau}^1,\ldots,\overline{\tau}^\ell\}$ be a basis of the linear subspace (over the real numbers) of  $\langle\tau^1,\ldots,\tau^k\rangle $ vanishing on the space $W$ of fundamental vector fields of $\Phi$. Assuming, without loss of generality, that the last $k-\ell$ differential one-forms in the basis $\{\tau^1,\ldots,\tau^k\}$ are linearly independent as linear forms on $W$,  we set $\overline{\tau}^\beta=c^\beta_{\,\alpha}\,\, \tau^\alpha$ for certain unique constants $c^\beta_\alpha$ with $\alpha=1,\ldots,k$ and $\beta=1,\ldots,\ell$.  We define $\overline{\bm \tau}=\sum_{\beta=1}^\ell c_\alpha^\beta\, \tau^\alpha\otimes e_\beta$ and $ \overline{\bm\omega}=\sum^\ell_{\beta=1} c^\beta_{\,\alpha}\,\, \omega^\alpha\otimes e_\beta$. If the map
    $$
    \pi:(\bar x^1,\ldots,\bar x^k,q,\bar p^1,\ldots,\bar p^k)\in \mathbb{R}^k\times\bigoplus_{\alpha=1}^k\cT Q\longmapsto (\bar x^1,\ldots,\bar x^\ell,q,\bar{p}^1,\ldots,\bar{p}^\ell)\in \mathbb{R}^\ell\times \bigoplus_{\alpha=1}^\ell\cT Q
    $$
    is the canonical projection, then $(M_\ell=\R^\ell\x\bigoplus_{\alpha=1}^\ell\cT Q,  {\bm\tau}_\ell, {\bm\omega}_\ell)$ is an $\ell$-cosymplectic manifold with
    $$
    \pi^*{\bm \omega}_\ell=\overline{\bm \omega}\,,\qquad\pi^*{\bm \tau}_\ell=\overline{\bm \tau}\,.
    $$
    Furthermore, there exists a Lie group action $\Phi_\ell:G\times M_\ell\rightarrow M_\ell$ that is equivariant relative to $\pi$.
\end{theorem}
\begin{proof}
Let $\{x^\alpha,y^j,p^\alpha_j\}$ denote locally adapted coordinates to $(M_k,{\bm \tau}_k,{\bm \omega}_k)$, namely
\[
{\bm \omega}_k=\d y^j\wedge \d p_j^\alpha\otimes e_\alpha\,,\qquad {\bm \tau}_k=\d x^\alpha\otimes e_\alpha\,,\qquad V_k=\bigg\langle \frac{\partial}{\partial p^\alpha_j}\bigg\rangle_{\alpha=1,\ldots,k,\ j=1,\ldots,\dim Q}\,. 
\]
Since $\Phi^*_g{\bm \omega}_k={\bm \omega}_k$ for every $g\in G$, then $\ker {\bm \omega}_k=\left\langle\tparder{}{x^1},\ldots,\tparder{}{x^k}\right\rangle$ is $G$-invariant with respect to the lifted Lie group action of $\Phi$ to $\T M_k$. Furthermore, as $\Phi_g^*\bm\tau_k=\bm\tau_k$ for every $g\in G$, there exists a local linear Lie group action $\Phi^{k}: G\times \mathbb{R}^k\rightarrow \R^k$ whose space of orbits, around a point of $\mathbb{R}^k$, is a quotient space, $\mathbb{R}^k/E$, for a certain linear subspace $E\subset \mathbb{R}^k$. Hence, $\mathbb{R}^k/E$ has a natural structure of $\ell$-dimensional linear space. 

Since ${\bm \tau}_k$ is closed and  $\Lie_{\xi_{M_k}}{\bm \tau
}_k=0$ for every $\xi\in\mathfrak{g}$, it follows that $\iota_{\xi_{M_k}}{\bm \tau}_k$ is constant for each $\xi\in \mathfrak{g}$. Therefore, $\mathfrak{W}=\langle \tau^1,\ldots,\tau^k\rangle$ can be considered as a linear subspace of the dual, $W^*$, to the linear (over the reals) space $W$ of fundamental vector fields of the Lie group action $\Phi$. Hence, there exists a linear subspace $\mathfrak{A}\subset \mathfrak{W}$ consisting of the elements of $\mathfrak{W}$ vanishing on $W$. Let  $\{\overline\tau^1,\ldots,\overline\tau^\ell\}$ be a basis of $\mathfrak{A}$. Then, we can define $\overline{\tau}^\beta=c^\beta_{\,\alpha}\,\,\tau^\alpha$ and $\overline{\omega}^\beta=c^\beta_{\,\alpha}\omega^\alpha$ for some unique constants $c^\beta_{\,\alpha}$, where $\alpha=1,\ldots,k$ and $\beta=1,\ldots,\ell$. 
Note that $\overline{\bm\tau}\in\Omega^1(\R^k\times \bigoplus_{\alpha=1}^k\cT Q,\R^\ell)$ and $\overline{\bm\omega}\in\Omega^2(\R^k\times \bigoplus_{\alpha=1}^k\cT Q,\R^\ell)$. It follows that $\overline{\bm\tau}$ and $\overline{\bm\omega}$ are closed, since $\bm\tau_k$ and $\bm\omega_k$ are closed and the coefficients $c^\beta_{\,\alpha}$, with $\alpha=1,\ldots,k$ and $\beta=1,\ldots \ell$, are constants. There exist new local  adapted coordinates to $M_k$ obtained linearly from the previous ones, let us say $\{\bar x^\alpha=A^\alpha_\beta x^\beta, y^j,\bar p^\alpha_j=A^{\alpha}_{ \beta}p^\beta_j\}$ for a certain constant $(k\times k)$-matrix $A^\alpha_\beta$, such that
\[
\overline{\bm \omega}=\sum_{\beta=1}^\ell\d y^j\wedge \d \bar p^\beta_j\otimes e_\beta\,,\qquad \overline{\bm \tau}=\sum_{\beta=1}^\ell\d \bar x^\beta\otimes e_\beta\,.
\]
Note that $\ker\overline {\bm \tau}\cap \ker \overline{\bm \omega}$ is an integrable regular distribution on $M_k$ given by
$$
\ker \overline{\bm \tau}\cap \ker \overline{\bm \omega}=\bigg\langle\frac{\partial }{\partial \bar p^\alpha_j},\frac{\partial}{\partial \bar{x}^\alpha}\bigg\rangle_{\footnotesize \begin{array}{c}\alpha=\ell+1,\ldots,k,\\ j=1,\ldots,
\dim Q\end{array}}.
$$
The pair $(\overline{\bm \tau},\overline{\bm \omega})$ is not a $k$-cosymplectic structure on $M_k$, but a $k$-precosymplectic one (see \cite{Gra2020} for details).

The space   $\T_x M_k/\left(\ker\overline{\boldsymbol\tau}\cap\ker\overline{\boldsymbol{\omega}}\right)_x$ is diffeomorphic to $\T_{\pi(x)}M_\ell$ for  $x\in M_k$, where $M_\ell=\mathbb{R}^\ell\times \bigoplus_{\alpha=1}^\ell\cT Q$. Since  $\overline{\bm \tau}$ and $\overline{\bm \omega}$ vanish on the fundamental vector fields $\ker \overline{\bm\omega}\cap \ker \overline{\bm \tau}$ and are closed, they are projectable via the canonical projection $\pi:M_k\rightarrow M_\ell$ onto $M_\ell$ giving rise to an $\ell$-cosymplectic manifold $(M_\ell,\bm\tau_\ell,\bm\omega_\ell)$, where ${\bm \tau}_\ell$ and ${\bm \omega}_\ell$ are the only differential forms on $M_\ell$ taking values in $\mathbb{R}^\ell$ so that
\[
\pi^*{\bm \tau}_\ell=\overline{\bm \tau}\,,\qquad\pi^*{\bm \omega}_\ell=\overline{\bm \omega}\,.
\]
Let $\xi$ be any element of $\mathfrak{g}$. Moreover,  $\iota_{\xi_{M_k}}\overline{\bm \tau}=0$ and every $\xi\in \mathfrak{g}$. Then, for every vector field $X$ on $M_k$ taking values in $\ker \overline{\bm \tau}\cap\ker \overline{\bm \omega}$, one has 
$$
\iota_{[{\xi_{M_k}},X]}\overline{\bm \omega}=\Lie_{{\xi_{M_k}}}\iota_X\overline{\bm \omega}-\iota_X\Lie_{{\xi_{M_k}}}\overline{\bm \omega}=0
$$  
and, similarly, $\iota_{[\xi_{M_k},X]}{\overline{\bm\tau}}=0$. This implies that the fundamental vector fields of $\Phi$ project onto $M_\ell$ and give rise to a new Lie group action $\Phi^\ell: G\times M_\ell\rightarrow M_\ell$ equivariant to $\Phi$ relative to the canonical projection $\pi: M_k\rightarrow M_\ell$. 

\end{proof}

The previous procedure can potentially be continued by using an $\ell$-polycosymplectic momentum map ${\bf J}_\ell$ to perform an $\ell$-polycosymplectic Marsden--Weinstein reduction according to the results of previous sections. Note that the fundamental vector fields of $\Phi^\ell$ leave invariant ${\bm \omega}_{\ell}$, which means that, for each $\xi \in \mathfrak{g}$, we have $\iota_{\xi_{M_\ell}}{\bm \omega}_\ell=\d\bm h_\xi$ for a certain ${\bm h}_\xi$ and by construction $\iota_{\xi_{M_\ell}}{\bm \tau}_\ell=0$. Nevertheless, we have to impose that $R_{\ell\alpha}{\bf h_\xi}=0$ for $\alpha=1,\ldots,\ell$. Note that the latter is not satisfied in general: the initial $\Phi$ is just $k$-cosymplectic and it is not ensured that it admits a $k$-cosymplectic momentum map.

We shall now demonstrate how Theorem \ref{Th::kcotosco} induces some dynamics on the reduced space $M_\ell$ from one in $M_k$ satisfying certain conditions. Recall that, in the reductions of $k$-cosymplectic structures treated in this section, the last coordinates of ${\bm \tau}$ will be considered, without loss of generality, to be linearly independent as linear mappings on the fundamental vector fields of $\Phi$.

Before continuing, let us set and recall some notation used and to be employed.

\begin{center}
    \begin{tikzcd}
    \parbox{11cm}{\textrm{Original $k$-cosymplectic manifold and Hamiltonian $k$-cosymplectic $k$-vector field.}}
    &
    \big((M_k,{\bm \tau}_k,{\bm\omega}_k),\, {\bm X}^h\big)
    \arrow[dd,""]
    \\
    &
    \\
    \parbox{11cm}{\textrm{New $k$-cosymplectic manifold and Hamiltonian $k$-vector field obtained by making linear combinations of the components in $\mathbb{R}^k$ of ${\bm \omega}$ and ${\bm \tau}$ to ensure that the first $\ell$  components of $\widehat{\bm \tau}$ vanish on the fundamental vector fields of a $k$-cosymplectic Lie group action.}}
    &
    \big((M_k,\widehat{\bm\tau},\widehat{\bm \omega}), \,\widehat{\bm X}^h\big)
    \arrow[dd]
    \\
    &
    \\
    \parbox{11cm}{\textrm{An $\ell$-precosymplectic manifold and its $k$-vector field obtained by cutting the last $k-\ell$ components of the previous $k$-cosymplectic manifold and $k$-cosymplectic Hamiltonian $k$-vector field.}}
    &
    \big((M_k,\overline{\bm\tau},\overline{\bm \omega}), \,\overline{\bm X}^h\big)
    \arrow[dd,"\pi"]
    \\
    &
    \\
    \parbox{11cm}{\textrm{Projected $\ell$-cosymplectic manifold and its $\ell$-cosymplectic Hamiltonian $\ell$-vector field.}}
    &
    \big((M_\ell,{\bm\tau}_\ell,{\bm \omega}_\ell),\, {\bm X}^{h_\ell}\big)
    \end{tikzcd}
    \end{center}


\begin{theorem}\label{The::RedDynPolyCosym} Let $(M_k,{\bm \tau}_k,{\bm \omega}_k,V_k)$ be a $k$-cosymplectic manifold and let  $\Phi: G\times M_k\rightarrow M_k$ be an associated  $k$-cosymplectic Lie group action.  Assume that  $h\in \Cinfty(M_k)$ and  ${\bf X}^h$ is an associated $k$-cosymplectic Hamiltonian $k$-vector field  that is invariant relative to $\Phi$. Let us consider that $h$ is also invariant relative to the vector fields taking values in $\ker \overline{\bm \tau}\cap \T\mathbb{R}^k$ while the Lie bracket of any component of ${\bf  {X}}^h$ with any vector field taking values in $\ker\overline{\bm \tau}\cap\ker \overline{\bm \omega}$ takes values in the kernel of $\T\pi$. Let us suppose that $
\sum_{\alpha=\ell+1}^k[\widehat{X}^h_{\alpha}]^\alpha_i=0$ for $i=1,\ldots,\dim Q$. Then, there exists a function $h_\ell\in \Cinfty(M_\ell)$ such that ${\bf X}^{h_\ell}$ is the projection of $(\widehat{X}^h_1,\ldots,\widehat{X}^h_\ell)$ onto $M_\ell$ and $\pi^*h_\ell=h$ on a submanifold of constant values of the momenta $p^\alpha_i$ with $\alpha=\ell+1,\ldots,k$ and $i=1,\ldots,\dim Q$.  The $\ell$-vector field ${\bf X}^{h_\ell}$ is Hamiltonian relative to $(M_\ell,{\bm \tau}_\ell,{\bm \omega}_{\ell})$ and the solutions for the HDW equations of $h_\ell$ are solutions of the original HDW equations for constant associated momenta with $\alpha=\ell+1,\ldots,k$ for $\overline{\bm \tau},{\overline{\bm\omega}}$.
\end{theorem}
\begin{proof}
Let $\widehat{c}_\alpha^\beta$ be the matrix of the change of bases mapping  $\{\tau^1,\ldots,\tau^k\}$ into the new  basis 
\[
\widehat{\tau}^1=\bar{\tau}^1,\ldots,\widehat{\tau}^\ell=\bar{\tau}^\ell,\widehat{\tau}^{\ell+1}=\tau^{\ell+1},\ldots,\widehat{\tau}^k={\tau}^k\,,
\]
and let $\widehat d_\alpha^\beta$ be the inverse matrix, namely $\tau^\alpha=\widehat d^\alpha_\beta \widehat\tau^\beta$, for $\alpha,\beta=1,\ldots,k$. Define a new Hamiltonian $k$-cosymplectic $k$-vector field $\widehat{\bf X}^h$ on $M_k$ relative to $(M_k,\widehat{\bm \tau}=\widehat{c}^\beta_\alpha \tau^\alpha\otimes e_\beta,\widehat{\bm \omega}=\widehat{c}_\alpha^\beta \omega^\alpha \otimes e_\beta)$ of the form
\[
{\widehat{X}}^h_\alpha=\widehat{d}_\alpha^\beta \,{X}^h_\beta\,,\qquad \alpha,\beta=1,\ldots,k\,.
\]
Since $\widehat c^\beta_\alpha$ is such that $\widehat c^\beta_\alpha=\delta^\beta_\alpha$ for $\beta=\ell+1,\ldots,k$ and $\alpha=1,\ldots,k$ by construction of  ${\widehat{\bm \tau}}$, then $\widehat{ d}^\beta_\alpha=\delta^\beta_\alpha$ for $\beta=\ell+1,\ldots,k$ and $\alpha=1,\ldots,k$. The relations between the new canonical coordinates in $\mathbb{R}^k\times\bigoplus_{\alpha=1}^k\cT Q$  and the previous ones are given by
\[
\widehat{x}^\beta=\widehat{c}^\beta_\alpha x^\alpha\,,\qquad \widehat{ p}_i^\beta=\widehat{c}^\beta_\alpha p_i^\alpha\,, \qquad \alpha,\beta=1,\ldots,k\,,\quad i=1,\ldots,\dim Q\,,
\]
while $q^1,\ldots,q^{\dim Q}$ are the same in the new and the old coordinate systems. 

If $\widehat{R}_\alpha=\widehat{d}^\beta_\alpha R_\beta$, it follows that 
\begin{equation}\label{eq:Com}
\iota_{\widehat{X}^h_\alpha}\widehat{\omega}^\alpha=\d h-(\widehat{R}_\alpha h)\widehat{\tau}^\alpha,\qquad \iota_{\widehat{X}_\alpha^h}\widehat{\tau}^\beta=\delta^\beta_\alpha.
\end{equation}

 It is worth noting that if $
\psi:s=(s^1,\ldots,s^k)\in \mathbb{R}^k\mapsto (x^\alpha(s),q^i(s),p_i^\alpha(s))\in \mathbb{R}^k\times \bigoplus_{\alpha=1}^k\cT Q$ is a solution to the HDW equations of the original ${\bf X}^h$, then the same $\psi$ is a solution for the HDW equations for $\widehat{\bf X}^h$ in the new coordinates  $
\psi:\widehat s=(\widehat s^1,\ldots,\widehat s^k)\in \mathbb{R}^k\mapsto (\widehat x^\alpha(\widehat s),q^i(\widehat s),\widehat p_i^\alpha(\widehat{s}))\in \mathbb{R}^k\times \bigoplus_{\alpha=1}^k\cT Q$ with $\widehat s^\beta=\widehat{c}^\beta_
\alpha s^\alpha$ for $\alpha,\beta=1,\ldots,k$, namely
\begin{equation}\label{eq::HDW}
\frac{\partial \widehat{x}^\beta}{\partial \widehat s^\alpha}=\delta^\beta_\alpha\,,\quad \frac{\partial q^i}{\partial \widehat s^\alpha}=\frac{\partial h}{\partial \widehat p^\alpha_i}\,,\quad \sum_{\alpha=1}^k\frac{\partial \widehat p^\alpha_i}{\partial \widehat s^\alpha}=-\frac{\partial h}{\partial q^i}\,,\qquad \alpha,\beta=1,\ldots, k\,,\quad i=1,\ldots,\dim Q\,.
\end{equation}

Then, let us show there is a new  $\ell$-vector field  $\overline{\bf X}^h$ on $\mathbb{R}^k\times\bigoplus_{\alpha=1}^k \cT Q$ related to $(\mathbb{R}^k\times \bigoplus_{\alpha=1}^k\cT Q, \overline{\bm \tau},\overline{\bm \omega})$ of the form
\[
\overline{\bf X}^h=\sum_{\alpha=1}^\ell\widehat{X}^h_{\alpha}\otimes e_\alpha\,,
\]
satisfying
\[
\iota_{\overline{\bf X}^h}\overline{\bm \omega}=\d_\ell h-\sum_{\alpha=1}^\ell(\overline{R}_\alpha h)\bar \tau^\alpha\,,\qquad \overline{R}_\alpha=\widehat{R}_\alpha\,,\qquad \alpha=1,\ldots,\ell\,,
\]
where $\d_\ell$ is the differential taking into account all canonical coordinates apart from  $\widehat{x}^\alpha$ and $\widehat{p}^\alpha_i$ for $\alpha=\ell+1,\ldots,k$ and $i=1,\ldots, \dim Q$.
If follows from \eqref{eq:Com} that
$$
\sum_{\beta=1}^\ell\iota_{\overline{X}^h_\beta }\overline{\omega}^\beta+\sum_{\beta=\ell+1}^k\frac{\partial h}{\partial \widehat{p}_i^\beta}\d \widehat{p}^\beta_i-\sum_{\beta=\ell+1}^k[\widehat{X}^h_\beta]^\beta_{i}\d q^i= \frac{\partial h}{\partial q^i}\d q^i+\sum_{\alpha=1}^\ell\frac{\partial h}{\partial \widehat{p}^\alpha_i}\d \widehat{p}^\alpha_i+\sum_{\alpha=\ell+1}^k\frac{\partial h}{\partial \widehat{p}^\alpha_i}\d \widehat{p}^\alpha_i\,.
$$
If we assume $\sum_{\beta=\ell+1}^k[{\widehat{X}}^h_\beta]^\beta_{i}=0$ for $i=1,\ldots\dim Q$, then
\begin{equation}\label{eq:proj}
\sum_{\beta=1}^\ell\iota_{\overline{X}^h_\beta }\overline{\omega}^\beta=\frac{\partial h}{\partial q^i}\d q^i+\sum_{\alpha=1}^\ell\frac{\partial h}{\partial \widehat{p}^\alpha_i}\d \widehat{p}^\alpha_i=\d_\ell h-\sum_{\alpha=1}^\ell(\overline{R}_\alpha h) \overline{\tau}_\alpha\,.
\end{equation}
In particular, the previous expression holds on the submanifold $S_\lambda$ for $\widehat p^\alpha_i=\lambda_i^\alpha$ for certain constants $\lambda_i^\alpha$,  with $\alpha=\ell+1,\ldots,k$ and $i=1,\ldots,\dim Q$. Note that the projection of this submanifold relative to  $\pi: M_k\rightarrow M_\ell$ is surjective and open. By the given assumptions, the restriction of $h$ to $S_\lambda$ is projectable onto a function $h_\ell$ on $M_\ell$. Since the Lie bracket of ${\bf X}^h$ with any vector field in $\ker \overline{\bm \tau}\cap \ker {\overline{\bm \omega}}$  belongs to the kernel of $\T \pi$, it follows that the same applies to ${ \widehat{\bf X}}^{h}$ and  $(\overline{X}^h_1,\ldots,\overline{X}^h_{\ell})$ is projectable onto $M_\ell$, which implies that the Lie derivatives of the $\overline{X}^h_1,\ldots,\overline{X}^h_\ell$ with $\partial/\partial{\bar{x}^{\ell+1}},\ldots,\partial/\partial \bar{x}^{k}$ and their associated momenta belong to the kernel of  $\T\pi$. And then, \eqref{eq:proj} projects onto $M_\ell$. Moreover, $\iota_{\overline{X}_\alpha}\overline{\tau}^\beta=\delta^\beta_\alpha$ for $\alpha,\beta=1,\ldots,\ell$. These facts show that the projection of $(\overline{X}^h_{1},\ldots,\overline{X}^h_{\ell})$ is Hamiltonian relative to the induced $\ell$-cosymplectic manifold $(M_\ell, {\bm \tau}_\ell,{\bm \omega}_\ell)$. The new local canonical variables are given by
\[
\bar{x}^\alpha=\widehat{x}^\alpha\,,\quad \bar{p}_i^\alpha = \widehat{p}_i^\alpha\qquad i = 1,\dotsc,\dim Q\,,\qquad \alpha=1,\ldots,\ell\,.
\]
It is recommendable to take a look at the HDW equations for the $\ell$-cosymplectic structure. They take the form
\begin{equation}\label{eq:polycosymplectic-hamiltonian-equations-coordinates-red}
        \parder{\bar x^\beta}{\bar s^\alpha} = \delta_\alpha^\beta\,,\quad
        \parder{q^i}{s^\alpha} = \parder{h_\ell}{\bar p_i^\alpha}\,,\quad
\sum_{\alpha=1}^\ell\parder{\bar p_i^\alpha}{\bar s^\alpha} = -\parder{h_\ell}{q^i} \,,\qquad \alpha,\beta=1,\ldots,\ell\,,\qquad i=1,\ldots,\dim Q\,.
\end{equation}
These are indeed the equations for the solutions to \eqref{eq::HDW} with constant $\bar{p}^\alpha_i$ with $\alpha=\ell+1,\ldots,k$ and such that $h$ does not depend on $\bar{x}_{\ell+1},\ldots, \bar{x}_k$. Hence, this allows us to reduce our problem.
\end{proof}

It is worth noting that if the fundamental vector fields of $\Phi$ are tangent to the manifolds $S_\lambda$, then the function $h_\ell$ will be again invariant relative to $\Phi$ and a standard $\ell$-polycosymplectic reduction will be available if the reduction of $\Phi$ satisfies certain additional conditions, e.g. that the reduction will have a momentum map.

Note that Theorem \ref{The::RedDynPolyCosym} gives some conditions ensuring that the $k$-cosymplectic Hamiltonian $k$-vector field $\widehat{\bf X}^h$ can be projected onto $M_\ell$. These conditions are not the most general ones for the above theorem to hold. For instance, one can assume that only the first $\ell$ components of $\widehat{\bf X}^h$ are projectable. 

Let us apply our techniques to a vibrating membrane having an exterior force that depends only on the radial distance. Let us recall that this system is determined by  the Hamiltonian function $\widetilde{h}\in\Cinfty(\mathbb{R}^3\times\bigoplus_{\alpha=1}^3\cT \mathbb{R})$ given by 
$$ \widetilde h(t,r,\theta,\zeta,p^t,p^r,p^\theta) = \frac{1}{2r}\left( (p^t)^2 - \frac{1}{c^2}(p^r)^2 - \frac{r^2}{c^2}(p^\theta)^2 \right) - r \zeta f(r)\,, $$
and the canonical three-cosymplectic structure on $\mathbb{R}^3\times\bigoplus_{\alpha=1}^3\cT \mathbb{R}$ given by
\[
{\bm \tau}=\d t\otimes e_1+\d r\otimes e_2+\d\theta\otimes e_3\,,\qquad {\bm \omega}=\d \zeta\wedge \d p^t\otimes e_1+\d \zeta\wedge \d p^r\otimes e_2+\d \zeta\wedge \d p^\theta\otimes e_3\,.
\]
A section 
$$
\psi:(t,r,\theta)\in \mathbb{R}^3\mapsto (t,r,\theta,\zeta(t,r,\theta),p^t(t,r,\theta),p^x(t,r,\theta),p^y(t,r,\theta))\in \R^3\times\textstyle\bigoplus_{\alpha=1}^3\cT \R=:M^v_3\,,
$$
becomes a solution of the HDW equations of the three-cosymplectic Hamiltonian three-vector field ${\bf X}^{\widetilde{h}}=(X_1^{\widetilde{h}},X_2^{\widetilde{h}},X_3^{\widetilde{h}})$ on $M^v_3$ given by 
\[
{X}^{\widetilde{h}}_1=\frac{\partial}{\partial t}+\frac{p^t}{r}\frac{\partial}{\partial \zeta},\quad
{X}^{\widetilde{h}}_2=\frac{\partial}{\partial r}-\frac{p^r}{c^2r}\frac{\partial}{\partial \zeta}+rf(r)\frac{\partial}{\partial p^r},\quad
{X}^{\widetilde{h}}_3=\frac{\partial}{\partial \theta}-\frac{rp^\theta}{c^2}\frac{\partial}{\partial \zeta}
\]
if
\begin{gather}
    \parder{p^t}{t} + \parder{p^r}{r} + \parder{p^\theta}{\theta} = rf(r)\,,\\
    \parder{\zeta}{t} = \frac{1}{r}p^t\,,\quad \parder{\zeta}{r} = -\frac{1}{rc^2}p^r\,,\quad \parder{\zeta}{\theta} = -\frac{r}{c^2}p^\theta\,.
\end{gather}
Combining these equations, one obtains the well-known equation
\[
\parder{^2\zeta}{t^2} - c^2\left( \parder{^2\zeta}{r^2} + \frac{1}{r}\parder{\zeta}{r} + \frac{1}{r^2}\parder{^2\zeta}{\theta^2} \right) = f(r)\,.
\]
 of a forced vibrating membrane in polar coordinates.  Let us assume a reduction of the space-time appearing in this problem. 
 
The Lie group action
\[
\Phi:\mathbb{R}^2\times M^v_3\ni(\lambda_1,\lambda_2;t,r,\theta,\zeta,p^t,p^r,p^\theta)\mapsto (t+\lambda_1,r,(\theta+\lambda_2)\!\!\!\!\mod 2\pi,\zeta,p^t,p^r,p^\theta)\in M^v_3\,,
\]
describes symmetries of $
\widetilde{h}$ and it is also  three-cosymplectic Lie group action, namely, it leaves invariant ${\bm \tau}, {\bm \omega}$, and their polarisation $V$.  
Then, the restriction of $\Phi$ to $\mathbb{R}^3$ reads
\[
\Phi_3:\R^2\times\R^3\ni (\lambda_1,\lambda_2;t,r,\theta)\mapsto(\lambda_1+t,r,(\theta+\lambda_2)\!\!\!\!\mod 2\pi)\in \R^3\,,
\]
and its space of orbits is diffeomorphic to $\R$. Recall that the existence of such a Lie group action was proved in the proof of Theorem \ref{The::RedDynPolyCosym}. In fact, its space of fundamental vector fields is
$$
D=\left\langle \frac{\partial}{\partial t},\frac{\partial}{\partial \theta}\right\rangle
$$
and the one-forms of $\langle \d r,\d t,\d \theta\rangle$ vanishing on it are $\langle \d r\rangle$. Hence, we have
$$
\overline{\bm \tau}=\d r,\qquad \overline{\bm \omega}=\d \zeta\wedge \d p^r.
$$
Note that ${\bf \widehat{X}}^{\widetilde{h}}=({\bf X}_2^{\widetilde{h}},{\bf X}_1^{\widetilde{h}},{\bf X}_3^{\widetilde{h}})$ and the HDW equations for ${\bf\widehat{ X}}^{\widetilde{h}}$ are the same as before (up to a reparametrization of the indexes of the variables in $\mathbb{R}^3$). 
Consider the submanifold $S_\lambda$ in $\R^3\times \bigoplus_{\alpha=1}^3\cT \R$ given by 
$$
p^t=\lambda_t,\quad p^\theta=\lambda_\theta,\qquad \lambda=(\lambda_t,\lambda_\theta)\in \mathbb{R}^2,
$$
for two constants $\lambda_t,\lambda_\theta\in \mathbb{R}$, which projects onto $\R^6$ diffeomorphically.
Note that there exists a function 
$$ k(r,\zeta,p^r) = \frac{1}{2r}\left( \lambda_t^2 - \frac{1}{c^2}(p^r)^2 - \frac{r^2}{c^2}\lambda_\theta^2 \right) - r \zeta f(r)\,, $$
whose pull-back to $\R^3\times \bigoplus_{\alpha=1}^3\cT\R$ matches the value of $\widetilde{h}$ on $S_\lambda$.

Note that $D$, namely the distribution spanned by the fundamental vector fields of $\Phi$, and  $\ker \overline{\bm \tau}\cap \ker \overline{\bm \omega}=\langle \partial/\partial t,\partial/\partial \theta,\partial/\partial p^t,\partial/\partial p^\theta\rangle$ is an  integrable distribution. Moreover, $\widetilde{h}$ is a first-integral of the vector fields of $\ker \overline{\bm \tau}\cap \T\mathbb{R}^k$. Additionally, one has that  $X_2^{\widetilde{h}2}+X_3^{\widetilde{h}3}=\widehat{X}_2^{\widetilde{h}2}+\widehat{X}_3^{\widetilde{h}}=0$. 

Then, one obtains that the space of leaves of $\Phi_3$ is diffeomorphic to $\mathbb{R}\times \cT \mathbb{R}$ and one has a projection
$$
\pi:(t,r,\theta,\zeta,p^t,p^r,p^\theta)\in\mathbb{R}^3\times \bigoplus_{\alpha=1}^3\cT \mathbb{R}\longmapsto (r,\zeta,p^r)\in \mathbb{R} \times   \cT \mathbb{R}.
$$
The Lie brackets of each of the components of $\bfX^h$ with vector fields of $\ker \overline{\bm \tau}\cap \ker \overline{\bm \omega}$ also belong to the kernel of $\T\pi$ and one obtains an induced $1$-cosymplectic structure
$$
(\mathbb{R} \times   \cT \mathbb{R},\d r,\d \zeta\wedge \d p^r).
$$
In fact, this is the canonical cosymplectic structure in the reduced manifold.  Then, the three-vector fields
$$
{X}^{\widetilde{h}}_1=\frac{\partial}{\partial t}+\frac{\lambda_t}{r}\frac{\partial}{\partial \zeta},\quad
{X}^{\widetilde{h}}_2=\frac{\partial}{\partial r}-\frac{p^r}{c^2r}\frac{\partial}{\partial \zeta}+rf(r)\frac{\partial}{\partial p^r},\quad
{X}^{\widetilde{h}}_3=\frac{\partial}{\partial \theta}-\frac{r\lambda_\theta}{c^2}\frac{\partial}{\partial \zeta}
$$
project  onto the quotient. The final HDW equations are
\begin{equation}
    \parder{p^r}{r} = rf(r)\,,\qquad \parder{\zeta}{r} = -\frac{p^r}{rc^2}\,.
\end{equation}
Combining these equations, the equation of a forced vibrating membrane in polar coordinates reads
$$ c^2\left( \parder{^2\zeta}{r^2} + \frac{1}{r}\parder{\zeta}{r}\right) = -f(r)\,. $$
These are the HDW equations obtained by assuming $p^\theta$ and $p^t$ to be constant in the initial HDW equations.
\section*{Conclusions and outlook}

We have found that $k$-polycosymplectic manifolds are equivalent to $k$-polysymplectic manifolds of a particular type: the here defined and analysed $k$-polysymplectic fibred manifolds. This relation is very interesting for studying geometric properties and developing techniques to study polycosymplectic manifolds via polysymplectic geometry as it shows that polycosymplectic geometry is a particular case of polysymplectic geometry.  Despite that, the relation may have limited interest in the study of their associated Hamiltonian systems. In particular, the equivalence relates manifolds of a different dimension, which may turn a problem related to a Hamiltonian system in a $k$-polycosymplectic manifold into a harder one in an associated $k$-polysymplectic fibred one of larger dimension. Notwithstanding, more research is necessary to determine how far this equivalence can be effectively used to study dynamical systems (see \cite{LMZ22} for similar investigations concerning symplectic and cosymplectic Hamiltonian systems and geometries). 

As a most relevant result, we have developed a $k$-polycosymplectic Marsden--Weinstein reduction procedure by proving that it can be understood as a $k$-polysymplectic Marsden--Weinstein reduction. This seems to close a problem that has been open for more than a decade now. Moreover, the $k$-polysymplectic reduction has been improved in several manners, e.g. by skipping the standard Ad-invariance of the momentum maps, and studied in particular cases, e.g. for $k$-polysymplectic fibred manifolds. Our $k$-polycosymplectic reduction has also been applied to physical problems. Although our work provides a $k$-polycosymplectic reduction by means of an improved $k$-polysymplectic reduction, it is still interesting to perform a $k$-polysymplectic reduction process without relying on a $k$-polysymplectic structure. Moreover, there exist other methods to extend the initial $k$-polycosymplectic manifold to develop a reduction via $k$-symplectic manifolds. 

An interesting approach to $k$-polycosymplectic Marsden--Weinstein reduction using other extension to polysymplectic manifolds was published a month after the first published version of our work in arXiv (see \cite{GM23}). Our work and \cite{GM23} are rather complementary, having different advantages. The results in \cite{GM23} are less general than ours in the sense that, for instance, they do not include our reductions in terms the space-time variables and we consider more general momentum-maps. Notwithstanding, \cite{GM23} noticed that the conditions in \cite{MRSV10} can be simplified, which is a significant advance, and it can further be included in our approach. The extension in \cite{GM23} to polysymplectic manifolds seems to have some nice properties too. We hope to merge the best ideas of our work and 
\cite{GM23} in a future paper.


Finally, we have developed a $k$-cosymplectic to $\ell$-cosymplectic manifold reduction as well as a related reduction for the HDW equations induced by a certain Hamiltonian function of the initial $k$-cosymplectic manifold. Our theory has been illustrated by a vibrating membrane. It is worth noting that, as far as we know, this is the first geometric reduction theory of $k$-cosymplectic manifolds that is able to reduce the base manifold of a field theory. In forthcoming works, we aim to look for new applications of our theory. Our results have many potential applications, but it demands many technical, but practical conditions. We aim to develop a more elegant and less technical description of our $k$-cosymplectic to $\ell$-cosymplectic reduction for the HDW equations given in Theorem \ref{The::RedDynPolyCosym}. Moreover, Theorem \ref{The::RedDynPolyCosym} can be generalised by assuming different conditions on the projectability of $\widehat{\bf X}^h$ or the components thereof. One of the ideas is to write conditions on a restriction of the initial $k$-vector field $\widehat{\bf X}^h$ or its components to $S_\lambda$ or another submanifold. We expect to study these problems in detail in the future. Moreover, we aim to investigate the reconstruction process, i.e. to analyse how the solutions of the reduced $\ell$-cosymplectic Hamiltonian problem on $M_\ell$ are related to the initial HDW equations on $M_k$.

Our main goal in the near future is to develop the Marsden--Weinstein multisymplectic reduction. In this sense, this paper is a good starting point since there are previous studies \cite{RRSV_2011} that allow us to relate $k$-polysymplectic and $k$-polycosymplectic structures to multisymplectic structures in related manifolds. Therefore, the results of this paper are a key point to approach the multisymplectic reduction.

\addcontentsline{toc}{section}{Acknowledgements}
\section*{Acknowledgements}
J. de Lucas, X. Rivas and B.M. Zawora acknowledge  financial support from the Novee Idee 2B-POB II project PSP: 501-D111-20-2004310 funded by the ``Inicjatywa Doskonałości - Uczelnia Badawcza'' (IDUB) program. X. Rivas acknowledges financial support from the Ministerio de Ciencia, Innovaci\'on y Universidades (Spain) project D2021-125515NB-21. B.M. Zawora would like to thank the UPC for his hospitality during the research stay that gave rise to the final form of this manuscript. S. Vilari\~no acknowledges partial financial support from Ministerio de Ciencia, Innovaci\'on y Universidades (Spain) project PID2021-125515NB-C22 and from DGA project E48\_20R.





\end{document}